\documentclass[APA,LATO1COL]{WileyNJD-v2}
\usepackage{moreverb}

\newcommand\BibTeX{{\rmfamily B\kern-.05em \textsc{i\kern-.025em b}\kern-.08em
T\kern-.1667em\lower.7ex\hbox{E}\kern-.125emX}}

\articletype{Original Article}%

\received{25 July, 2021}
\revised{<day> <Month>, <year>}
\accepted{<day> <Month>, <year>}

\usepackage{amssymb,amsthm,mathtools,lmodern}

\usepackage{mathrsfs} % to get \mathscr font
\usepackage{dsfont} % to get \mathds for indicator functions
\usepackage{caption,subcaption}
\newcommand{\N}{\mathbb{N}}
\newcommand{\R}{\mathbb{R}}
\newcommand{\QQ}{\mathbb{Q}}
\newcommand{\PP}{\mathbb{P}}
\newcommand{\EE}{\mathbb{E}}
\newcommand{\VV}{\mathbb{V}\mathrm{ar}}

\newcommand{\bb}[1]{\boldsymbol{#1}}

\newcommand{\OO}{\mathcal O}
\newcommand{\oo}{\mathrm{o}}
\newcommand{\leqdef}{\vcentcolon=}

\newcommand{\rd}{{\rm d}}
\newcommand{\ind}{\mathds{1}}
\newcommand{\e}{\varepsilon}
\newcommand{\XX}{X}

\allowdisplaybreaks

\begin{document}

\title{A multivariate normal approximation for the Dirichlet density and some applications}

\author[1,2]{Fr\'ed\'eric Ouimet*}
\address[1]{\orgdiv{Division of Physics, Mathematics and Astronomy}, \orgname{California Institute of Technology}, \orgaddress{\state{California}, \country{USA}}}
\address[2]{\orgdiv{Department of Mathematics and Statistics}, \orgname{McGill University}, \orgaddress{\state{Quebec}, \country{Canada}}}

\corres{*Fr\'ed\'eric Ouimet\\ \email{frederic.ouimet2@mcgill.ca}}

\presentaddress{\orgdiv{Department of Mathematics and Statistics}, \orgname{McGill University}, \orgaddress{\state{Quebec}, \country{Canada}}}

\abstract[Abstract]{In this short note, we prove an asymptotic expansion for the ratio of the Dirichlet density to the multivariate normal density with the same mean and covariance matrix. The expansion is then used to derive an upper bound on the total variation between the corresponding probability measures and rederive the asymptotic variance of the Dirichlet kernel estimators introduced by \cite{doi:10.2307/2347365} and studied theoretically in \cite{arXiv:2002.06956}. Another potential application related to the asymptotic equivalence between the Gaussian variance regression problem and the Gaussian white noise problem is briefly mentioned but left open for future research.}


\jnlcitation{\cname{%
\author{F. Ouimet}} (\cyear{2022}),
\ctitle{A multivariate normal approximation for the Dirichlet density and some applications}, \cjournal{Stat}, \cvol{11 (1), e410}.}

\maketitle

%\footnotetext{\textbf{Abbreviations:} ***}

\vspace{-6mm}
\section{Introduction}\label{sec:intro}

For any $d\in \N$ and $\bb{v}\in \R^d$, let $\|\bb{v}\|_1 \leqdef \sum_{i=1}^d |v_i|$ denote the $\ell^1$ norm and define the $d$-dimensional simplex as
\begin{align}\label{eq:def:simplex}
    \mathcal{S}_d \leqdef \big\{\bb{v}\in [0,1]^d : \|\bb{v}\|_1 \leq 1\big\}.
\end{align}
Given the parameters $N\in \N$ and $(\bb{\alpha},\beta)\in (0,\infty)^{d+1}$, the $\mathrm{Dirichlet}\hspace{0.2mm}(N \bb{\alpha}, N \beta)$ density function is defined by
\vspace{-1mm}
\begin{align}\label{eq:Dirichlet.pdf}
    K_{N\hspace{-0.2mm},\bb{\alpha}\hspace{-0.2mm},\beta}(\bb{x}) = \frac{\Gamma(N \|\bb{\alpha}\|_1 + N \beta)}{\Gamma(N \beta) \prod_{i=1}^d \Gamma(N \alpha_i)} \cdot (1 - \|\bb{x}\|_1)^{N \beta - 1} \prod_{i=1}^d x_i^{N \alpha_i - 1}, \quad \bb{x}\in \mathcal{S}_d.
\end{align}
The covariance matrix of the Dirichlet distribution is well-known to be $(N \|\bb{\alpha}\|_1 + N \beta + 1)^{-1} \, \Sigma_{\bb{r}}$, where
\begin{align}\label{eq:Sigma}
    \Sigma_{\bb{r}} \leqdef \text{diag}(\bb{r}) - \bb{r} \bb{r}^{\top} \quad \text{and} \quad \bb{r} \leqdef \frac{\bb{\alpha}}{\|\bb{\alpha}\|_1 + \beta},
\end{align}
see, e.g., \cite[p.39]{MR2830563}.
By adapting Theorem~1 and Equation~(21) in \cite{MR1157720}, we also know that
\begin{align}
    \det(\Sigma_{\bb{r}}) = \prod_{i=1}^{d+1} r_i \quad \text{and} \quad (\Sigma_{\bb{r}}^{-1})_{ij} = \frac{1}{r_i} \ind_{\{i = j\}} + \frac{1}{r_{d+1}}, ~~ i,j\in \{1,2,\dots,d\},
\end{align}
where $r_{d+1} \leqdef 1 - \|\bb{r}\|_1 = \frac{\beta}{\|\bb{\alpha}\|_1 + \beta}$.

The first goal of the paper (Theorem~\ref{thm:p.k.expansion}) is to establish an asymptotic expansion for the ratio of the Dirichlet density \eqref{eq:Dirichlet.pdf} to the multivariate normal density with the same mean and covariances, namely:
\begin{align}\label{eq:phi.M}
    (N \|\bb{\alpha}\|_1 + N \beta + 1)^{d/2} \phi_{\Sigma_{\bb{r}}}(\bb{\delta}_{\bb{x}}), \quad \bb{x}\in \R^d, \qquad \text{where } \phi_{\Sigma_{\bb{r}}}(\bb{y}) \leqdef \frac{\exp(-\frac{1}{2} \bb{y}^{\top} \Sigma_{\bb{r}}^{-1} \bb{y})}{\sqrt{(2\pi)^d \det(\Sigma_{\bb{r}})}},
\end{align}

\vspace{-2mm}
\noindent
and where
\vspace{-1mm}
\begin{align}
    \bb{\delta}_{\bb{x}} \leqdef (\delta_{1,x_1},\delta_{2,x_2},\dots,\delta_{d,x_d}) \qquad \text{and} \qquad \delta_{i,x_i} \leqdef \frac{x_i - r_i}{(N \|\bb{\alpha}\|_1 + N \beta + 1)^{-1/2}}.
\end{align}
%For a general presentation on local limit theorems, see, e.g., \cite{MR1295242}.

\newpage
The second goal of the paper is to apply the asymptotic expansion to derive an upper bound on the total variation between the probability measures on $\R^d$ induced by \eqref{eq:Dirichlet.pdf} and \eqref{eq:phi.M}, and to rederive the asymptotic variance of the Dirichlet kernel estimators in the context of density estimation for compositional data.
These two applications are treated in Section~\ref{sec:total.variation.bound} and Section~\ref{sec:Dirichlet.asymmetric.kernel}, respectively.
There could be many other potential applications, see for example the excellent survey by \cite{MR3007210} on quantile coupling inequalities.

In fact, the original motivation for the present paper was the PhD thesis of Huibin Zhou \citep{Zhou2004phd}, in which a multi-resolution coupling methodology between beta and normal random variables is applied to prove the asymptotic equivalence between the Gaussian variance regression problem and the Gaussian white noise problem under Besov smoothness constraints (see also the related works of \cite{MR1425958}, \cite{MR1611772}, \cite{MR1633574}, \cite{MR1633574}, \cite{MR1922538}, \cite{MR2125610}, \cite{MR2202326}, \cite{MR2351100}, \cite{MR2435461}, \cite{MR2549558}, \cite{MR2589320} and \cite{MR3010397}).
In \cite{Zhou2004phd}, the main idea was that the information we get from the sampled observations $X_i\sim \mathrm{Normal}(0,f(t_i))$, where the $t_i$'s form a fixed partition of $[0,1]$ and $f$ is an unknown density function, can be encoded using the (Gaussian) increments of a properly scaled Brownian motion with drift $t\mapsto \frac{1}{\sqrt{2}} \int_0^t \log f(s) \rd s$, and vice versa. Ultimately, the crucial step in the proof involves multiscale inductive quantile couplings (comparisons) between conditionally scaled chi-squared random variables (i.e., beta random variables) and Gaussian analogues, akin to the multiscale argument in \cite{MR1922539} used to prove the asymptotic equivalence, in \cite{MR2102503}, between the density estimation problem and a similar Gaussian white noise problem, and akin to the dyadic scheme used in the proof of the KMT approximation by many authors (see, e.g., \cite{MR375412,MR402883}, \cite{MR893903}, \cite{MR972783}, \cite{MR996984}, \cite{MR1616527}, \cite{Major_2000_tech_report}, \cite{Dudley_2005_KMT}).
We believe that the main result here (Theorem~\ref{thm:p.k.expansion}) could lead to a significant simplification of the proof of \cite[Theorem~2.1]{Zhou2004phd}, in analogy with the removal of the inductive part of the proof for the Le Cam distance upper bound between multinomial and multivariate normal experiments from \cite[Theorem~1]{MR1922539}, shown in \cite{MR4249129}. This point is left open for future research.

The general reason that we are interested in developing normal approximations for the Dirichlet density, other than for the two applications given in Section~\ref{sec:applications} and the potential simplification of the proof of the asymptotic equivalence mentioned above, is because the (multivariate) normal distribution is at the heart of the asymptotic theory for many statistical methods.
Any problem that would involve the Dirichlet density and/or its moments (assuming large parameters $\bb{\alpha}$ and $\beta$) can be ``transferred'', using Theorem~\ref{thm:p.k.expansion}, to a problem involving the corresponding Gaussian density and/or its moments, which is often easier to deal with.
A typical example of this are quantile coupling inequalities (which are ubiquitous in asymptotic theory), where cumulative distribution functions (integrated densities in the continuous setting) need to be compared. Another example could be the derivation optimal Berry-Esseen type bounds, see, e.g., \cite{MR2743033}, \cite{MR3201658} and \cite{MR3717995}, and references therein.
For a general treatment of normal approximations and further motivation on this subject, we refer the reader to \cite{MR0436272}, \cite{MR1295242} and \cite{MR2732624}.

\begin{remark}
    Throughout the paper, the notation $u = \OO(v)$ means that $\limsup_{N\to \infty} |u / v| < C$, where $C > 0$ is a universal constant.
    Whenever $C$ might depend on some parameters, we add subscripts (for example, $u = \OO_{\bb{\alpha}\hspace{-0.2mm},\beta}(v)$).
    %Similarly, $u = \oo(v)$ means that $\lim_{N\to \infty} |u / v| = 0$, and subscripts indicate which parameters the convergence rate can depend on.
    Also, we write
    \begin{align}\label{eq:def:epsilon.N}
        x_{d+1} \leqdef 1 - \|\bb{x}\|_1, \qquad \alpha_{d+1} \leqdef \beta, \qquad \text{and} \qquad \e_N \leqdef  \frac{1}{N (\|\bb{\alpha}\|_1 + \beta)}.
    \end{align}
    In particular, the definition of $x_{d+1}$ and $r_{d+1}$ implies that $\delta_{d+1,x_{d+1}} = - \sum_{i=1}^d \delta_{i,x_i}$.
\end{remark}

\vspace{-8mm}
\section{Main result}\label{sec:main.result}

    First, we prove an asymptotic expansion for the ratio of the Dirichlet density to the multivariate normal density with the same mean and covariances.

    \begin{theorem}\label{thm:p.k.expansion}
        Pick any $\eta\in (0,1)$, and let
        \begin{align}\label{eq:thm:p.k.expansion.condition}
            B_{\eta} \leqdef \big\{\bb{x}\in \mathcal{S}_d : |\delta_{i,x_i}| \leq \eta N^{1/6}, ~\text{for all } i\in \{1,2,\dots,d+1\}\big\}
        \end{align}
        denote the bulk of the Dirichlet distribution.
        Then, uniformly for $\bb{x}\in B_{\eta}$, we have, as $N\to \infty$,
        \begin{equation}\label{eq:LLT.order.2}
            \begin{aligned}
                \log\left(\frac{K_{N\hspace{-0.2mm},\bb{\alpha}\hspace{-0.2mm},\beta}(\bb{x})}{(1 + \e_N^{-1})^{d/2} \phi_{\Sigma_{\bb{r}}}(\bb{\delta}_{\bb{x}})}\right)
                &= \e_N^{1/2} \cdot \left\{- \sum_{i = 1}^{d+1} \bigg(\frac{\delta_{i,x_i}}{r_i}\bigg) + \frac{1}{3} \sum_{i = 1}^{d+1} \delta_{i,x_i} \bigg(\frac{\delta_{i,x_i}}{r_i}\bigg)^2\right\} \\
                &\quad+ \e_N \cdot \left\{\frac{1}{2} \sum_{i=1}^{d+1} (1 + r_i) \bigg(\frac{\delta_{i,x_i}}{r_i}\bigg)^2 - \frac{1}{4} \sum_{i=1}^{d+1} \delta_{i,x_i} \bigg(\frac{\delta_{i,x_i}}{r_i}\bigg)^3 - \frac{d}{2} + \frac{1}{12} \Big\{1 - \sum_{i=1}^{d+1} r_i^{-1}\Big\}\right\} + \OO_{\bb{\alpha}\hspace{-0.2mm},\beta\hspace{-0.2mm},\eta}\Bigg(\frac{(1 + \|\bb{\delta}_{\bb{x}}\|_1)^5}{N^{3/2}}\Bigg).
            \end{aligned}
        \end{equation}
        Some numerical evidence for the validity of this theorem is shown in Appendix~\ref{sec:simulations}.
    \end{theorem}

    \begin{proof}[Proof of Theorem~\ref{thm:p.k.expansion}]
        Using Stirling's formula,
        \begin{align}
            \log \Gamma(z) = \frac{1}{2} \log(2\pi) + (z - \tfrac{1}{2}) \log z - z + \frac{1}{12z} + \OO(z^{-3}), \quad z\to \infty,
        \end{align}
        see, e.g., \cite[p.257]{MR0167642}, and taking the logarithm in \eqref{eq:Dirichlet.pdf}, we obtain
        \begin{align}\label{eq:log.p.before}
            \log K_{N\hspace{-0.2mm},\bb{\alpha}\hspace{-0.2mm},\beta}(\bb{x})
            &= \log \Gamma(N \|\bb{\alpha}\|_1 + N \beta) - \sum_{i=1}^{d+1} \log \Gamma(N \alpha_i) + \sum_{i=1}^{d+1} (N \alpha_i - 1) \log x_i \notag \\
            &= -\frac{d}{2} \log(2\pi) - \frac{d}{2} \log \e_N - \frac{1}{2} \sum_{i=1}^{d+1} \log r_i + \sum_{i=1}^{d+1} (N \alpha_i - 1) \log\Big(\frac{x_i}{r_i}\Big) + \frac{\big\{1 - \sum_{i=1}^{d+1} r_i^{-1}\big\}}{12 \, N \, (\|\bb{\alpha}\|_1 + \beta)} + \OO_{\bb{\alpha}\hspace{-0.2mm},\beta}(N^{-3}).
        \end{align}
        By writing $\frac{x_i}{r_i} = 1 + \frac{\delta_{i,x_i}}{r_i} (1 + \e_N^{-1})^{-1/2}$ in \eqref{eq:log.p.before}, we deduce
        \begin{equation}\label{eq:log.p}
            \begin{aligned}
                \log K_{N\hspace{-0.2mm},\bb{\alpha}\hspace{-0.2mm},\beta}(\bb{x})
                &= -\log \sqrt{(2\pi)^d \, (1 + \e_N^{-1})^{-d} \, \prod_{i=1}^{d+1} r_i} - \frac{d}{2} \log(1 + \e_N) \\
                &\quad+ \sum_{i=1}^{d+1} (\e_N^{-1} r_i - 1) \log\Big(1 + \frac{\delta_{i,x_i}}{r_i} (1 + \e_N^{-1})^{-1/2}\Big) + \e_N \cdot \frac{1}{12} \Big\{1 - \sum_{i=1}^{d+1} r_i^{-1}\Big\} + \OO_{\bb{\alpha}\hspace{-0.2mm},\beta}(N^{-3}).
            \end{aligned}
        \end{equation}
        By applying the Taylor expansion
        \begin{align}
            \log(1 + y) = y - \frac{y^2}{2} + \frac{y^3}{3} - \frac{y^4}{4} + \OO_{\eta}(y^5), \quad \text{valid for } |y| \leq \eta < 1,
        \end{align}
        and noticing that $\delta_{d+1,x_{d+1}} = - \sum_{i=1}^d \delta_{i,x_i}$, we have
        \begin{equation}
            \begin{aligned}
                \log K_{N\hspace{-0.2mm},\bb{\alpha}\hspace{-0.2mm},\beta}(\bb{x})
                &= -\log \sqrt{(2\pi)^d \, (1 + \e_N^{-1})^{-d} \, \prod_{i=1}^{d+1} r_i} - \frac{d}{2} \left\{\e_N + \OO_{\bb{\alpha}\hspace{-0.2mm},\beta}(N^{-2})\right\} - (1 + \e_N^{-1})^{-1/2} \sum_{i=1}^{d+1} \frac{\delta_{i,x_i}}{r_i} + \frac{(1 + \e_N^{-1})^{-1}}{2} \sum_{i=1}^{d+1} \frac{\delta_{i,x_i}^2}{r_i^2} \\[-1mm]
                &\quad- (1 + \e_N)^{-1} \sum_{i=1}^d \frac{\delta_{i,x_i}^2}{2} \left\{\frac{1}{r_i} - \frac{2}{3} \cdot \frac{\delta_{i,x_i}}{r_i^2} (1 + \e_N^{-1})^{-1/2} + \frac{1}{2} \cdot \frac{\delta_{i,x_i}^2}{r_i^3} (1 + \e_N^{-1})^{-1} + \OO_{\bb{\alpha}\hspace{-0.2mm},\beta\hspace{-0.2mm},\eta}\Bigg(\frac{1 + |\delta_{i,x_i}|^3}{N^{3/2}}\Bigg)\right\} \\
                &\quad- (1 + \e_N)^{-1} \sum_{i,j=1}^d \frac{\delta_{i,x_i} \delta_{j,x_j}}{2} \left\{\frac{1}{r_{d+1}} + \frac{2}{3} \cdot \sum_{\ell=1}^d \frac{\delta_{\ell,x_{\ell}}}{r_{d+1}^2} (1 + \e_N^{-1})^{-1/2} + \frac{1}{2} \cdot \sum_{\ell\hspace{-0.2mm},m=1}^d \frac{\delta_{\ell,x_{\ell}} \delta_{m,x_m}}{r_{d+1}^3} (1 + \e_N^{-1})^{-1} + \OO_{\bb{\alpha}\hspace{-0.2mm},\beta\hspace{-0.2mm},\eta}\Bigg(\frac{(1 + \|\bb{\delta}_{\bb{x}}\|_1)^3}{N^{3/2}}\Bigg)\right\} \\
                &\quad+ \e_N \cdot \frac{1}{12} \Big\{1 - \sum_{i=1}^{d+1} r_i^{-1}\Big\} + \OO_{\bb{\alpha}\hspace{-0.2mm},\beta}(N^{-3}).
            \end{aligned}
        \end{equation}
        We can rewrite this as
        \begin{equation}\label{eq:LLT.order.2.log.before}
            \begin{aligned}
                \log K_{N\hspace{-0.2mm},\bb{\alpha}\hspace{-0.2mm},\beta}(\bb{x})
                &= -\log \sqrt{(2\pi)^d \, (1 + \e_N^{-1})^{-d} \, \prod_{i=1}^{d+1} r_i} - \e_N^{1/2} \sum_{i = 1}^{d+1} \frac{\delta_{i,x_i}}{r_i} + \frac{\e_N}{2} \sum_{i=1}^{d+1} \frac{\delta_{i,x_i}^2}{r_i^2} - (1 + \e_N)^{-1} \sum_{i,j=1}^d \frac{\delta_{i,x_i} \delta_{j,x_j}}{2} \left\{(\Sigma_{\bb{r}}^{-1})_{ij} + S_{N,ij}\right\} \\[-0.5mm]
                &\quad+ \e_N \cdot \Bigg[-\frac{d}{2} + \frac{1}{12} \Big\{1 - \sum_{i=1}^{d+1} r_i^{-1}\Big\}\Bigg] + \OO_{\bb{\alpha}\hspace{-0.2mm},\beta\hspace{-0.2mm},\eta}\Bigg(\frac{(1 + \|\bb{\delta}_{\bb{x}}\|_1)^2}{N^{3/2}}\Bigg).
            \end{aligned}
        \end{equation}
        where the $d \times d$ matrices $\Sigma_{\bb{r}}^{-1}$ and $S_N$ have the $(i,j)$ components:
        \begin{align}\label{eq:def.M.and.S}
            (\Sigma_{\bb{r}}^{-1})_{ij}
            &\leqdef \frac{1}{r_i} \bb{1}_{\{i = j\}} + \frac{1}{r_{d+1}}, \\
                S_{N,ij}
            &\leqdef \frac{2 \e_N^{1/2}}{3} \sum_{\ell=1}^d \frac{\delta_{\ell,x_\ell}}{(1 + \e_N)^{1/2}} \bigg\{\frac{-1}{r_i^2} \bb{1}_{\{i = j = \ell\}} + \frac{1}{r_{d+1}^2}\bigg\} + \frac{\e_N}{2} \sum_{\ell \hspace{-0.1mm},\hspace{-0.1mm} m = 1}^d \frac{\delta_{\ell,x_\ell} \delta_{m,x_m}}{(1 + \e_N)} \bigg\{\frac{1}{r_i^3} \bb{1}_{\{i = j = \ell = m\}} + \frac{1}{r_{d+1}^3}\bigg\} + \OO_{\bb{\alpha}\hspace{-0.2mm},\beta\hspace{-0.2mm},\eta}\Bigg(\frac{(1 + \|\bb{\delta}_{\bb{x}}\|_1)^3}{N^{3/2}}\Bigg).
        \end{align}
        After expanding \eqref{eq:LLT.order.2.log.before} using $(1 + \e_N)^{-1} = 1 - \e_N + \dots$, and rearranging some terms, we get
        \begin{equation}\label{eq:LLT.order.2.log.before}
            \begin{aligned}
                \log\left(\frac{K_{N\hspace{-0.2mm},\bb{\alpha}\hspace{-0.2mm},\beta}(\bb{x})}{(1 + \e_N^{-1})^{d/2} \phi_{\Sigma_{\bb{r}}}(\bb{\delta}_{\bb{x}})}\right)
                &= - \e_N^{1/2} \sum_{i = 1}^{d+1} \frac{\delta_{i,x_i}}{r_i} + \frac{\e_N}{2} \sum_{i=1}^{d+1} \frac{\delta_{i,x_i}^2}{r_i^2} + \e_N \sum_{i,j=1}^d \frac{\delta_{i,x_i} \delta_{j,x_j}}{2} (\Sigma_{\bb{r}}^{-1})_{ij} - \e_N^{1/2} \sum_{i,j,\ell=1}^d \frac{\delta_{i,x_i} \delta_{j,x_j} \delta_{\ell,x_{\ell}}}{3} \bigg\{\frac{-1}{r_i^2} \bb{1}_{\{i = j = \ell\}} + \frac{1}{r_{d+1}^2}\bigg\} \\[-0.5mm]
                &\quad- \e_N \sum_{i,j,\ell,m=1}^d \frac{\delta_{i,x_i} \delta_{j,x_j} \delta_{\ell,x_{\ell}} \delta_{m,x_m}}{4} \bigg\{\frac{1}{r_i^3} \bb{1}_{\{i = j = \ell = m\}} + \frac{1}{r_{d+1}^3}\bigg\} + \e_N \cdot \Bigg[-\frac{d}{2} + \frac{1}{12} \Big\{1 - \sum_{i=1}^{d+1} r_i^{-1}\Big\}\Bigg] \\
                &\quad+ \OO_{\bb{\alpha}\hspace{-0.2mm},\beta\hspace{-0.2mm},\eta}\Bigg(\frac{(1 + \|\bb{\delta}_{\bb{x}}\|_1)^5}{N^{3/2}}\Bigg).
            \end{aligned}
        \end{equation}
        To obtain \eqref{eq:LLT.order.2}, simply rewrite the above using the fact that $\delta_{d+1,x_{d+1}} = - \sum_{i=1}^d \delta_{i,x_i}$.
        This ends the proof.
    \end{proof}

\section{Applications}\label{sec:applications}

    In this section, we present two applications of Theorem~\ref{thm:p.k.expansion}.
    We find an upper bound on the total variation between Dirichlet and multivariate normal distributions (Section~\ref{sec:total.variation.bound}) and we present an alternative proof for the asymptotic variance of Dirichlet kernel estimators found in Theorem~4.2 of \cite{arXiv:2002.06956} (Section~\ref{sec:Dirichlet.asymmetric.kernel}).

    \vspace{-4mm}
    \subsection{Total variation bound between Dirichlet and multivariate normal distributions}\label{sec:total.variation.bound}

        \begin{theorem}\label{thm:total.variation}
            Let $(\bb{\alpha},\beta)\in (0,\infty)^{d+1}$ be given.
            Let $\PP_{\bb{\alpha}\hspace{-0.2mm},\beta}$ be the probability measure on $\R^d$ induced by the $\mathrm{Dirichlet}\hspace{0.2mm}(N \bb{\alpha}, N \beta)$ distribution, and let $\QQ_{\bb{\alpha}\hspace{-0.2mm},\beta}$ be the probability measure on $\R^d$ induced by the $\mathrm{Normal}_d(\bb{r}, (1 + \e_N^{-1})^{-1} \, \Sigma_{\bb{r}})$ distribution, where recall $\Sigma_{\bb{r}} \leqdef \mathrm{diag}(\bb{r}) - \bb{r} \bb{r}^{\top}$.
            Then, we have, as $N\to \infty$,
            \vspace{-2mm}
            \begin{align}
                \|\PP_{\bb{\alpha}\hspace{-0.2mm},\beta} - \QQ_{\bb{\alpha}\hspace{-0.2mm},\beta}\| = \OO\left(\e_N^{1/2} \cdot d \, \sqrt{\frac{\max_{1 \leq i \leq d+1} r_i}{\min_{1 \leq i \leq d+1} r_i}}\right),
            \end{align}
            where $\| \cdot \|$ denotes the total variation norm.
        \end{theorem}

        Given the many relations there exist between the total variation and other probability metrics such as the discrepancy metric, the Prokhorov metric and the Hellinger distance (see, e.g., \cite[p.421]{doi:10.2307/1403865}), many corollaries follow straightforwardly from Theorem~\ref{thm:total.variation}.
        The details are omitted for conciseness.

        \begin{proof}[Proof of Theorem~\ref{thm:total.variation}]
            Let $\bb{X}\sim \PP_{\bb{\alpha}\hspace{-0.2mm},\beta}$.
            By the comparison of the total variation norm with the Hellinger distance on page 726 of \cite{MR1922539}, we already know that
            \vspace{-2mm}
            \begin{align}\label{eq:first.bound.total.variation}
                \|\PP_{\bb{\alpha}\hspace{-0.2mm},\beta} - \QQ_{\bb{\alpha}\hspace{-0.2mm},\beta}\| \leq \sqrt{2 \, \PP\big(\bb{X}\in B_{1/2}^c\big) + \EE\left[\log\Big(\frac{\rd \PP_{\bb{\alpha}\hspace{-0.2mm},\beta}}{\rd \QQ_{\bb{\alpha}\hspace{-0.2mm},\beta}}(\bb{X})\Big) \, \ind_{\{\bb{X}\in B_{1/2}\}}\right]}.
            \end{align}
            Then, by applying a union bound followed by large deviation bounds for the beta distribution (see, e.g., Theorem~2.1 of \cite{MR3718704}), we get, for $N$ large enough,
            \begin{align}\label{eq:concentration.bound}
                \PP\big(\bb{X}\in B_{1/2}^c\big)
                &\leq \sum_{i=1}^{d+1} \PP\Big(|\delta_{i,X_i}| > \frac{1}{2} N^{1/6}\Big) \leq (d + 1) \cdot 2 \, \exp\Big(-\frac{1}{2} N^{1/3}\Big).
            \end{align}
            By Theorem~\ref{thm:p.k.expansion},
            \begin{equation}\label{eq:estimate.I.begin}
                \begin{aligned}
                    \EE\left[\log\bigg(\frac{\rd \PP_{\bb{\alpha}\hspace{-0.2mm},\beta}}{\rd \QQ_{\bb{\alpha}\hspace{-0.2mm},\beta}}(\bb{X})\bigg) \, \ind_{\{\bb{X}\in B_{1/2}\}}\right]
                    &= \e_N^{1/2} \cdot \EE\left[\Bigg\{- \sum_{i = 1}^d \delta_{i,x_i} \bigg(\frac{1}{r_i} - \frac{1}{r_{d+1}}\bigg) + \frac{1}{3} \sum_{i\hspace{-0.1mm},\hspace{0.1mm} j\hspace{-0.1mm},\hspace{0.1mm} \ell = 1}^d \delta_{i,x_i} \delta_{j,x_j} \delta_{\ell,x_{\ell}} \bigg(\frac{1}{r_i^2} \ind_{\{i = j = \ell\}} - \frac{1}{r_{d+1}^2}\bigg)\Bigg\} \, \ind_{\{\bb{X}\in B_{1/2}\}}\right] \\
                    &\quad+ \e_N \cdot \OO\left(\Bigg|\sum_{i=1}^{d+1} \frac{\EE[(X_i - r_i)^2]}{\e_N r_i^2}\Bigg| + \Bigg|\sum_{i=1}^{d+1} \frac{\EE[(X_i - r_i)^4]}{\e_N^2 r_i^3}\Bigg| + d + \sum_{i=1}^{d+1} r_i^{-1}\right) + \OO_{d,\bb{\alpha}\hspace{-0.2mm},\beta}(N^{-3/2}).
                \end{aligned}
            \end{equation}
            By Lemma~\ref{lem:Leblanc.2012.boundary.Lemma.1}, the second to last $\OO(\cdot)$ term above is
            \begin{align}\label{eq:estimate.I.next}
                =\OO\Bigg(\sum_{i=1}^{d+1} r_i^{-1}\Bigg) = \OO\left(\frac{d}{\min_{1 \leq i \leq d+1} r_i}\right) = \OO\left(d^{\hspace{0.2mm}2} \, \frac{\max_{1 \leq i \leq d+1} r_i}{\min_{1 \leq i \leq d+1} r_i}\right).
            \end{align}
            (The last equality follows from $\frac{1}{d + 1} \leq \max_{1 \leq i \leq d + 1} r_i$, which itself is consequence of the fact that $r_i \geq 0$ and $\sum_{i=1}^{d+1} r_i = 1$.)
            By putting \eqref{eq:estimate.I.next} in \eqref{eq:estimate.I.begin} and using Lemma~\ref{lem:Leblanc.2012.boundary.Lemma.1.with.set.A}, we get
            \vspace{-2mm}
            \begin{align}\label{eq:estimate.I}
                \eqref{eq:estimate.I.begin}
                &= \e_N^{1/2} \cdot \Bigg\{\frac{\e_N^2 (1 + \e_N^{-1})^{3/2}}{3} \cdot \sum_{i\hspace{-0.1mm},\hspace{0.1mm} j\hspace{-0.1mm},\hspace{0.1mm} \ell = 1}^d \frac{\Big(
                    \begin{array}{l}
                        4 r_i r_j r_{\ell} - 2 r_i r_{\ell} \ind_{\{i = j\}} - 2 r_j r_{\ell} \ind_{\{i = \ell\}} \\[-1.2mm]
                        - 2 r_i r_j \ind_{\{j = \ell\}} + 2 r_i \ind_{\{i = j = \ell\}}
                    \end{array}
                    \Big)}{(1 + \e_N) (1 + 2 \e_N)} \cdot \Big\{\frac{1}{r_i^2} \bb{1}_{\{i = j = \ell\}} - \frac{1}{r_{d+1}^2}\Big\} + \OO\Bigg(\frac{d^{\hspace{0.2mm}3} \, (\PP(\bb{X}\in B_{1/2}^c))^{1/4}}{(\min_{1 \leq i \leq d+1} r_i)^2}\Bigg)\Bigg\} \notag \\
                &\qquad+ \e_N \cdot \OO\Bigg(d^{\hspace{0.2mm}2} \, \frac{\max_{1 \leq i \leq d+1} r_i}{\min_{1 \leq i \leq d+1} r_i}\Bigg) + \OO_{d,\bb{\alpha}\hspace{-0.2mm},\beta}(N^{-3/2}) \notag \\[1mm]
                &= \OO\Bigg(\e_N^{1/2} \cdot \frac{d^{\hspace{0.2mm}3} \, (\PP(\bb{X}\in B_{1/2}^c))^{1/4}}{(\min_{1 \leq i \leq d+1} r_i)^2}\Bigg) + \OO\Bigg(\e_N \cdot d^{\hspace{0.2mm}2} \, \frac{\max_{1 \leq i \leq d+1} r_i}{\min_{1 \leq i \leq d+1} r_i}\Bigg).
            \end{align}

            \vspace{1mm}
            \noindent
            Now, putting \eqref{eq:concentration.bound} and \eqref{eq:estimate.I} together in \eqref{eq:first.bound.total.variation} gives the conclusion.
        \end{proof}

    \subsection{Asymptotic variance of Dirichlet kernel estimators}\label{sec:Dirichlet.asymmetric.kernel}

        Assume that we have a sequence of observations $\bb{X}_1, \bb{X}_2, \dots, \bb{X}_n$ that are independent and $F$ distributed ($F$ is unknown), with density $f$ supported on the $d$-dimensional simplex $\mathcal{S}_d$.
        Then, for a given bandwidth parameter $b > 0$, let
        \begin{align}\label{eq:Dirichlet.estimator}
            \hat{f}_{n,b}(\bb{s}) \leqdef \frac{1}{n} \sum_{i=1}^n K_{1/b,\bb{s} + b, 1 - \|\bb{s}\|_1 + b}(\bb{X}_i), \quad \bb{s}\in \mathcal{S}_d,
        \end{align}
        be the {\it Dirichlet kernel estimator} for the density function $f$.
        This estimator was introduced by \cite{doi:10.2307/2347365} as a nonparametric method of density estimation for compositional data and its asymptotic properties were studied theoretically for the first time in \cite{arXiv:2002.06956}.
        For a detailed overview of the literature on asymmetric kernel estimators, we refer the reader to \cite{MR3821525} or Section~2 in \cite{arXiv:2002.06956}.

        One interesting application of the normal approximation in Theorem~\ref{thm:p.k.expansion} is the derivation of the asymptotic variance of $\hat{f}_{n,b}$ at each point $\bb{s}$ in the interior of the simplex.
        This result was already known from Theorem~4.2 in \cite{arXiv:2002.06956}, but the method of proof we present here is completely different.

        \begin{theorem}\label{thm:asymptotic.variance}
            Assume that $f$ is Lipschitz continuous and let $\bb{s}\in \mathrm{Int}(\mathcal{S}_d)$, then
            \begin{align}
                \VV(\hat{f}_{n,b}(\bb{s})) = \frac{n^{-1} b^{-d/2} (f(\bb{s}) + \OO_{d,\bb{s}}(b^{1/2}))}{\sqrt{(4\pi)^d \, \prod_{i=1}^{d+1} s_i}}, \quad n\to \infty.
            \end{align}
        \end{theorem}

        From this result, other asymptotic expressions can be derived such as the mean squared error and the mean integrated squared error and we can also optimize the bandwidth parameter $b$ with respect to them, see, e.g., Corollary~4.3 and Theorem~4.4 in \cite{arXiv:2002.06956}.

        \begin{proof}[Proof of Theorem~\ref{thm:asymptotic.variance}]
            Straightforward computations show that
            \begin{align}\label{eq:histogram.estimator.var.asymp}
                \VV(\hat{f}_{n,b}(\bb{s}))
                &= n^{-1} \, \EE\big[K_{1/b, \bb{s} + b, 1 - \|\bb{s}\|_1 + b}(\bb{X})^2\big] - n^{-1} \big(\EE\big[K_{1/b, \bb{s} + b, 1 - \|\bb{s}\|_1 + b}(\bb{X})\big]\big)^2 \notag \\[0.5mm]
                &= n^{-1} \, \EE\big[K_{1/b, \bb{s} + b, 1 - \|\bb{s}\|_1 + b}(\bb{X})^2\big] - \OO(n^{-1}),
            \end{align}
            where
            \vspace{-1mm}
            \begin{align}\label{eq:histogram.estimator.var.asymp.next}
                \EE\big[K_{1/b, \bb{s} + b, 1 - \|\bb{s}\|_1 + b}(\bb{X})^2\big]
                &\stackrel{\eqref{eq:LLT.order.2}}{=} \int_{\mathcal{S}_d} \left(\frac{\exp\big(-\frac{1}{2} \bb{\delta}_{\bb{x}}^{\top} \Sigma_{\bb{r}}^{-1} \, \bb{\delta}_{\bb{x}}\big)}{\sqrt{(2\pi)^d \, (1 + \e_N^{-1})^{-d} \, \prod_{i=1}^{d+1} r_i}}\right)^2 f(\bb{x}) \, \rd \bb{x} + \oo_{d,\bb{s}}(1) \notag \\
                &\stackrel{\phantom{\eqref{eq:LLT.order.2}}}{=} \frac{2^{-d/2} (f(\bb{s}) + \OO_{d,\bb{s}}(b^{1/2}))}{\sqrt{(2\pi)^d \, b^{\hspace{0.2mm}d} \, \prod_{i=1}^{d+1} r_i}} \int_{\mathcal{S}_d} \frac{\exp\big(-\frac{1}{2} \bb{\delta}_{\bb{x}}^{\top} (\frac{1}{2} \Sigma_{\bb{r}})^{-1} \, \bb{\delta}_{\bb{x}}\big)}{\sqrt{(2\pi)^d \, 2^{-d} \, (1 + \e_N^{-1})^{-d} \, \prod_{i=1}^{d+1} r_i}} \rd \bb{x} \, + \, \oo_{d,\bb{s}}(1) \notag \\[0.5mm]
                &\stackrel{\phantom{\eqref{eq:LLT.order.2}}}{=} \frac{b^{-d/2} (f(\bb{s}) + \OO_{d,\bb{s}}(b^{1/2}))}{\sqrt{(4\pi)^d \, \prod_{i=1}^{d+1} r_i}} (1 + \oo_d(1)) + \oo_{d,\bb{s}}(1).
            \end{align}
            Since $r_i = (s_i + b) / (1 + b (d + 1)) = s_i + \oo_{d,\bb{s}}(1)$ for all $i\in \{1,2,\dots,d+1\}$, plugging the estimate \eqref{eq:histogram.estimator.var.asymp.next} in \eqref{eq:histogram.estimator.var.asymp} gives us the conclusion.
        \end{proof}

\appendix

\section{Moments of the Dirichlet distribution}

Below, we compute some of the central moments (up to four) of the Dirichlet distribution.
The lemma is used to estimate the $\asymp \e_N$ errors in \eqref{eq:estimate.I.begin} of the proof of Theorem~\ref{thm:total.variation}, and also as a preliminary result for the proof of Lemma~\ref{lem:Leblanc.2012.boundary.Lemma.1.with.set.A}.

\begin{lemma}\label{lem:Leblanc.2012.boundary.Lemma.1}
    Let $N\in \N$ and $(\bb{\alpha},\beta)\in (0,\infty)^{d+1}$ be given.
    If $\bb{\XX} = (\XX_1,\XX_2,\dots,\XX_d)\sim \mathrm{Dirichlet}\hspace{0.2mm}(N \bb{\alpha}, N \beta)$ according to \eqref{eq:Dirichlet.pdf}, then, for all $i,j,\ell\in \{1,2,\dots,d\}$,
    \vspace{-3mm}
    \begin{align}
        &\EE\big[(\XX_i - r_i)(\XX_j - r_j)\big] = \e_N r_i \cdot \frac{(\ind_{\{i = j\}} - r_j)}{(1 + \e_N)}, \label{eq:thm:central.moments.eq.2} \\
        &\EE\big[(\XX_i - r_i)(\XX_j - r_j)(\XX_{\ell} - r_{\ell})\big] = \e_N^2 \cdot \frac{(4 r_i r_j r_{\ell} - 2 r_i r_{\ell} \ind_{\{i = j\}} - 2 r_j r_{\ell} \ind_{\{i = \ell\}} - 2 r_i r_j \ind_{\{j = \ell\}} + 2 r_i \ind_{\{i = j = \ell\}})}{(1 + \e_N) (1 + 2 \e_N)}, \label{eq:thm:central.moments.eq.3} \\[0.5mm]
        &\EE\big[(\XX_i - r_i)^4\big] = \e_N^2 r_i^2 \cdot 3 (1 - r_i)^2 + \OO_{\bb{\alpha}\hspace{-0.2mm},\beta}(N^{-3}), \label{eq:thm:central.moments.4.0}
    \end{align}
    where recall $\e_N \leqdef 1 / (N \|\bb{\alpha}\|_1 + N \beta)$ and $r_i \leqdef \EE[\XX_i] = \alpha_i / (\|\bb{\alpha}\|_1 + \beta)$ for all $i\in \{1,2,\dots,d\}$.
\end{lemma}

\begin{proof}[Proof of Lemma~\ref{lem:Leblanc.2012.boundary.Lemma.1}]
    Equation~\eqref{eq:thm:central.moments.eq.2} can be found in \cite[p.39]{MR2830563}.
    Since
    \begin{align}\label{eq:non.central.two}
        \EE[\XX_i \, \XX_j]
        &=
        \begin{cases}
            \frac{\Gamma(N \alpha_i + 1) \Gamma(N \alpha_j + 1)}{\Gamma(N \alpha_i) \Gamma(N \alpha_j)} \cdot \frac{\Gamma(N \|\bb{\alpha}\|_1 + N \beta)}{\Gamma(N \|\bb{\alpha}\|_1 + N \beta + 2)}, &\mbox{if } i \neq j, \\[1mm]
            \frac{\Gamma(N \alpha_i + 2)}{\Gamma(N \alpha_i)} \cdot \frac{\Gamma(N \|\bb{\alpha}\|_1 + N \beta)}{\Gamma(N \|\bb{\alpha}\|_1 + N \beta + 2)}, &\mbox{if } i = j,
        \end{cases} \notag \\[1mm]
        &= \frac{\alpha_i (\alpha_j + \ind_{\{i = j\}} N^{-1})}{(\|\bb{\alpha}\|_1 + \beta) (\|\bb{\alpha}\|_1 + \beta + N^{-1})},
    \end{align}
    and
    \begin{align}\label{eq:non.central.three}
        \EE[\XX_i \, \XX_j \, \XX_{\ell}]
        &= \begin{cases}
                \frac{\Gamma(N \alpha_i + 1) \Gamma(N \alpha_j + 1) \Gamma(N \alpha_{\ell} + 1)}{\Gamma(N \alpha_i) \Gamma(N \alpha_j) \Gamma(N \alpha_{\ell})} \cdot \frac{\Gamma(N \|\bb{\alpha}\|_1 + N \beta)}{\Gamma(N \|\bb{\alpha}\|_1 + N \beta + 3)}, &\mbox{if } i \neq j \neq \ell \neq i, \\[1mm]
                \frac{\Gamma(N \alpha_i + 2) \Gamma(N \alpha_{\ell} + 1)}{\Gamma(N \alpha_i) \Gamma(N \alpha_{\ell})} \cdot \frac{\Gamma(N \|\bb{\alpha}\|_1 + N \beta)}{\Gamma(N \|\bb{\alpha}\|_1 + N \beta + 3)}, &\mbox{if } i = j \neq \ell, \\[1mm]
                \frac{\Gamma(N \alpha_i + 2) \Gamma(N \alpha_j + 1)}{\Gamma(N \alpha_i) \Gamma(N \alpha_j)} \cdot \frac{\Gamma(N \|\bb{\alpha}\|_1 + N \beta)}{\Gamma(N \|\bb{\alpha}\|_1 + N \beta + 3)}, &\mbox{if } i = \ell \neq j, \\[1mm]
                \frac{\Gamma(N \alpha_i + 1) \Gamma(N \alpha_j + 2)}{\Gamma(N \alpha_i) \Gamma(N \alpha_j)} \cdot \frac{\Gamma(N \|\bb{\alpha}\|_1 + N \beta)}{\Gamma(N \|\bb{\alpha}\|_1 + N \beta + 3)}, &\mbox{if } j = \ell \neq i, \\[1mm]
                \frac{\Gamma(N \alpha_i + 3)}{\Gamma(N \alpha_i)} \cdot \frac{\Gamma(N \|\bb{\alpha}\|_1 + N \beta)}{\Gamma(N \|\bb{\alpha}\|_1 + N \beta + 3)}, &\mbox{if } i = j = \ell,
            \end{cases} \notag \\[1.5mm]
        &= \frac{(\alpha_i + \ind_{\{i = \ell \neq j\}} N^{-1}) (\alpha_j + \ind_{\{i = j\}} N^{-1}) (\alpha_{\ell} + \ind_{\{j = \ell\}} N^{-1} + \ind_{\{i = j = \ell\}} N^{-1})}{(\|\bb{\alpha}\|_1 + \beta) (\|\bb{\alpha}\|_1 + \beta + N^{-1}) (\|\bb{\alpha}\|_1 + \beta + 2 N^{-1})},
    \end{align}
    we have
    \begin{align}
        \EE\big[(\XX_i - r_i)(\XX_j - r_j)(\XX_{\ell} - r_{\ell})\big]
        &= \EE[\XX_i \, \XX_j \, \XX_{\ell}] - r_{\ell} \, \EE[\XX_i \, \XX_j] - r_j \, \EE[\XX_i \, \XX_{\ell}] - r_i \, \EE[\XX_j \, \XX_{\ell}] + 2 \, r_i \, r_j \, r_{\ell} \notag \\[0.5mm]
        &= \frac{\left\{
            \begin{array}{l}
                (\alpha_i + \ind_{\{i = \ell \neq j\}} N^{-1}) (\alpha_j + \ind_{\{i = j\}} N^{-1}) (\alpha_{\ell} + (\ind_{\{j = \ell\}} + \ind_{\{i = j = \ell\}}) N^{-1}) \cdot (\|\bb{\alpha}\|_1 + \beta)^2 \\
                - \alpha_i \alpha_{\ell} (\alpha_j + \ind_{\{i = j\}} N^{-1}) \cdot (\|\bb{\alpha}\|_1 + \beta) (\|\bb{\alpha}\|_1 + \beta + 2 N^{-1}) \\
                 - \alpha_i \alpha_j (\alpha_{\ell} + \ind_{\{i = \ell\}} N^{-1}) \cdot (\|\bb{\alpha}\|_1 + \beta) (\|\bb{\alpha}\|_1 + \beta + 2 N^{-1}) \\
                 - \alpha_i \alpha_j (\alpha_{\ell} + \ind_{\{j = \ell\}} N^{-1}) \cdot (\|\bb{\alpha}\|_1 + \beta) (\|\bb{\alpha}\|_1 + \beta + 2 N^{-1}) \\
                 + 2 \alpha_i \alpha_j \alpha_{\ell} \cdot (\|\bb{\alpha}\|_1 + \beta + N^{-1}) (\|\bb{\alpha}\|_1 + \beta + 2 N^{-1})
            \end{array}
            \right\}
            }{(\|\bb{\alpha}\|_1 + \beta)^3 (\|\bb{\alpha}\|_1 + \beta + N^{-1}) (\|\bb{\alpha}\|_1 + \beta + 2 N^{-1})} \notag \\[1mm]
        &= N^{-2} \cdot \frac{\left\{
            \begin{array}{l}
                4 \alpha_i \alpha_j \alpha_{\ell} - 2 \alpha_i \alpha_{\ell} \ind_{\{i = j\}} (\|\bb{\alpha}\|_1 + \beta) - 2 \alpha_j \alpha_{\ell} \ind_{\{i = \ell\}} (\|\bb{\alpha}\|_1 + \beta) \\
                - 2 \alpha_i \alpha_j \ind_{\{j = \ell\}} (\|\bb{\alpha}\|_1 + \beta) + 2 \alpha_i \ind_{\{i = j = \ell\}} (\|\bb{\alpha}\|_1 + \beta)^2
            \end{array}
            \right\}}{(\|\bb{\alpha}\|_1 + \beta)^3 (\|\bb{\alpha}\|_1 + \beta + N^{-1}) (\|\bb{\alpha}\|_1 + \beta + 2 N^{-1})} \notag \\[2mm]
        &= N^{-2} \cdot \frac{(4 r_i r_j r_{\ell} - 2 r_i r_{\ell} \ind_{\{i = j\}} - 2 r_j r_{\ell} \ind_{\{i = \ell\}} - 2 r_i r_j \ind_{\{j = \ell\}} + 2 r_i \ind_{\{i = j = \ell\}})}{(\|\bb{\alpha}\|_1 + \beta + N^{-1}) (\|\bb{\alpha}\|_1 + \beta + 2 N^{-1})},
    \end{align}
    which proves \eqref{eq:thm:central.moments.eq.3}.
    Finally, trivial calculations show that
    \begin{align}
        \EE\big[\XX_i^4\big]
        &= \frac{\Gamma(N \alpha_i + 4)}{\Gamma(N \alpha_i)} \cdot \frac{\Gamma(N \|\bb{\alpha}\|_1 + N \beta)}{\Gamma(N \|\bb{\alpha}\|_1 + N \beta + 4)} \notag \\[1mm]
        &= \frac{N \alpha_i (N \alpha_i + 1) (N \alpha_i + 2) (N \alpha_i + 3) \cdot N^3 (\|\bb{\alpha}\|_1 + \beta)^3}{N^7 (\|\bb{\alpha}\|_1 + \beta)^7 \cdot (1 + \OO_{\bb{\alpha}\hspace{-0.2mm},\beta}(N^{-1}))}.
    \end{align}
    We deduce
    \begin{align}\label{eq:calculation.power.4}
        \EE\big[(\XX_i - r_i)^4\big]
        &= \EE[\XX_i^4] - 4 \, r_i \, \EE[\XX_i^3] + 6 \, r_i^2 \, \EE[\XX_i^2] - 3 \, r_i^4 \notag \\[0.5mm]
        &= \frac{\left\{\hspace{-1mm}
            \begin{array}{l}
                N \alpha_i (N \alpha_i + 1) (N \alpha_i + 2) (N \alpha_i + 3) \cdot N^3 (\|\bb{\alpha}\|_1 + \beta)^3 \\
                - 4 N^2 \alpha_i^2 (N \alpha_i + 1) (N \alpha_i + 2) \cdot N^2 (\|\bb{\alpha}\|_1 + \beta)^2 (N \|\bb{\alpha}\|_1 + N \beta + 3) \\
                + 6 N^3 \alpha_i^3 (N \alpha_i + 1) \cdot N (\|\bb{\alpha}\|_1 + \beta) \prod_{\ell=2}^3 (N \|\bb{\alpha}\|_1 + N \beta + \ell) \\
                - 3 N^4 \alpha_i^4 \cdot \prod_{\ell=1}^3 (N \|\bb{\alpha}\|_1 + N \beta + \ell)
            \end{array}
            \hspace{-1mm}\right\}}{N^7 (\|\bb{\alpha}\|_1 + \beta)^7 \cdot (1 + \OO_{\bb{\alpha}\hspace{-0.2mm},\beta}(N^{-1}))} \notag \\[1mm]
        &= N^{-2} \cdot \frac{3 \alpha_i^2 (\|\bb{\alpha}\|_1 - \alpha_i + \beta)^2}{(\|\bb{\alpha}\|_1 + \beta)^6} + \OO_{\bb{\alpha}\hspace{-0.2mm},\beta}(N^{-3}),
    \end{align}
    which proves \eqref{eq:thm:central.moments.4.0}.
    This ends the proof.
\end{proof}

\newpage
We can also estimate the moments of Lemma~\ref{lem:Leblanc.2012.boundary.Lemma.1} on various events.
The lemma below is used to estimate the $\asymp \e_N^{1/2}$ errors in \eqref{eq:estimate.I.begin} of the proof of Theorem~\ref{thm:total.variation}.

\begin{lemma}\label{lem:Leblanc.2012.boundary.Lemma.1.with.set.A}
    Let $(\bb{\alpha},\beta)\in (0,\infty)^{d+1}$ be given, and let $A\in \mathscr{B}(\R^d)$ be a Borel set.
    If $\bb{\XX} = (\XX_1,\XX_2,\dots,\XX_d)\sim \mathrm{Dirichlet}\hspace{0.2mm}(N \bb{\alpha}, N \beta)$ according to \eqref{eq:Dirichlet.pdf}, then, for all $i,j,\ell\in \{1,2,\dots,d\}$ and $N$ large enough,
    \begin{align}
        &\left|\EE\big[(\XX_i - r_i) \ind_{\{\bb{\XX}\in A\}}\big]\right| \leq \e_N^{1/2} \big(\PP(\bb{\XX}\in A^c)\big)^{1/2}, \label{eq:thm:central.moments.eq.1.set.A} \\[1mm]
        &\left|\EE\big[(\XX_i - r_i)(\XX_j - r_j)(\XX_{\ell} - r_{\ell}) \ind_{\{\bb{\XX}\in A\}}\big]  - \e_N^2 \cdot \frac{(4 r_i r_j r_{\ell} - 2 r_i r_{\ell} \ind_{\{i = j\}} - 2 r_j r_{\ell} \ind_{\{i = \ell\}} - 2 r_i r_j \ind_{\{j = \ell\}} + 2 r_i \ind_{\{i = j = \ell\}})}{(1 + \e_N) (1 + 2 \e_N)}\right| \leq \e_N^{3/2} \big(\PP(\bb{\XX}\in A^c)\big)^{1/4}, \label{eq:thm:central.moments.eq.3.set.A}
    \end{align}
    where recall $\e_N \leqdef 1 / (N \|\bb{\alpha}\|_1 + N \beta)$ and $r_i \leqdef \EE[\XX_i] = \alpha_i / (\|\bb{\alpha}\|_1 + \beta)$ for all $i\in \{1,2,\dots,d\}$.
\end{lemma}

\begin{proof}[Proof of Lemma~\ref{lem:Leblanc.2012.boundary.Lemma.1.with.set.A}]
    For the bound in \eqref{eq:thm:central.moments.eq.1.set.A}, note that $\EE[\XX_i - r_i] = 0$.
    By Cauchy-Schwarz and a bound on the second moment of the beta distribution (see, e.g., \eqref{eq:thm:central.moments.eq.2}), we have
    \begin{align}\label{eq:thm:central.moments.eq.1.set.A.proof}
        \left|\EE\big[(\XX_i - r_i) \ind_{\{\bb{\XX}\in A\}}\big]\right|
        = \left|\EE\big[(\XX_i - r_i) \ind_{\{\bb{\XX}\in A^c\}}\big]\right|
        \leq \left(\EE\big[(\XX_i - r_i)^2\big]\right)^{1/2} \big(\PP(\bb{\XX}\in A^c)\big)^{1/2}
        \leq \e_N^{1/2} \big(\PP(\bb{\XX}\in A^c)\big)^{1/2}.
    \end{align}
    For the bound in \eqref{eq:thm:central.moments.eq.3.set.A}, Equation \eqref{eq:thm:central.moments.eq.3}, H\"older's inequality and a bound on the fourth central moment of the beta distribution (see, e.g., \eqref{eq:thm:central.moments.4.0}) yield, for $N$ large enough,
    \begin{align}\label{eq:thm:central.moments.eq.3.set.A.proof}
        &\left|\EE\big[(\XX_i - r_i)(\XX_j - r_j)(\XX_{\ell} - r_{\ell}) \, \ind_{\{\bb{\XX}\in A\}}\big] - \e_N^2 \cdot \frac{(4 r_i r_j r_{\ell} - 2 r_i r_{\ell} \ind_{\{i = j\}} - 2 r_j r_{\ell} \ind_{\{i = \ell\}} - 2 r_i r_j \ind_{\{j = \ell\}} + 2 r_i \ind_{\{i = j = \ell\}})}{(1 + \e_N) (1 + 2 \e_N)}\right| \notag \\
        &\qquad= \Big|\EE\big[(\XX_i - r_i)(\XX_j - r_j)(\XX_{\ell} - r_{\ell}) \, \ind_{\{\bb{\XX}\in A^c\}}\big]\Big| \notag \\
        &\qquad\leq \big(\EE\big[(\XX_i - r_i)^4\big]\big)^{1/4} \big(\EE\big[(\XX_j - r_j)^4\big]\big)^{1/4} \big(\EE\big[(\XX_{\ell} - r_{\ell})^4\big]\big)^{1/4} \big(\PP(\bb{\XX}\in A^c)\big)^{1/4} \notag \\[1mm]
        &\qquad\leq \e_N^{1/2} \e_N^{1/2} \e_N^{1/2} \big(\PP(\bb{\XX}\in A^c)\big)^{1/4}.
    \end{align}
    This ends the proof.
\end{proof}

\vspace{-8mm}
\section{Simulations}\label{sec:simulations}

    \vspace{-1mm}
    In this appendix, we provide some numerical evidence (displayed graphically) for the validity of the expansion in Theorem~\ref{thm:p.k.expansion}.
    We compare three levels of approximation for various choices of $\bb{\alpha}$ and $\beta$.
    For any given $(\bb{\alpha},\beta)\in (0,\infty)^{d+1}$, define
    \begin{align}
        E_0
        &\leqdef \sup_{\bb{x}\in \R^d : \|\bb{x} - \bb{r}\|_{\infty} \leq \e_N^{1/2}} \left|\log\left(\frac{K_{N\hspace{-0.2mm},\bb{\alpha}\hspace{-0.2mm},\beta}(\bb{x})}{(1 + \e_N^{-1})^{d/2} \phi_{\Sigma_{\bb{r}}}(\bb{\delta}_{\bb{x}})}\right)\right|, \label{eq:E.0} \\[0.5mm]
        E_1
        &\leqdef \sup_{\bb{x}\in \R^d : \|\bb{x} - \bb{r}\|_{\infty} \leq \e_N^{1/2}} \left|\log\left(\frac{K_{N\hspace{-0.2mm},\bb{\alpha}\hspace{-0.2mm},\beta}(\bb{x})}{(1 + \e_N^{-1})^{d/2} \phi_{\Sigma_{\bb{r}}}(\bb{\delta}_{\bb{x}})}\right) - \e_N^{1/2} \cdot \left\{- \sum_{i = 1}^{d+1} \bigg(\frac{\delta_{i,x_i}}{r_i}\bigg) + \frac{1}{3} \sum_{i = 1}^{d+1} \delta_{i,x_i} \bigg(\frac{\delta_{i,x_i}}{r_i}\bigg)^2\right\}\right|, \label{eq:E.1} \\[1mm]
        E_2
        &\leqdef \sup_{\bb{x}\in \R^d : \|\bb{x} - \bb{r}\|_{\infty} \leq \e_N^{1/2}} \left|\log\left(\frac{K_{N\hspace{-0.2mm},\bb{\alpha}\hspace{-0.2mm},\beta}(\bb{x})}{(1 + \e_N^{-1})^{d/2} \phi_{\Sigma_{\bb{r}}}(\bb{\delta}_{\bb{x}})}\right) - \e_N^{1/2} \cdot \left\{- \sum_{i = 1}^{d+1} \bigg(\frac{\delta_{i,x_i}}{r_i}\bigg) + \frac{1}{3} \sum_{i = 1}^{d+1} \delta_{i,x_i} \bigg(\frac{\delta_{i,x_i}}{r_i}\bigg)^2\right\}\right. \notag \\
        &\quad\left.\hspace{25mm}- \e_N \cdot \left\{\frac{1}{2} \sum_{i=1}^{d+1} (1 + r_i) \bigg(\frac{\delta_{i,x_i}}{r_i}\bigg)^2 - \frac{1}{4} \sum_{i=1}^{d+1} \delta_{i,x_i} \bigg(\frac{\delta_{i,x_i}}{r_i}\bigg)^3 - \frac{d}{2} + \frac{1}{12} \Big\{1 - \sum_{i=1}^{d+1} r_i^{-1}\Big\}\right\}\right|. \label{eq:E.2}
    \end{align}
    Note that $\|\bb{x} - \bb{r}\|_{\infty} \leq \e_N^{1/2}$ implies $\|\bb{\delta}_{\bb{x}}\|_{\infty} \leq (1 + \e_N)^{1/2} \approx 1$, so we expect from Theorem~\ref{thm:p.k.expansion} that the errors above ($E_0$, $E_1$ and $E_2$) will have the asymptotic behavior
    \vspace{-1mm}
    \begin{align}
        E_i = \OO_{d,\bb{\alpha},\beta}(\e_N^{(1 + i)/2}), \quad \text{for all } i\in \{0,1,2\},
    \end{align}
    or equivalently,
    \begin{align}\label{eq:liminf.exponent.bound}
        \liminf_{N\to \infty} \frac{\log E_i}{\log \e_N} \geq \frac{1 + i}{2}, \quad \text{for all } i\in \{0,1,2\}.
    \end{align}

    \vspace{2mm}
    \noindent
    The property \eqref{eq:liminf.exponent.bound} is illustrated in Figures~\ref{fig:error.exponents.plots.beta.1},~\ref{fig:error.exponents.plots.beta.2}~and~\ref{fig:error.exponents.plots.beta.3} below, for various choices of $\bb{\alpha}$ and $\beta$.
    Similarly, the corresponding the log-log plots of the errors as a function of $N$ are displayed in Figures~\ref{fig:loglog.errors.plots.beta.1},~\ref{fig:loglog.errors.plots.beta.2}~and~\ref{fig:loglog.errors.plots.beta.3}.
    The simulations are limited to $N \leq 10^5$ because numerical errors start to perturb the results near that point, but the evidence remains overwhelming.

    \phantom{vertical spacing}
    \begin{figure}[H]
        \captionsetup[subfigure]{labelformat=empty}
        \captionsetup{width=0.8\linewidth}
        \vspace{-0.5cm}
        \centering
        \begin{subfigure}[b]{0.22\textwidth}
            \centering
            \includegraphics[width=\textwidth, height=0.85\textwidth]{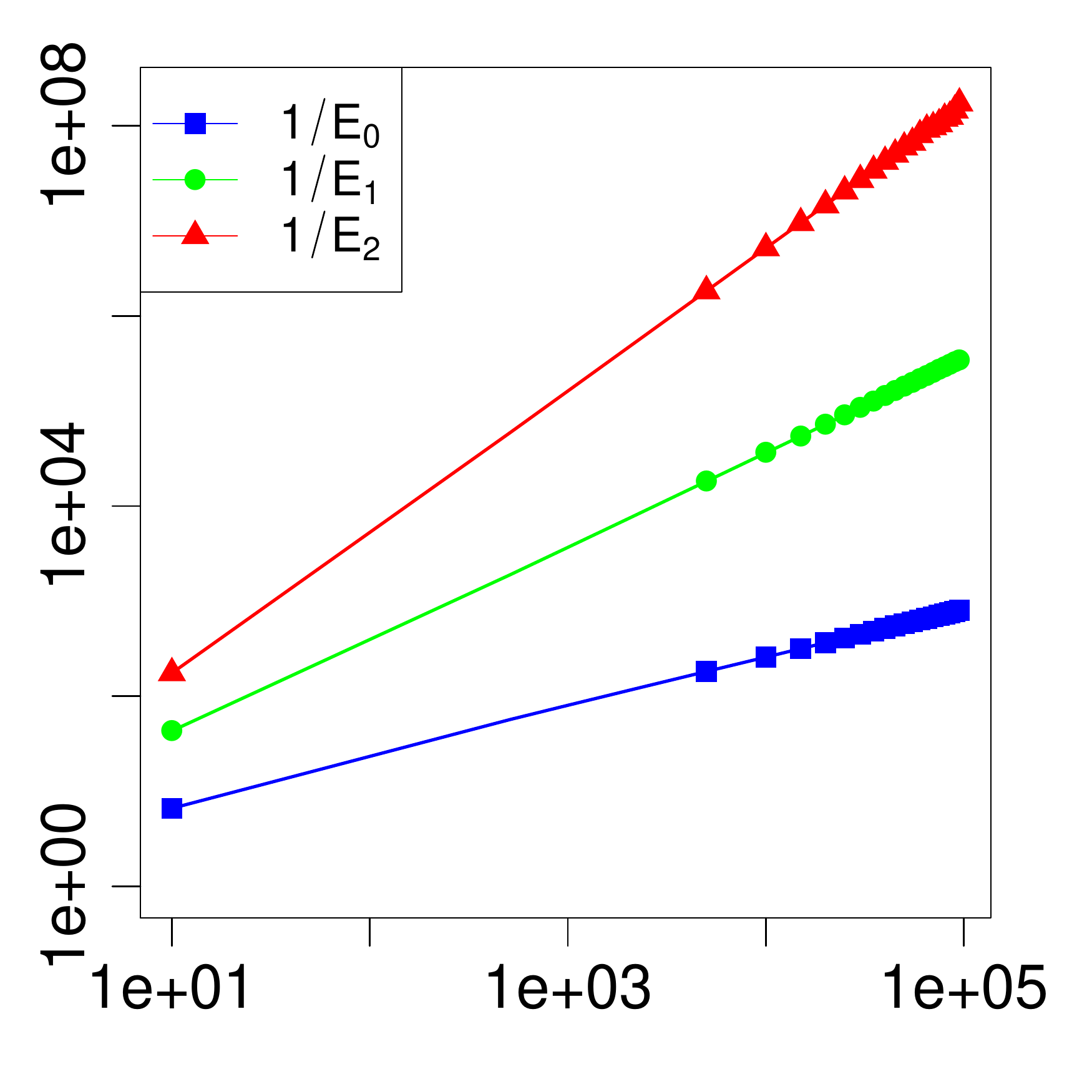}
            \vspace{-0.9cm}
            \caption{$\bb{\alpha} = (1,1)$ and $\beta = 1$}
        \end{subfigure}
        \quad
        \begin{subfigure}[b]{0.22\textwidth}
            \centering
            \includegraphics[width=\textwidth, height=0.85\textwidth]{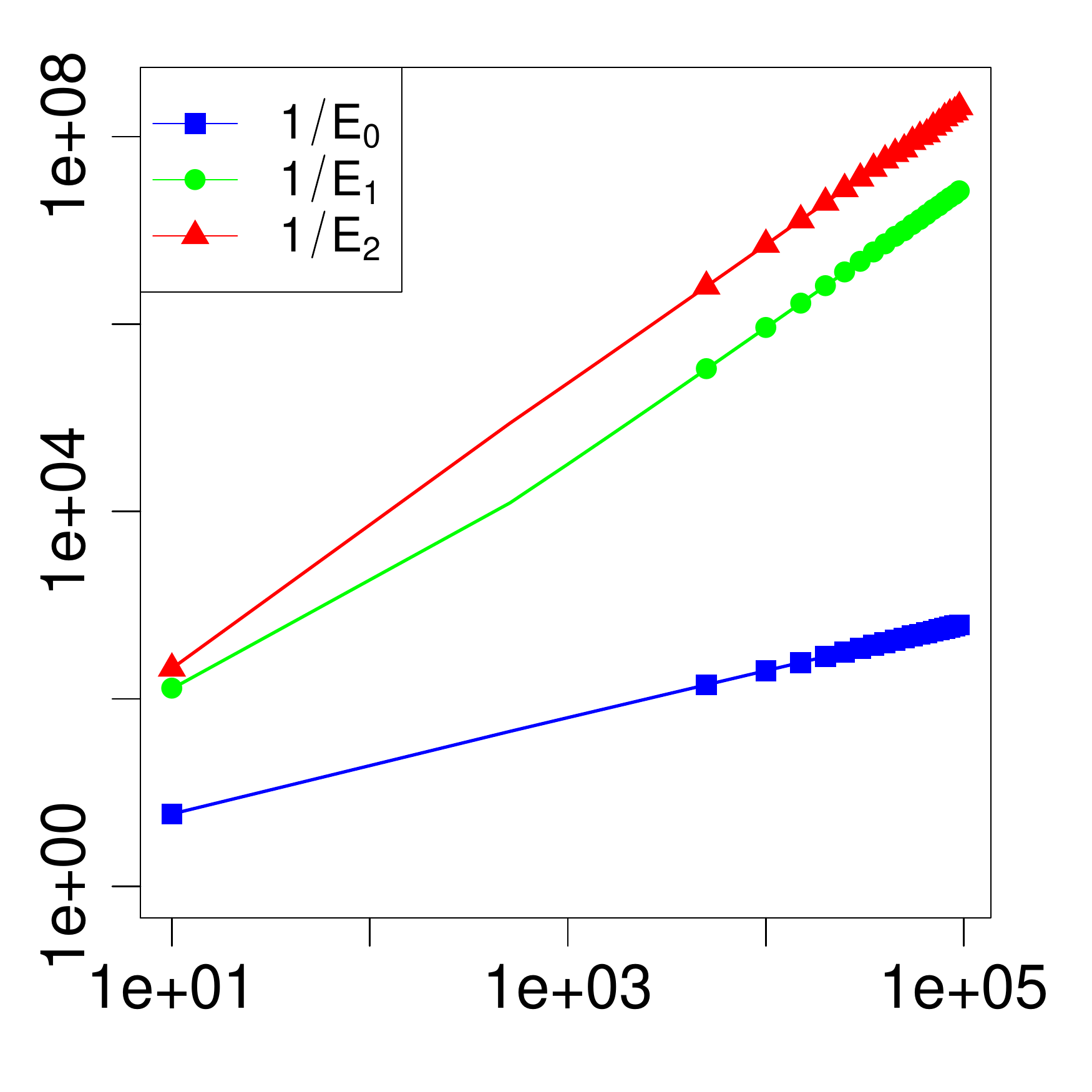}
            \vspace{-0.9cm}
            \caption{$\bb{\alpha} = (1,2)$ and $\beta = 1$}
        \end{subfigure}
        \quad
        \begin{subfigure}[b]{0.22\textwidth}
            \centering
            \includegraphics[width=\textwidth, height=0.85\textwidth]{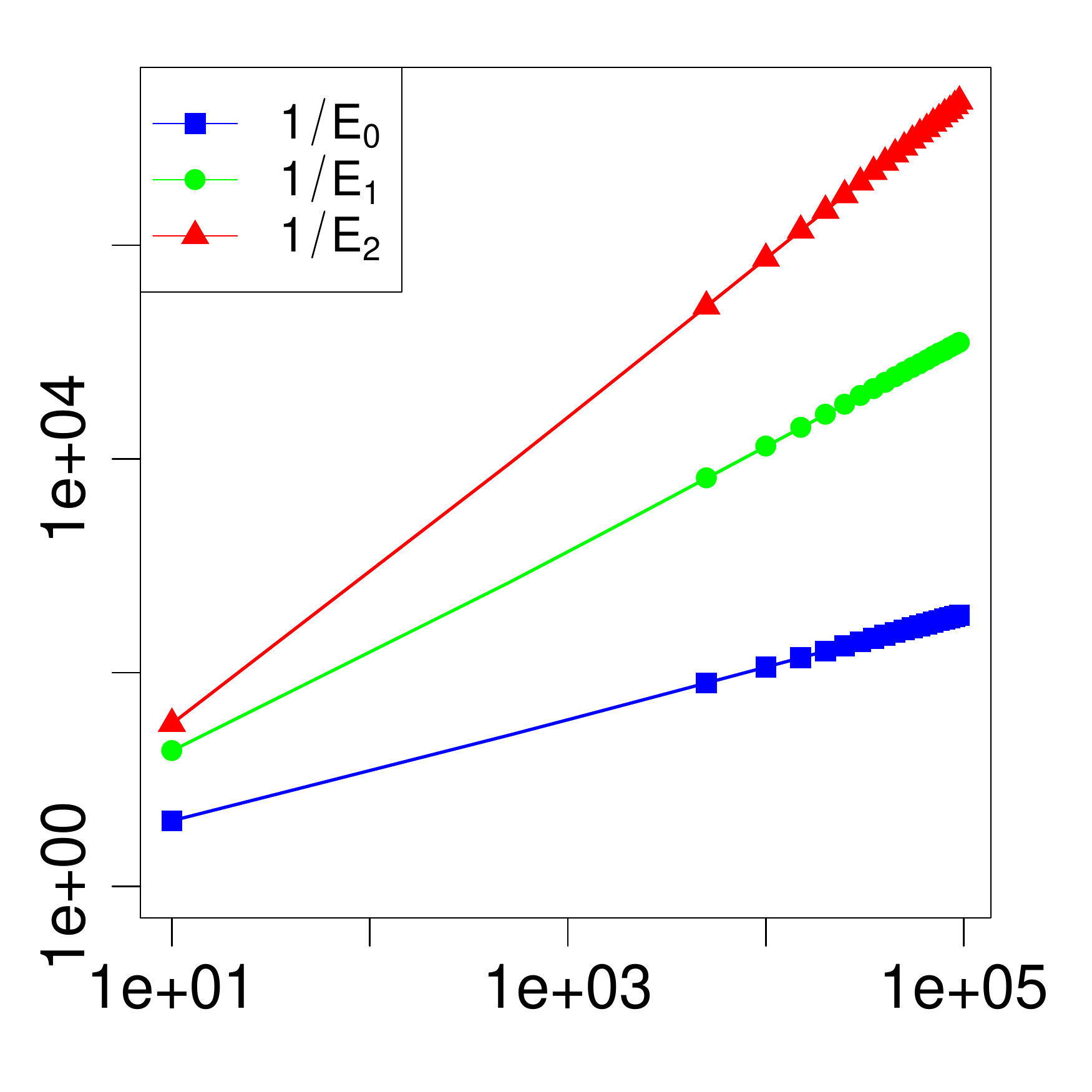}
            \vspace{-0.9cm}
            \caption{$\bb{\alpha} = (1,3)$ and $\beta = 1$}
        \end{subfigure}
        \quad
        \begin{subfigure}[b]{0.22\textwidth}
            \centering
            \includegraphics[width=\textwidth, height=0.85\textwidth]{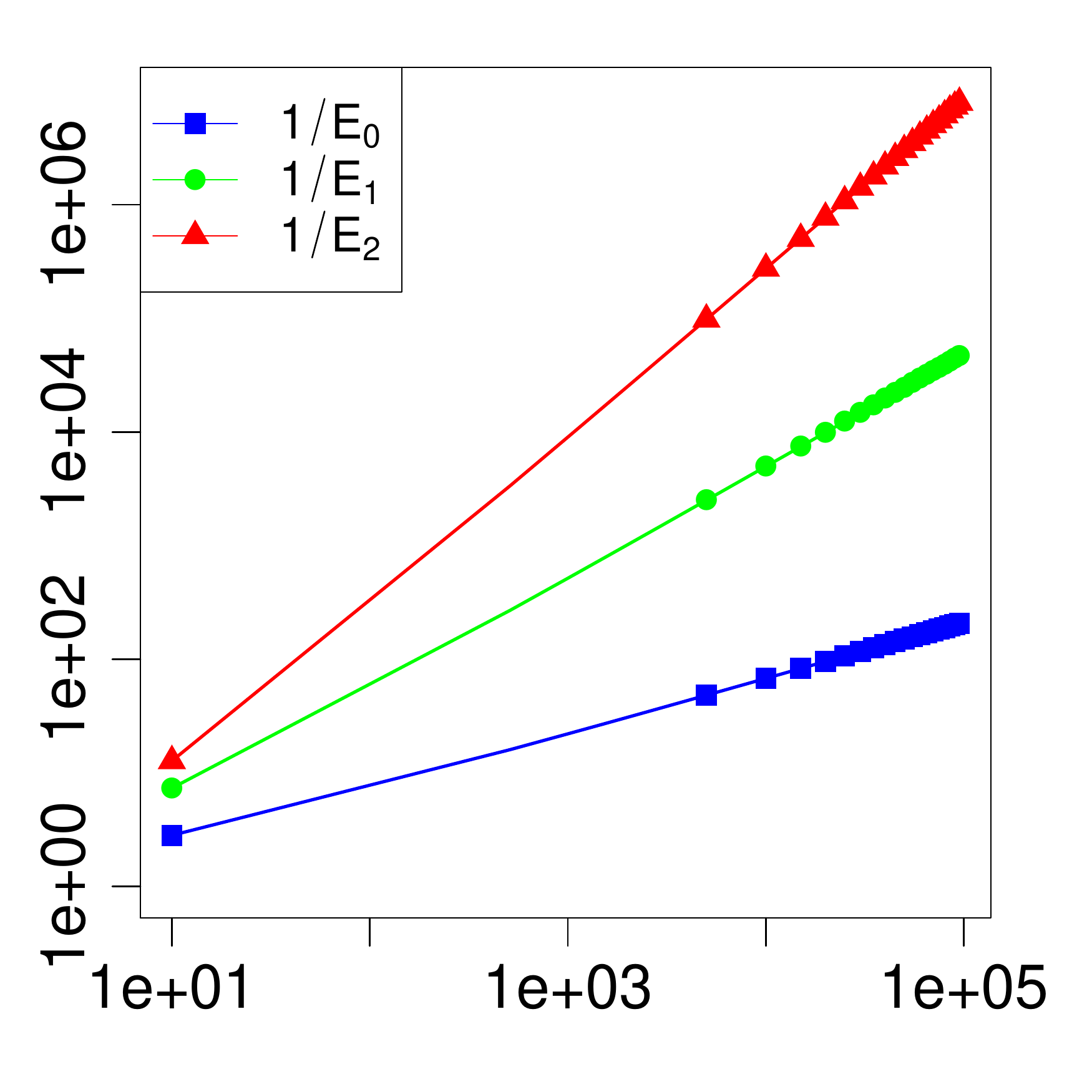}
            \vspace{-0.9cm}
            \caption{$\bb{\alpha} = (1,4)$ and $\beta = 1$}
        \end{subfigure}
        \begin{subfigure}[b]{0.22\textwidth}
            \centering
            \includegraphics[width=\textwidth, height=0.85\textwidth]{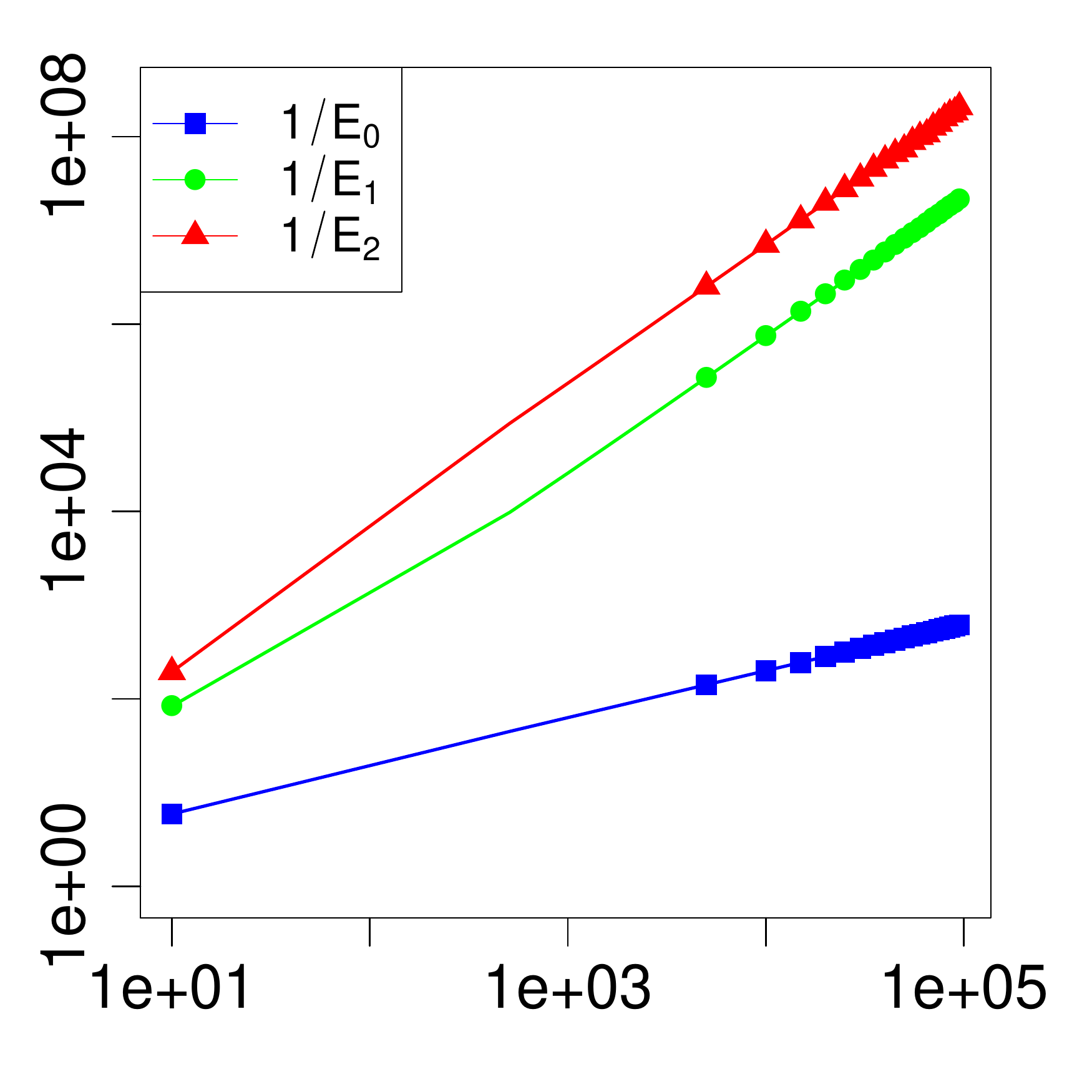}
            \vspace{-0.9cm}
            \caption{$\bb{\alpha} = (2,1)$ and $\beta = 1$}
        \end{subfigure}
        \quad
        \begin{subfigure}[b]{0.22\textwidth}
            \centering
            \includegraphics[width=\textwidth, height=0.85\textwidth]{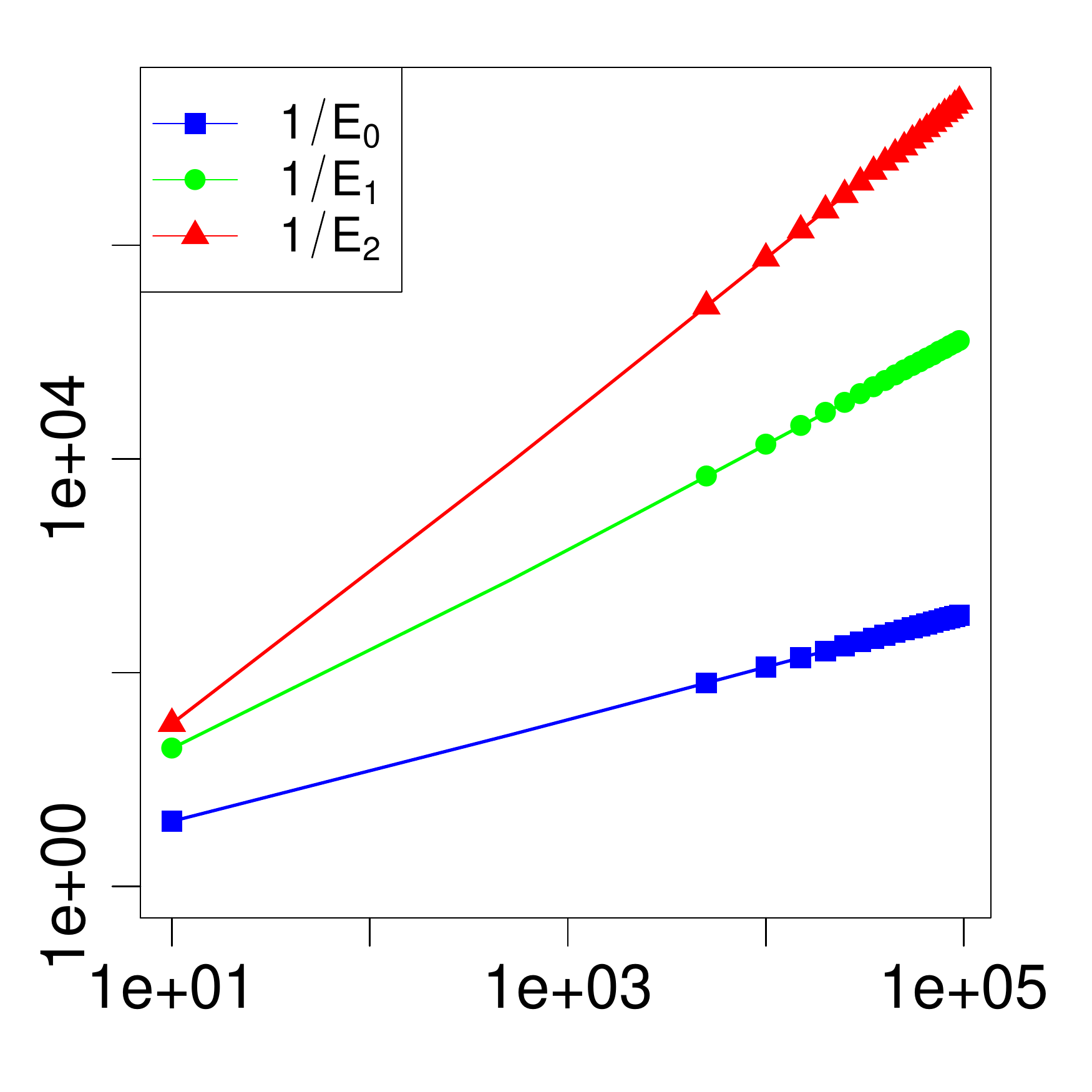}
            \vspace{-0.9cm}
            \caption{$\bb{\alpha} = (2,2)$ and $\beta = 1$}
        \end{subfigure}
        \quad
        \begin{subfigure}[b]{0.22\textwidth}
            \centering
            \includegraphics[width=\textwidth, height=0.85\textwidth]{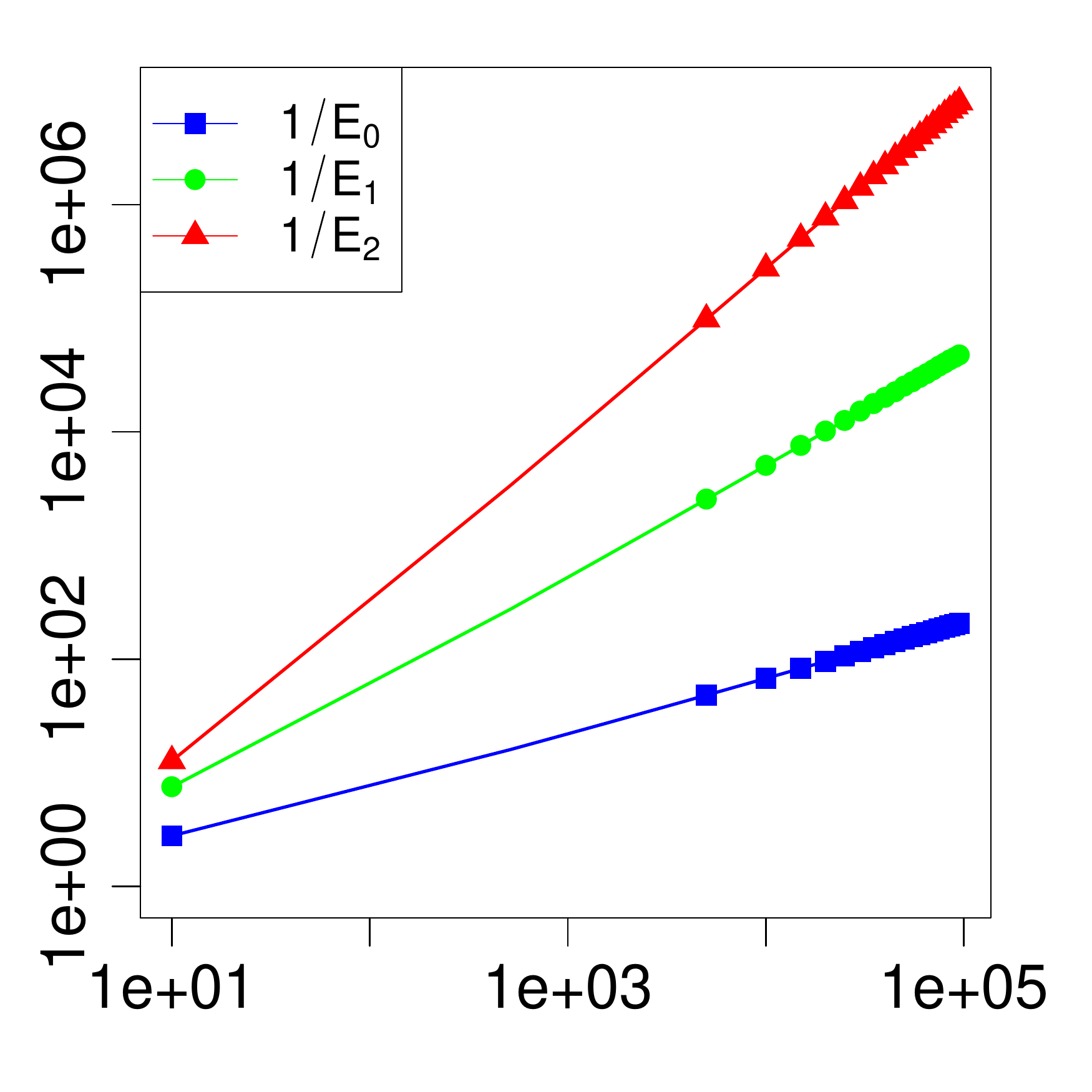}
            \vspace{-0.9cm}
            \caption{$\bb{\alpha} = (2,3)$ and $\beta = 1$}
        \end{subfigure}
        \quad
        \begin{subfigure}[b]{0.22\textwidth}
            \centering
            \includegraphics[width=\textwidth, height=0.85\textwidth]{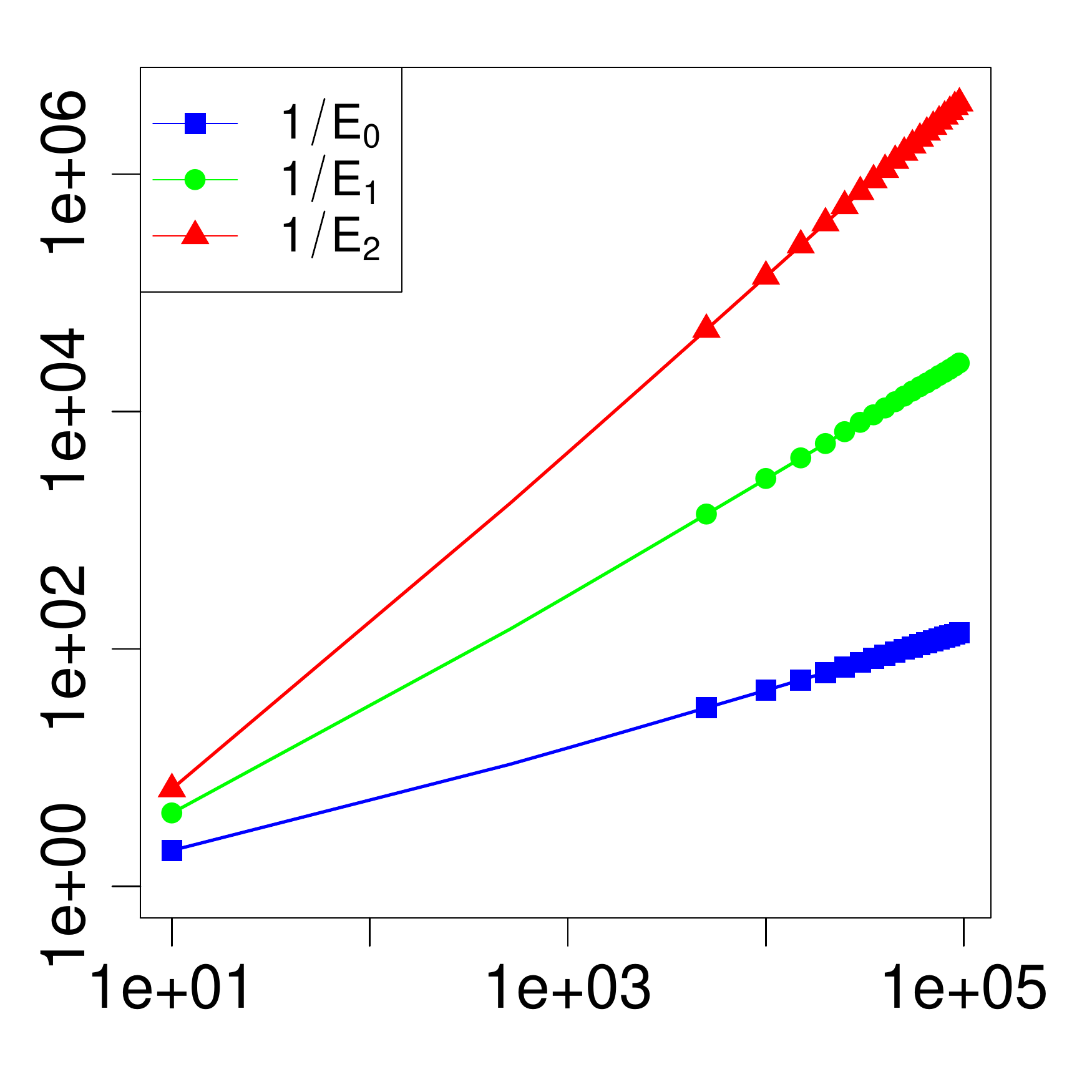}
            \vspace{-0.9cm}
            \caption{$\bb{\alpha} = (2,4)$ and $\beta = 1$}
        \end{subfigure}
        \begin{subfigure}[b]{0.22\textwidth}
            \centering
            \includegraphics[width=\textwidth, height=0.85\textwidth]{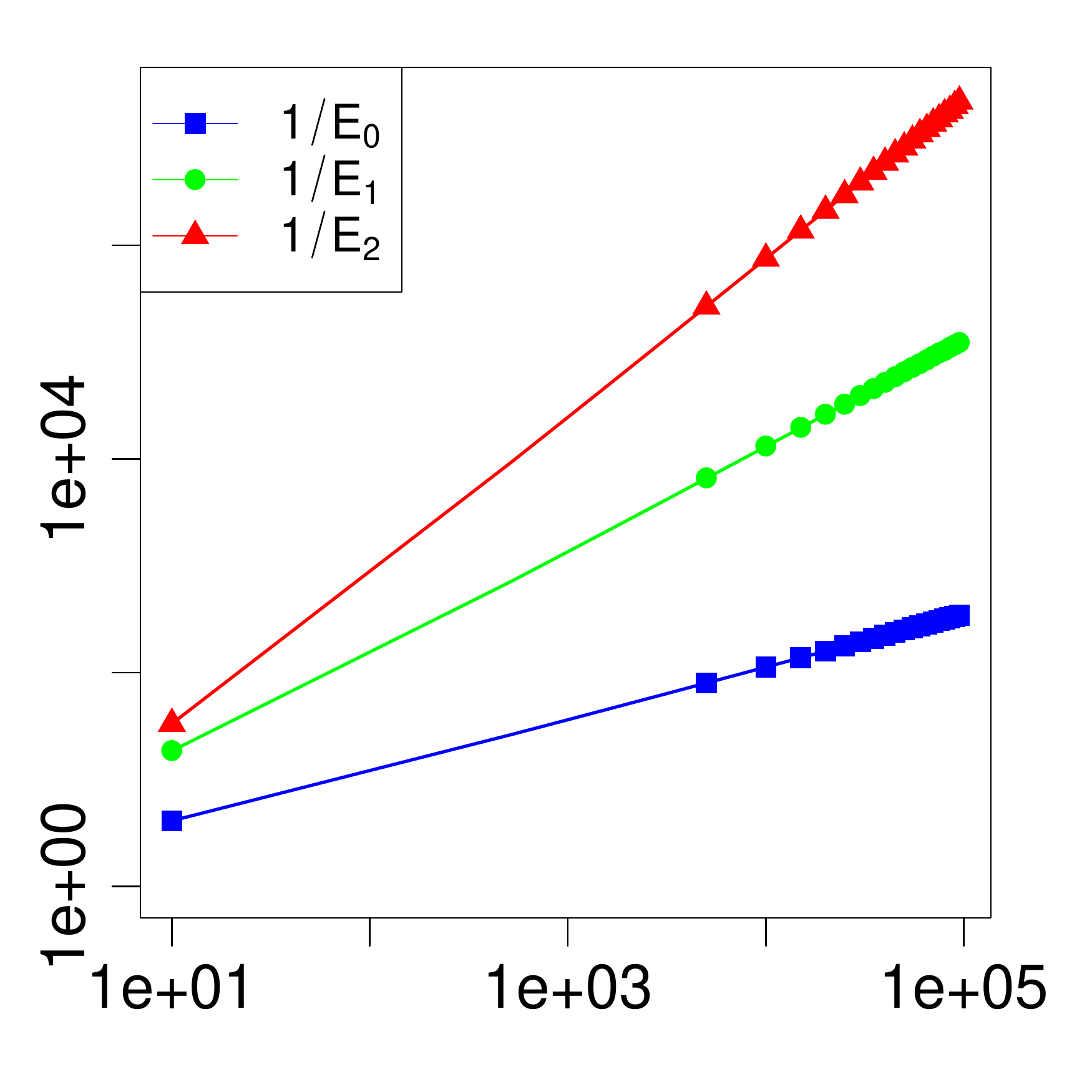}
            \vspace{-0.9cm}
            \caption{$\bb{\alpha} = (3,1)$ and $\beta = 1$}
        \end{subfigure}
        \quad
        \begin{subfigure}[b]{0.22\textwidth}
            \centering
            \includegraphics[width=\textwidth, height=0.85\textwidth]{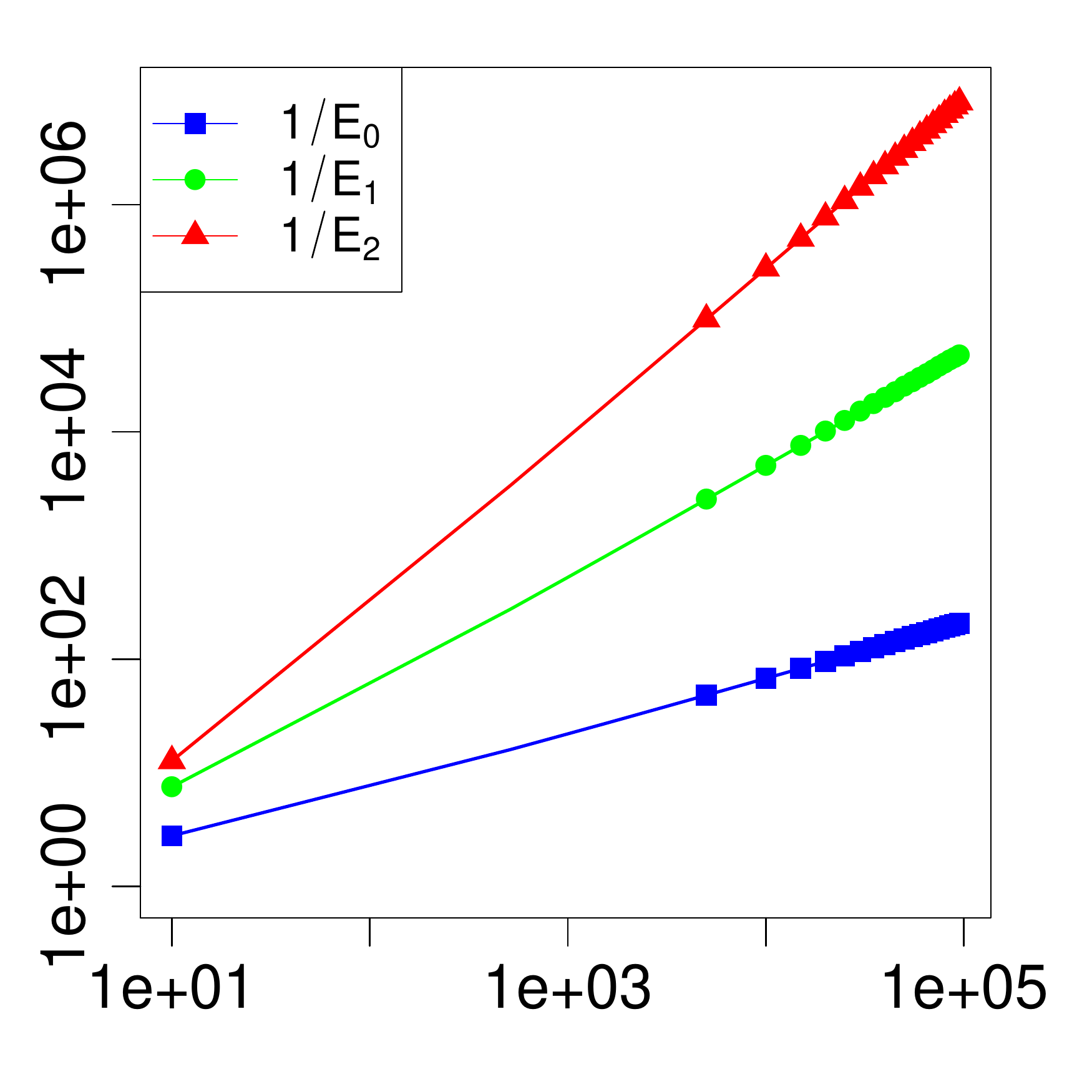}
            \vspace{-0.9cm}
            \caption{$\bb{\alpha} = (3,2)$ and $\beta = 1$}
        \end{subfigure}
        \quad
        \begin{subfigure}[b]{0.22\textwidth}
            \centering
            \includegraphics[width=\textwidth, height=0.85\textwidth]{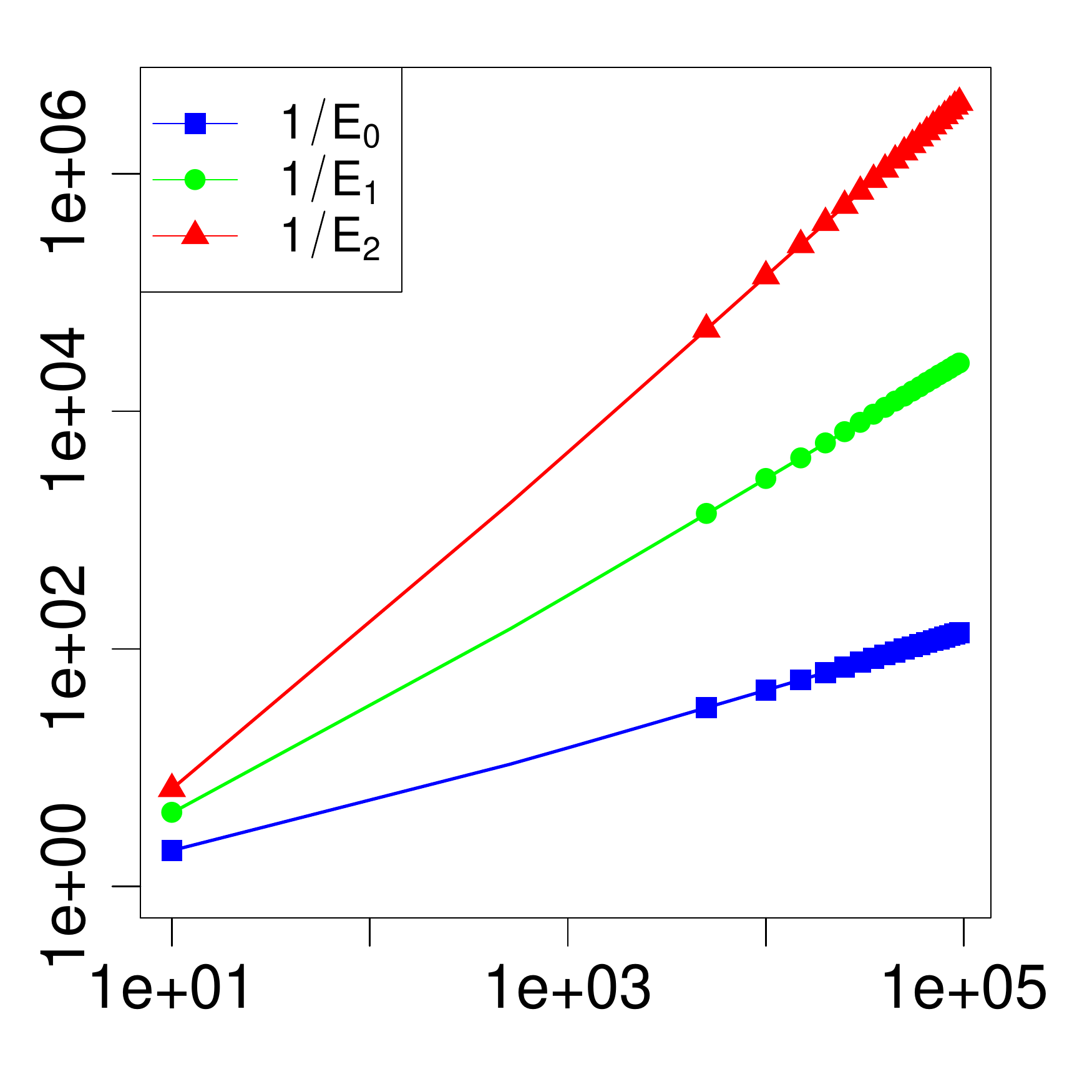}
            \vspace{-0.9cm}
            \caption{$\bb{\alpha} = (3,3)$ and $\beta = 1$}
        \end{subfigure}
        \quad
        \begin{subfigure}[b]{0.22\textwidth}
            \centering
            \includegraphics[width=\textwidth, height=0.85\textwidth]{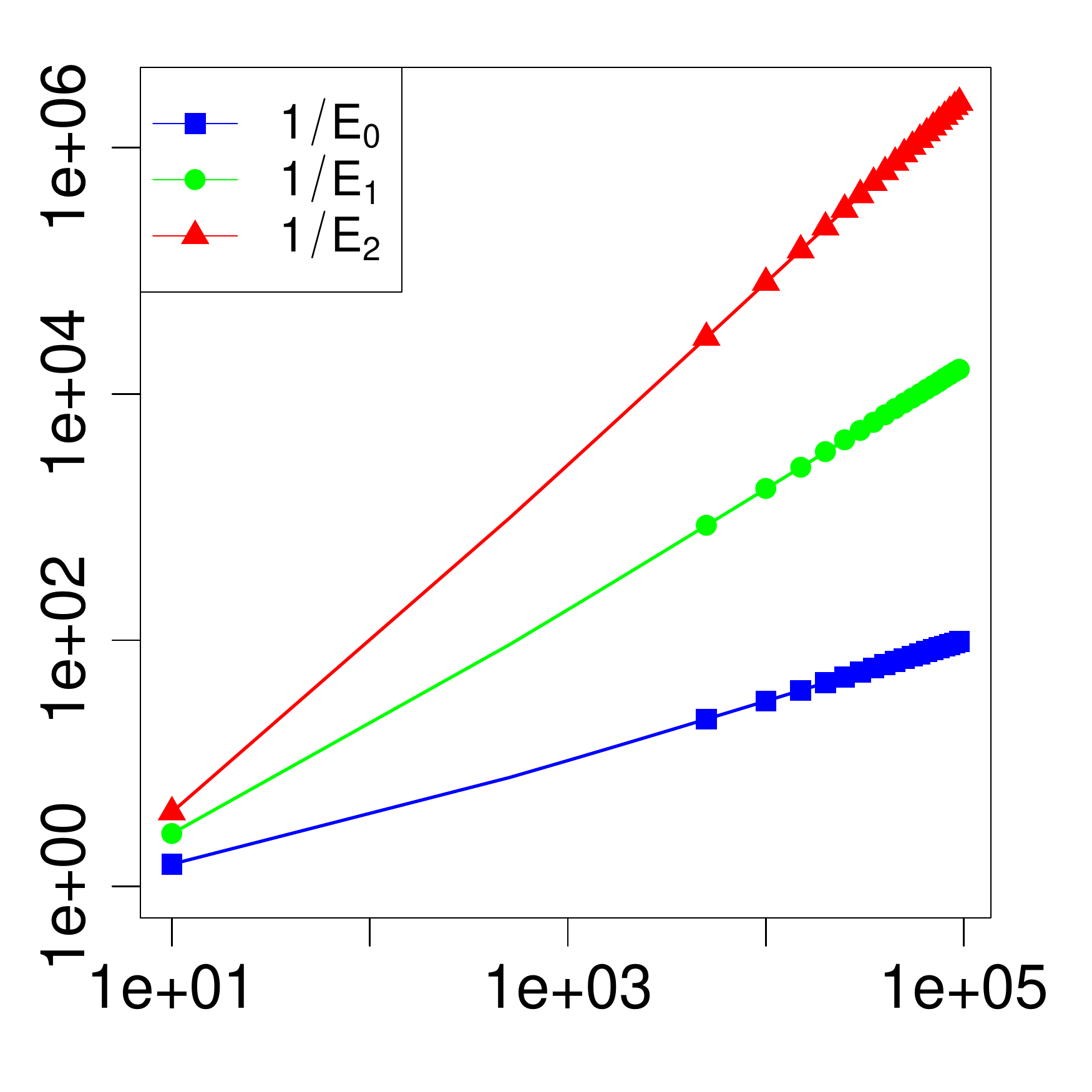}
            \vspace{-0.9cm}
            \caption{$\bb{\alpha} = (3,4)$ and $\beta = 1$}
        \end{subfigure}
        \caption{Plots of $1 / E_i$ as a function of $N$, for various choices of $\bb{\alpha}$, when $\beta = 1$. Both the horizontal and vertical axes are on a logarithmic scale. The plots clearly illustrate how the addition of correction terms from Theorem~\ref{thm:p.k.expansion} to the base approximation \eqref{eq:E.0} improves it.}
        \label{fig:loglog.errors.plots.beta.1}
    \end{figure}
    \begin{figure}[H]
        \captionsetup[subfigure]{labelformat=empty}
        \captionsetup{width=0.8\linewidth}
        \vspace{-0.5cm}
        \centering
        \begin{subfigure}[b]{0.22\textwidth}
            \centering
            \includegraphics[width=\textwidth, height=0.85\textwidth]{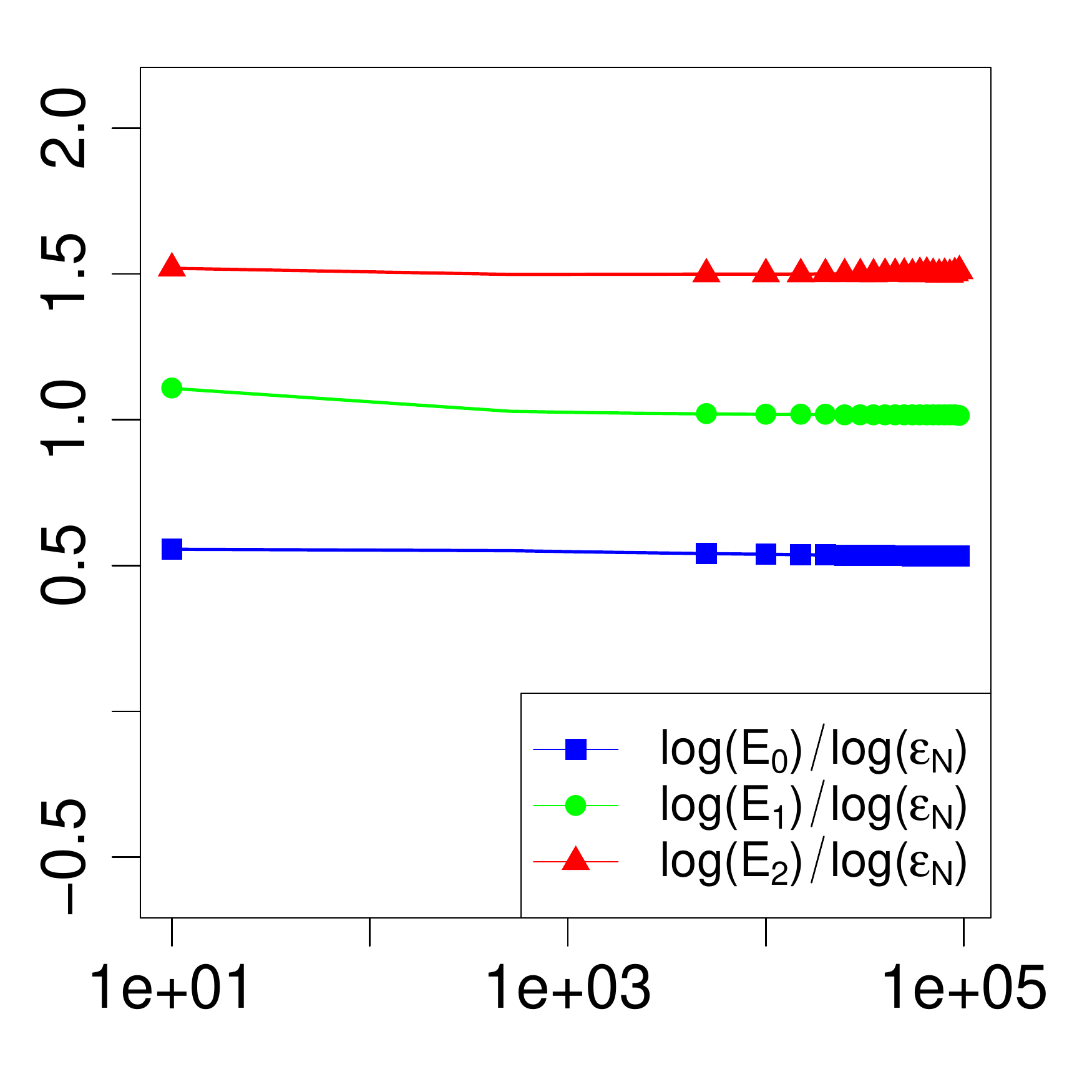}
            \vspace{-0.9cm}
            \caption{$\bb{\alpha} = (1,1)$ and $\beta = 1$}
        \end{subfigure}
        \quad
        \begin{subfigure}[b]{0.22\textwidth}
            \centering
            \includegraphics[width=\textwidth, height=0.85\textwidth]{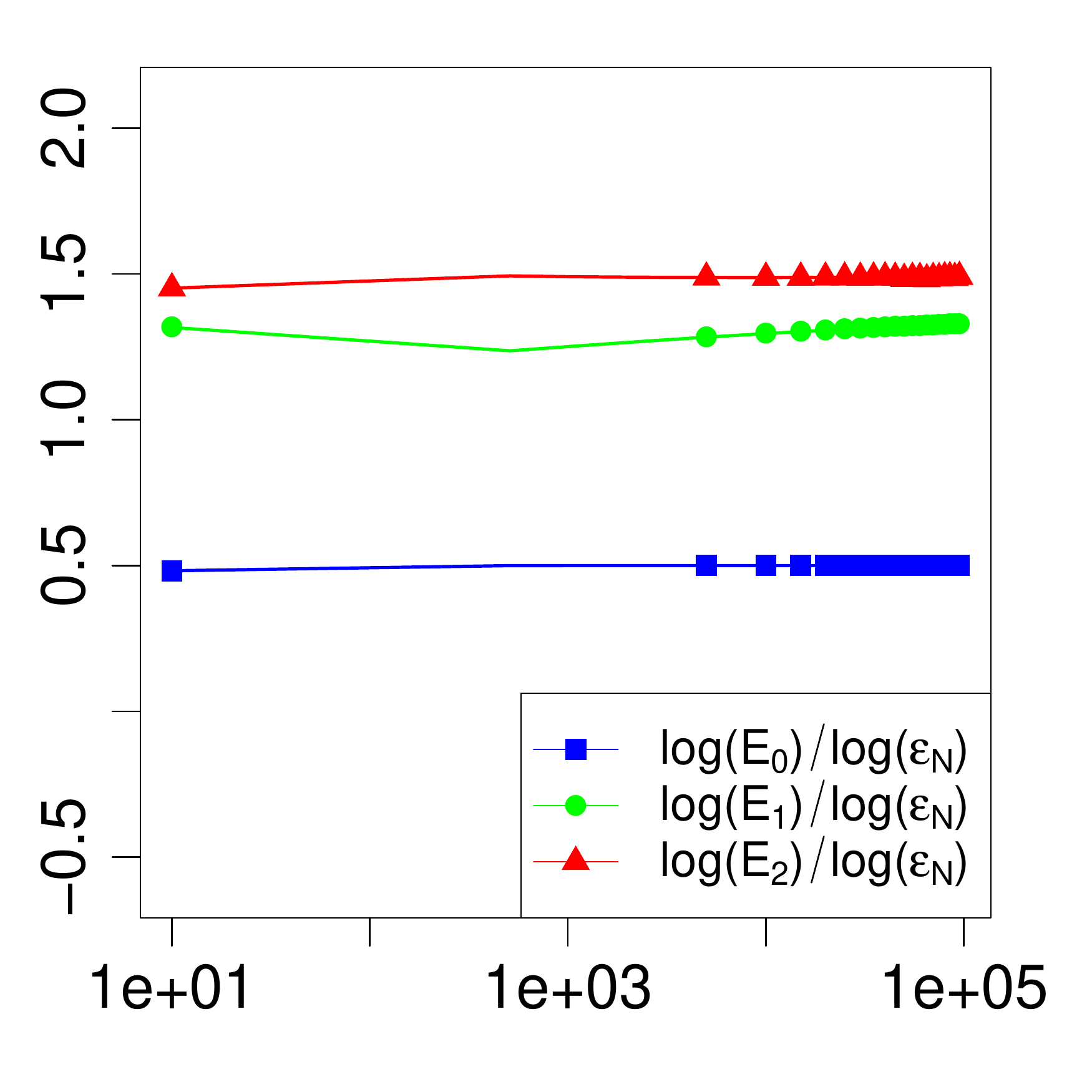}
            \vspace{-0.9cm}
            \caption{$\bb{\alpha} = (1,2)$ and $\beta = 1$}
        \end{subfigure}
        \quad
        \begin{subfigure}[b]{0.22\textwidth}
            \centering
            \includegraphics[width=\textwidth, height=0.85\textwidth]{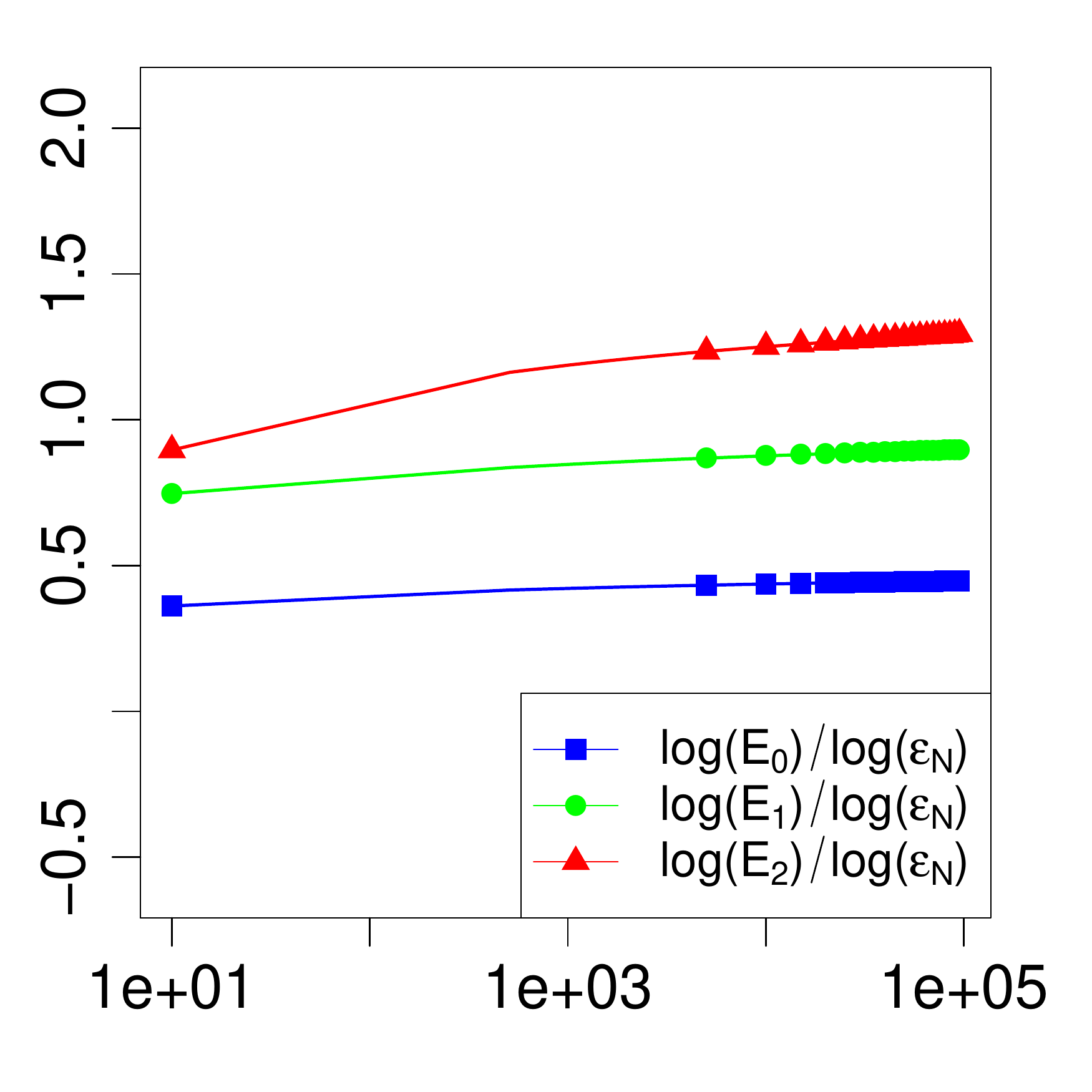}
            \vspace{-0.9cm}
            \caption{$\bb{\alpha} = (1,3)$ and $\beta = 1$}
        \end{subfigure}
        \quad
        \begin{subfigure}[b]{0.22\textwidth}
            \centering
            \includegraphics[width=\textwidth, height=0.85\textwidth]{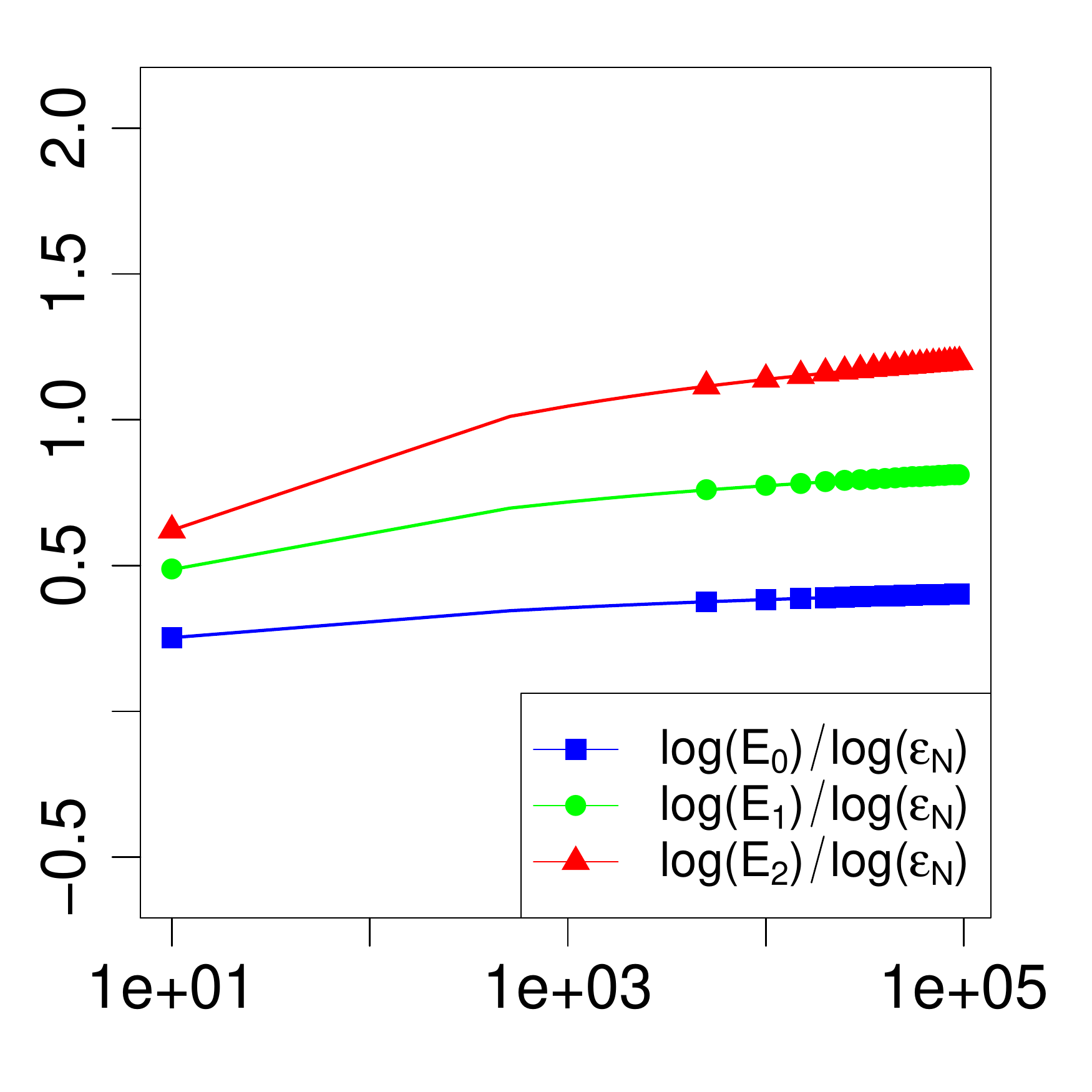}
            \vspace{-0.9cm}
            \caption{$\bb{\alpha} = (1,4)$ and $\beta = 1$}
        \end{subfigure}
        \begin{subfigure}[b]{0.22\textwidth}
            \centering
            \includegraphics[width=\textwidth, height=0.85\textwidth]{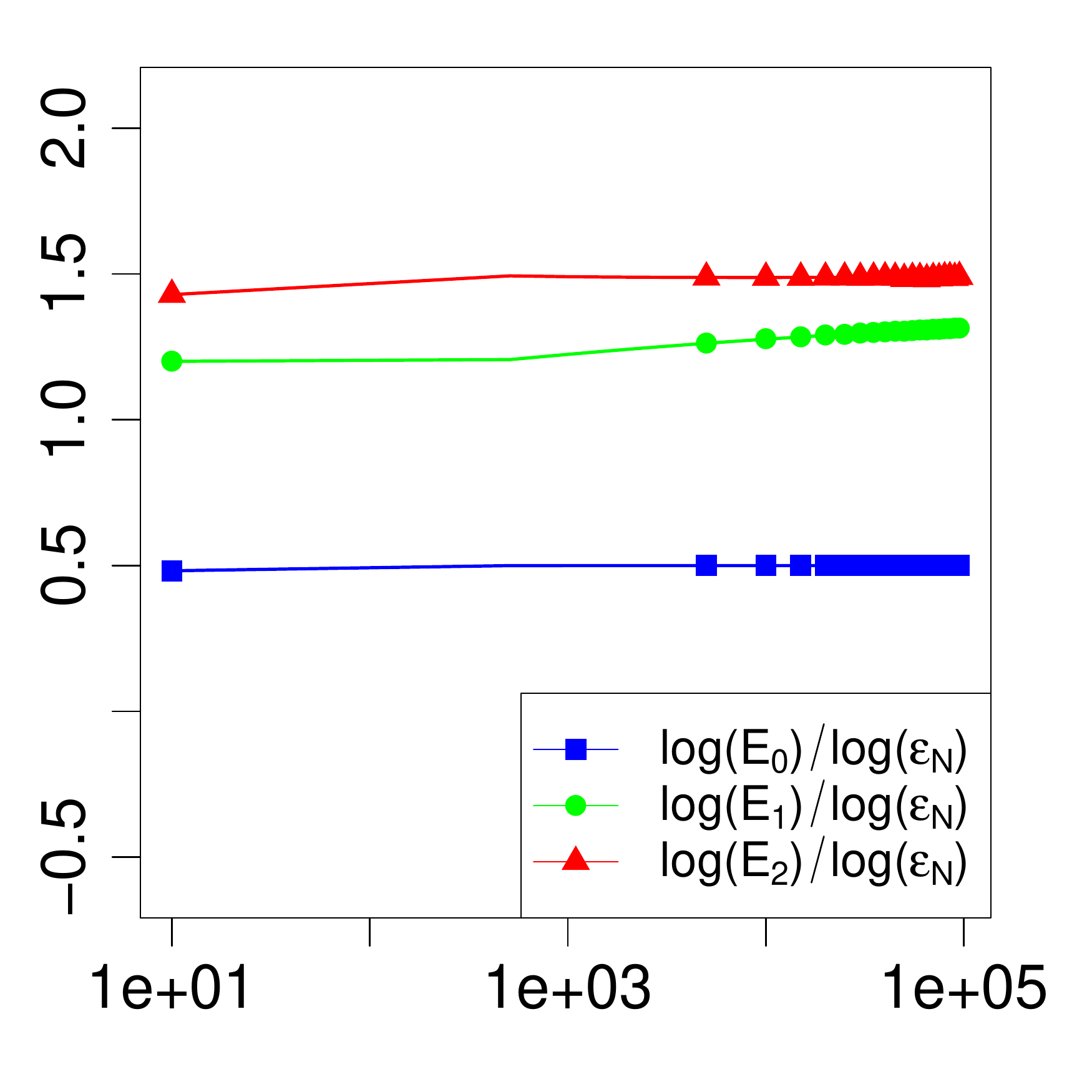}
            \vspace{-0.9cm}
            \caption{$\bb{\alpha} = (2,1)$ and $\beta = 1$}
        \end{subfigure}
        \quad
        \begin{subfigure}[b]{0.22\textwidth}
            \centering
            \includegraphics[width=\textwidth, height=0.85\textwidth]{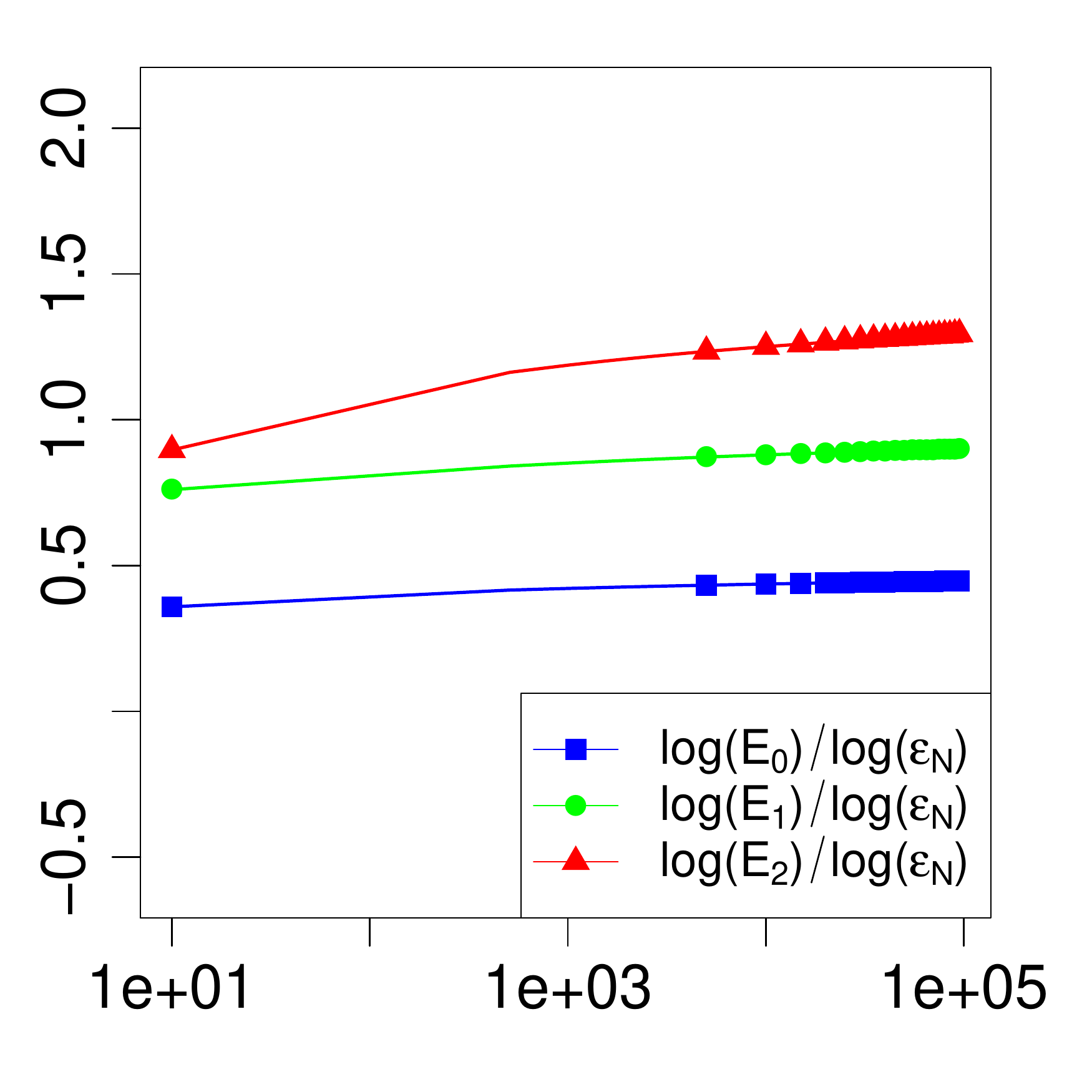}
            \vspace{-0.9cm}
            \caption{$\bb{\alpha} = (2,2)$ and $\beta = 1$}
        \end{subfigure}
        \quad
        \begin{subfigure}[b]{0.22\textwidth}
            \centering
            \includegraphics[width=\textwidth, height=0.85\textwidth]{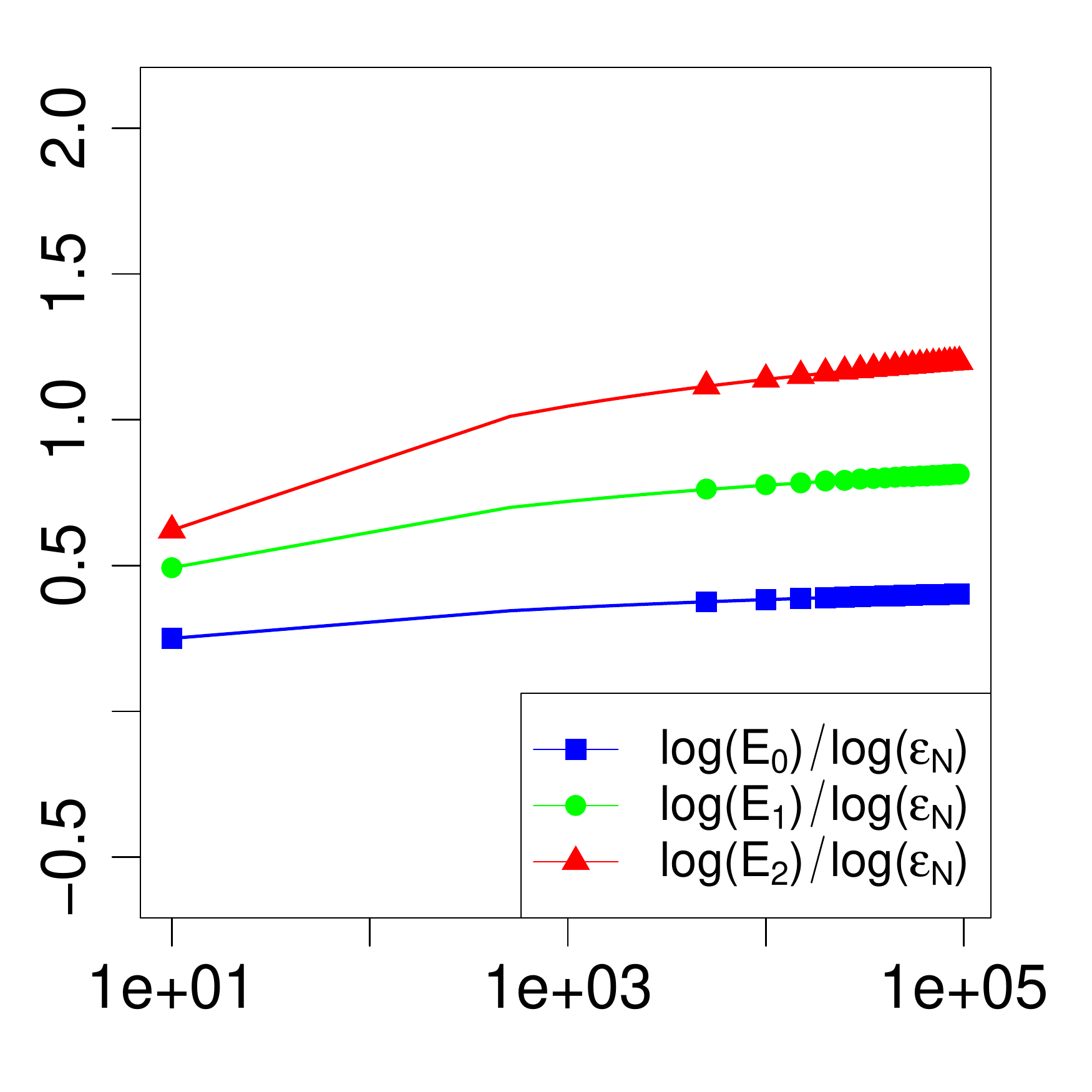}
            \vspace{-0.9cm}
            \caption{$\bb{\alpha} = (2,3)$ and $\beta = 1$}
        \end{subfigure}
        \quad
        \begin{subfigure}[b]{0.22\textwidth}
            \centering
            \includegraphics[width=\textwidth, height=0.85\textwidth]{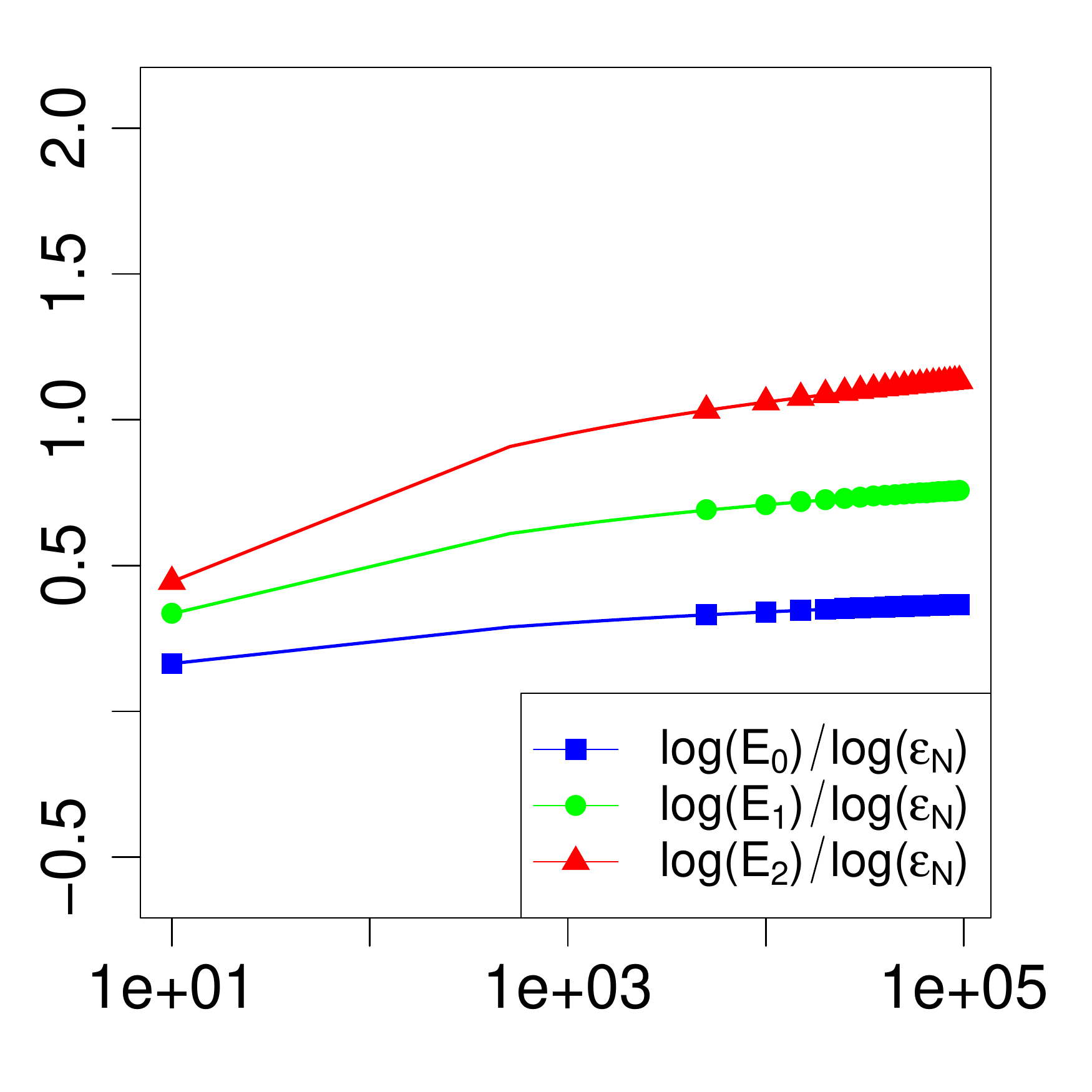}
            \vspace{-0.9cm}
            \caption{$\bb{\alpha} = (2,4)$ and $\beta = 1$}
        \end{subfigure}
        \begin{subfigure}[b]{0.22\textwidth}
            \centering
            \includegraphics[width=\textwidth, height=0.85\textwidth]{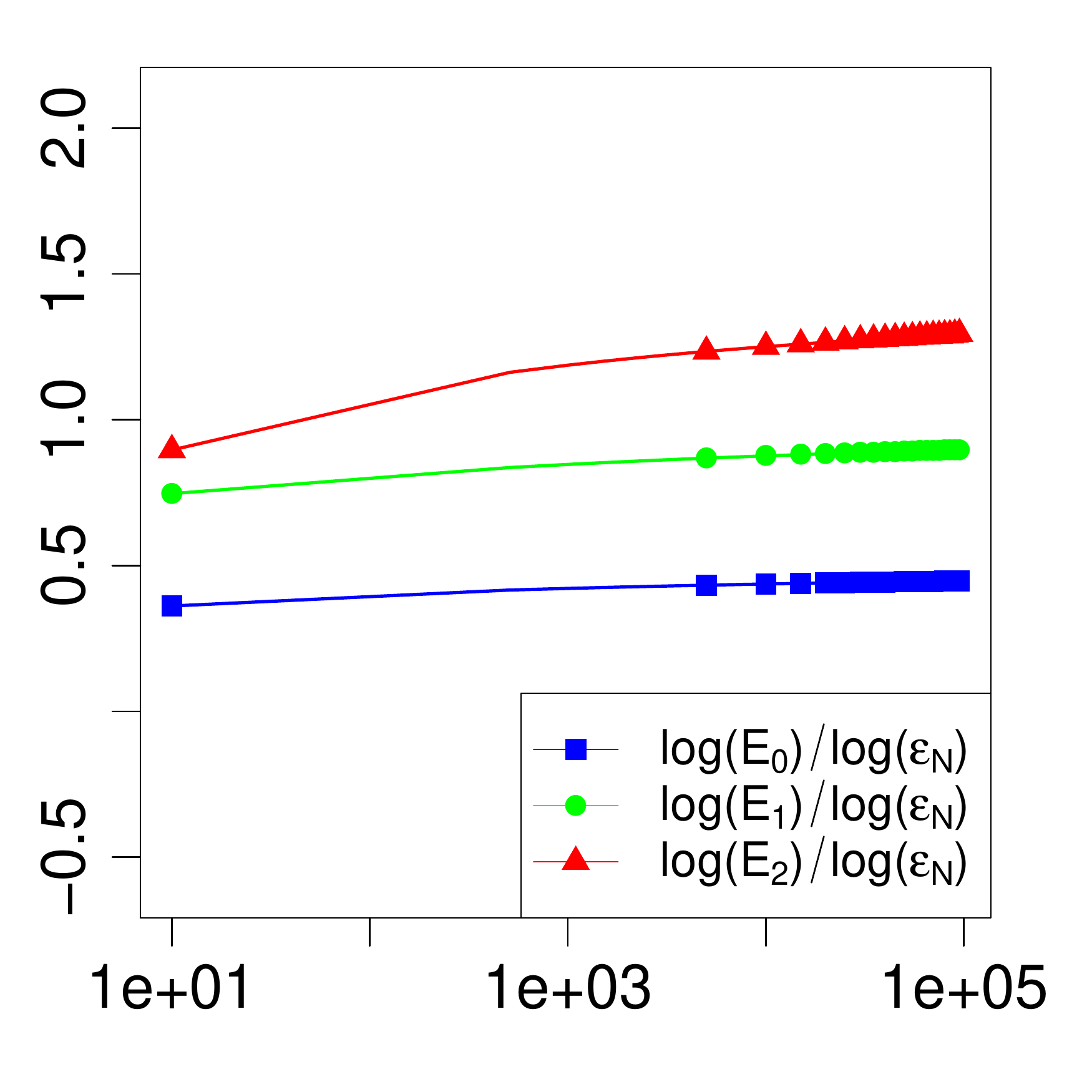}
            \vspace{-0.9cm}
            \caption{$\bb{\alpha} = (3,1)$ and $\beta = 1$}
        \end{subfigure}
        \quad
        \begin{subfigure}[b]{0.22\textwidth}
            \centering
            \includegraphics[width=\textwidth, height=0.85\textwidth]{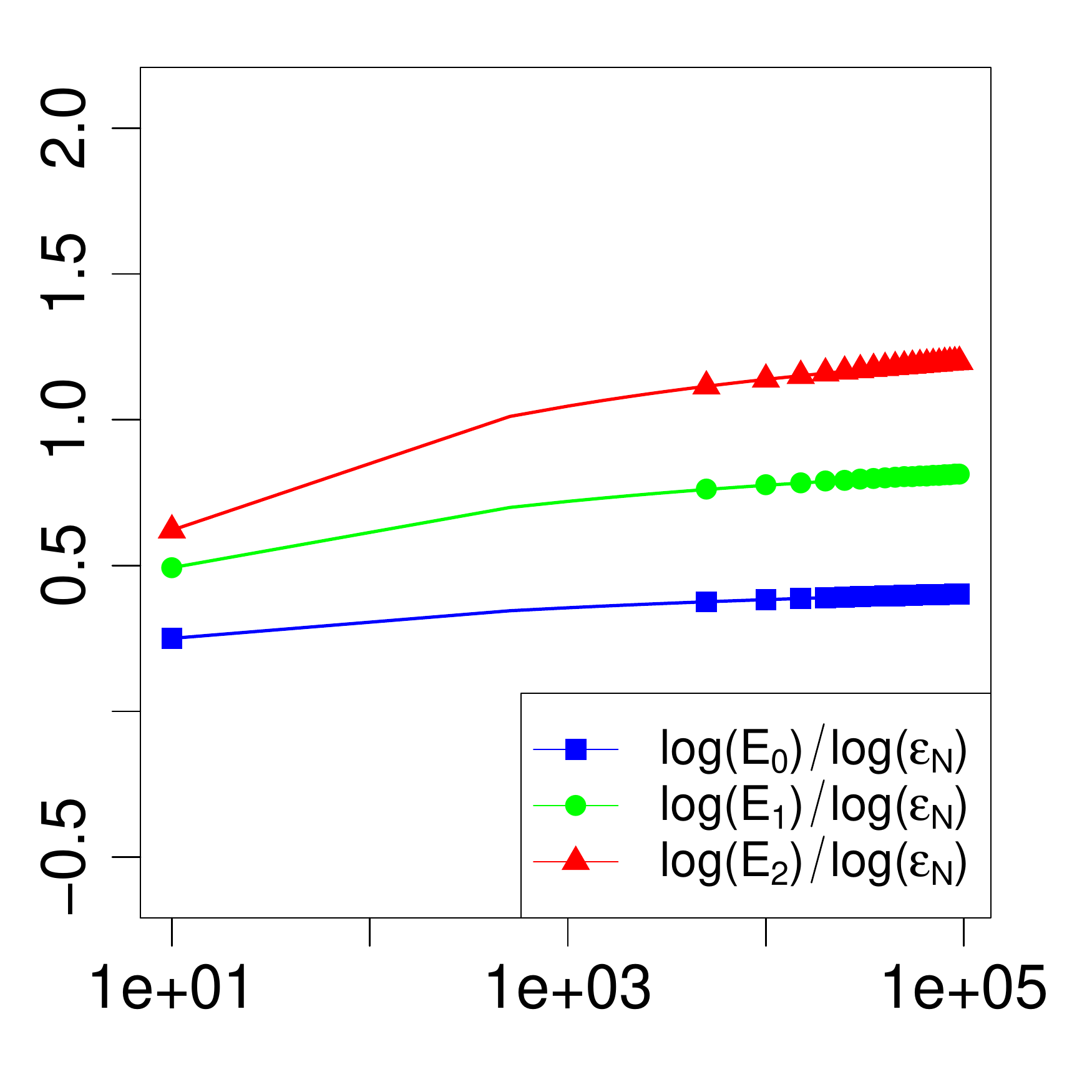}
            \vspace{-0.9cm}
            \caption{$\bb{\alpha} = (3,2)$ and $\beta = 1$}
        \end{subfigure}
        \quad
        \begin{subfigure}[b]{0.22\textwidth}
            \centering
            \includegraphics[width=\textwidth, height=0.85\textwidth]{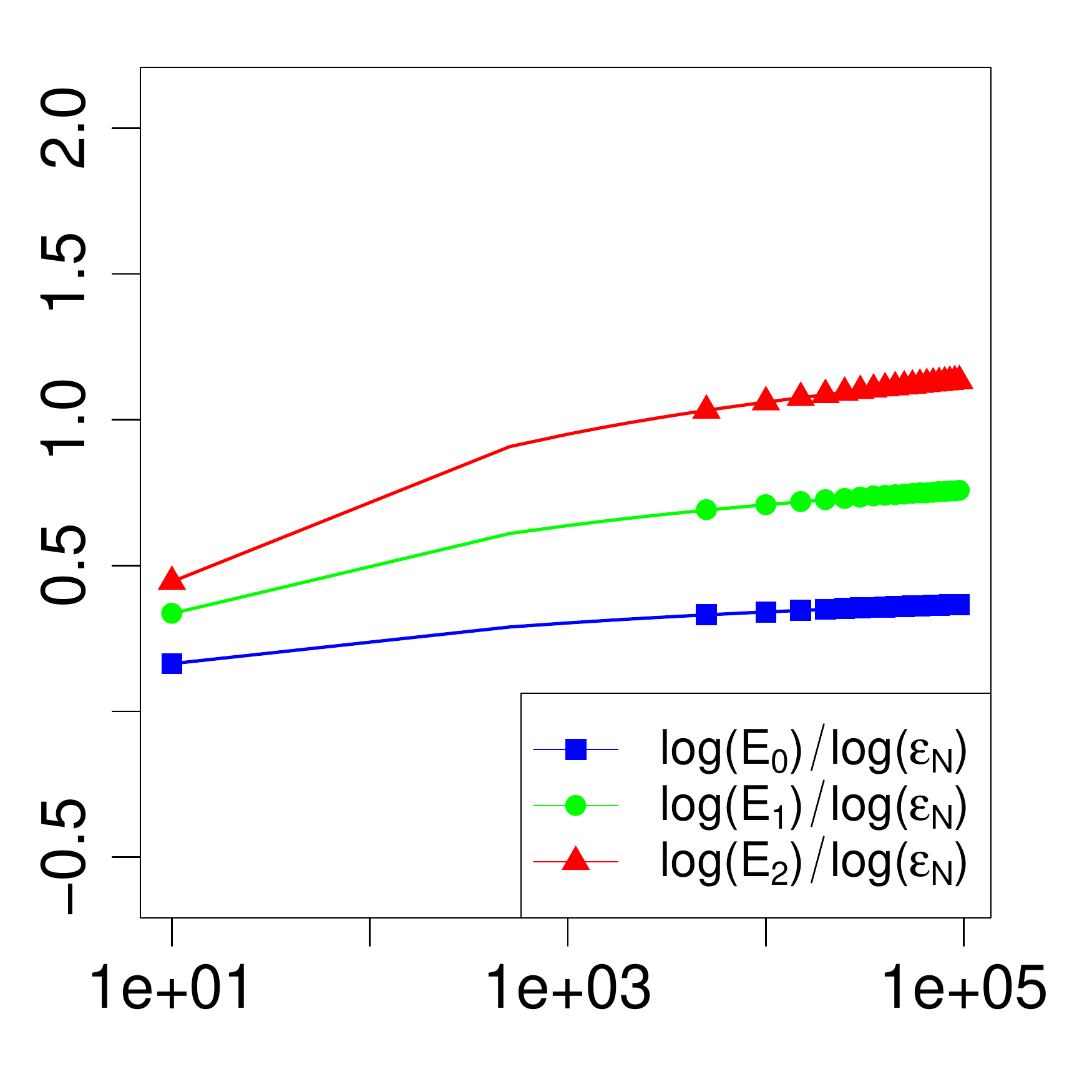}
            \vspace{-0.9cm}
            \caption{$\bb{\alpha} = (3,3)$ and $\beta = 1$}
        \end{subfigure}
        \quad
        \begin{subfigure}[b]{0.22\textwidth}
            \centering
            \includegraphics[width=\textwidth, height=0.85\textwidth]{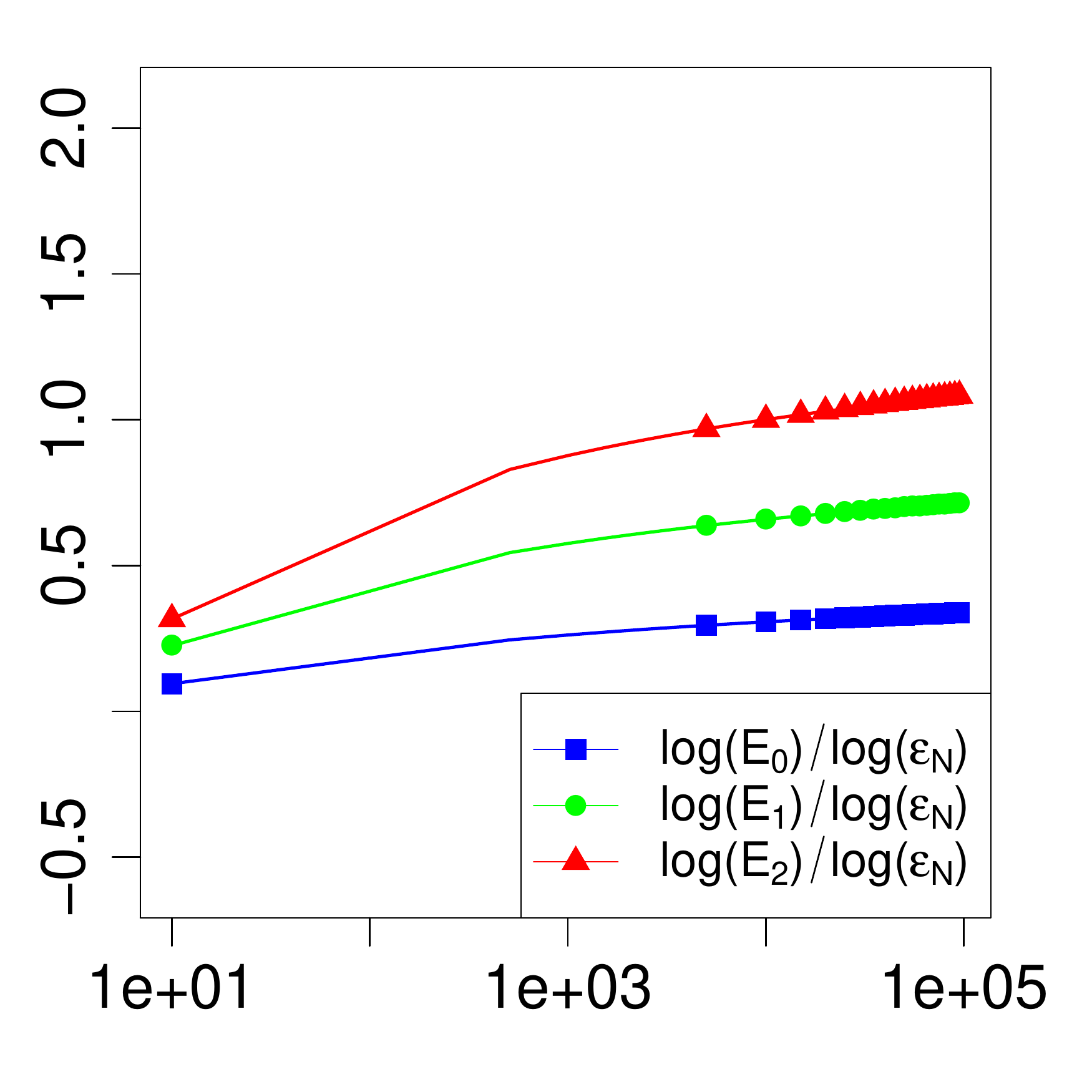}
            \vspace{-0.9cm}
            \caption{$\bb{\alpha} = (3,4)$ and $\beta = 1$}
        \end{subfigure}
        \caption{Plots of $\log E_i / \log \e_N$ as a function of $N$, for various choices of $\bb{\alpha}$, when $\beta = 1$. The horizontal axis is on a logarithmic scale. The plots confirm \eqref{eq:liminf.exponent.bound} and bring strong evidence for the validity of Theorem~\ref{thm:p.k.expansion}.}
        \label{fig:error.exponents.plots.beta.1}
    \end{figure}

    \phantom{vertical spacing}
    \begin{figure}[H]
        \captionsetup[subfigure]{labelformat=empty}
        \captionsetup{width=0.8\linewidth}
        \vspace{-0.5cm}
        \centering
        \begin{subfigure}[b]{0.22\textwidth}
            \centering
            \includegraphics[width=\textwidth, height=0.85\textwidth]{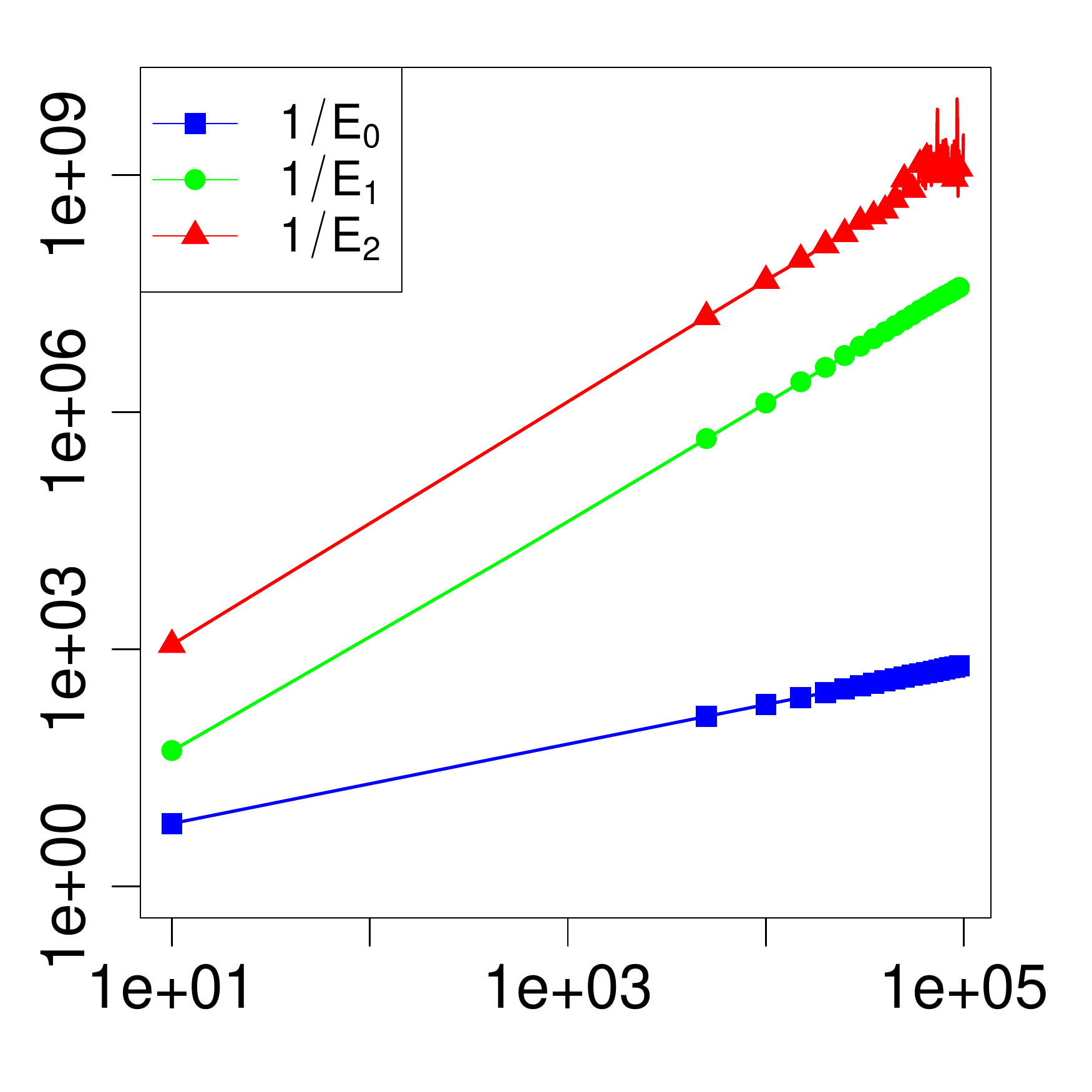}
            \vspace{-0.9cm}
            \caption{$\bb{\alpha} = (1,1)$ and $\beta = 2$}
        \end{subfigure}
        \quad
        \begin{subfigure}[b]{0.22\textwidth}
            \centering
            \includegraphics[width=\textwidth, height=0.85\textwidth]{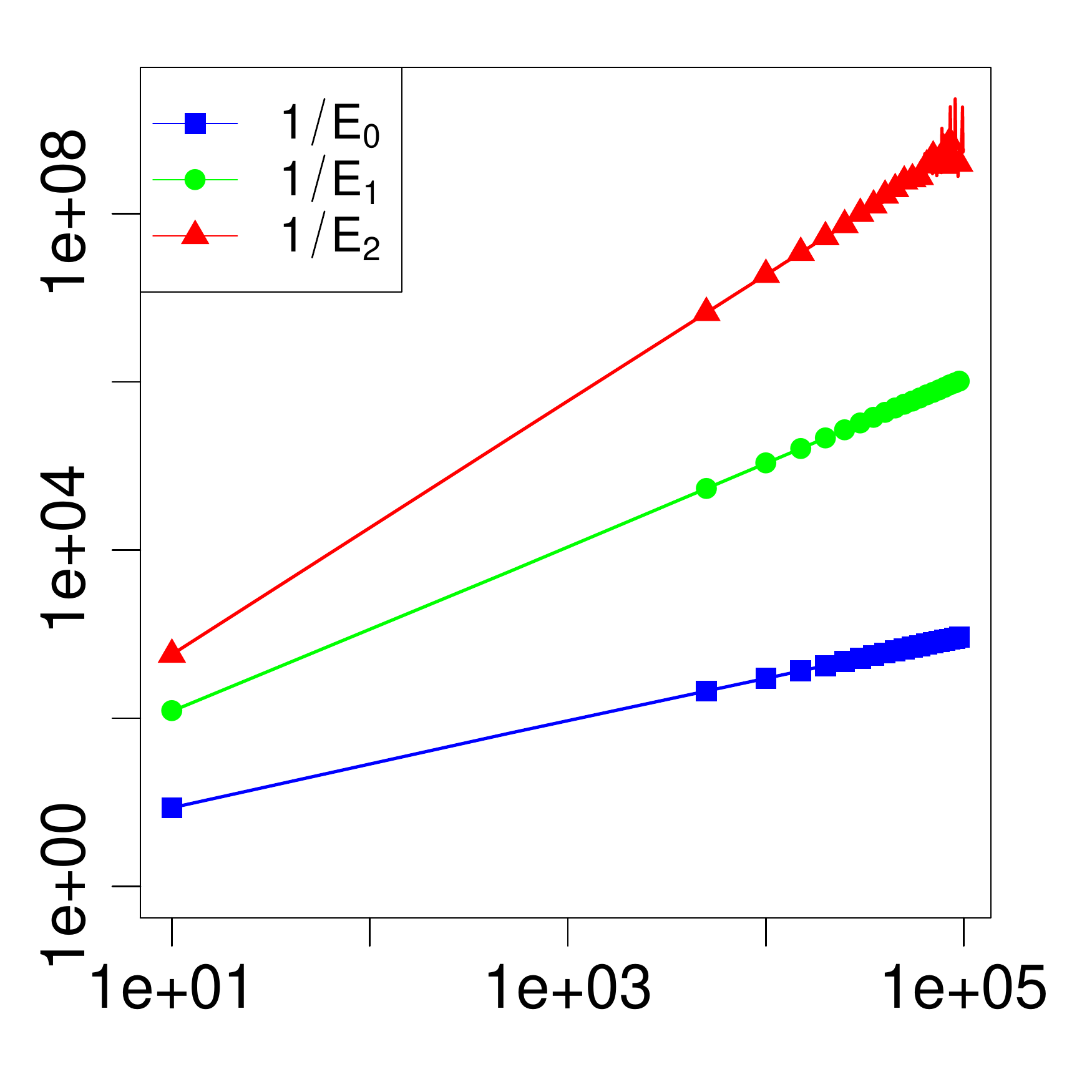}
            \vspace{-0.9cm}
            \caption{$\bb{\alpha} = (1,2)$ and $\beta = 2$}
        \end{subfigure}
        \quad
        \begin{subfigure}[b]{0.22\textwidth}
            \centering
            \includegraphics[width=\textwidth, height=0.85\textwidth]{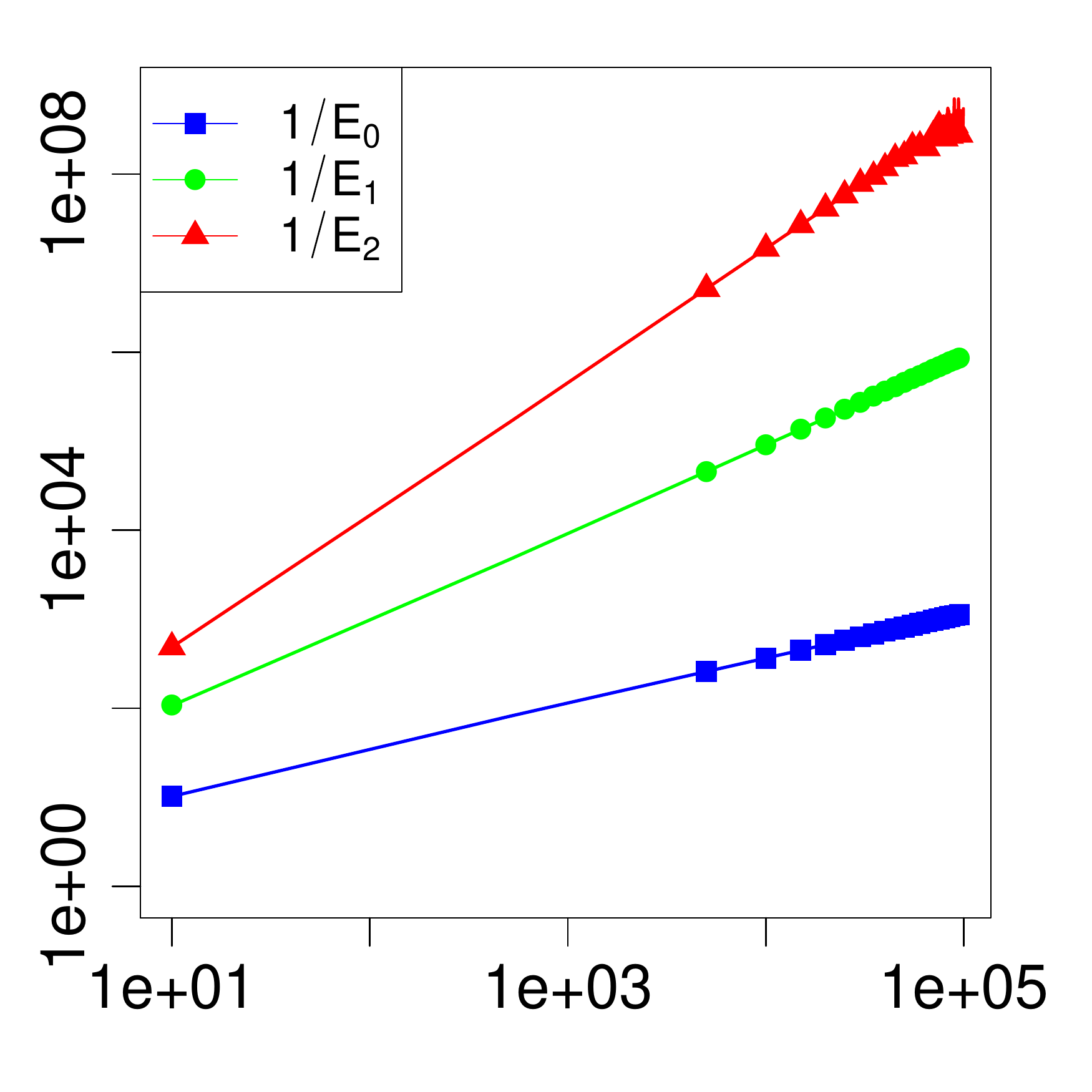}
            \vspace{-0.9cm}
            \caption{$\bb{\alpha} = (1,3)$ and $\beta = 2$}
        \end{subfigure}
        \quad
        \begin{subfigure}[b]{0.22\textwidth}
            \centering
            \includegraphics[width=\textwidth, height=0.85\textwidth]{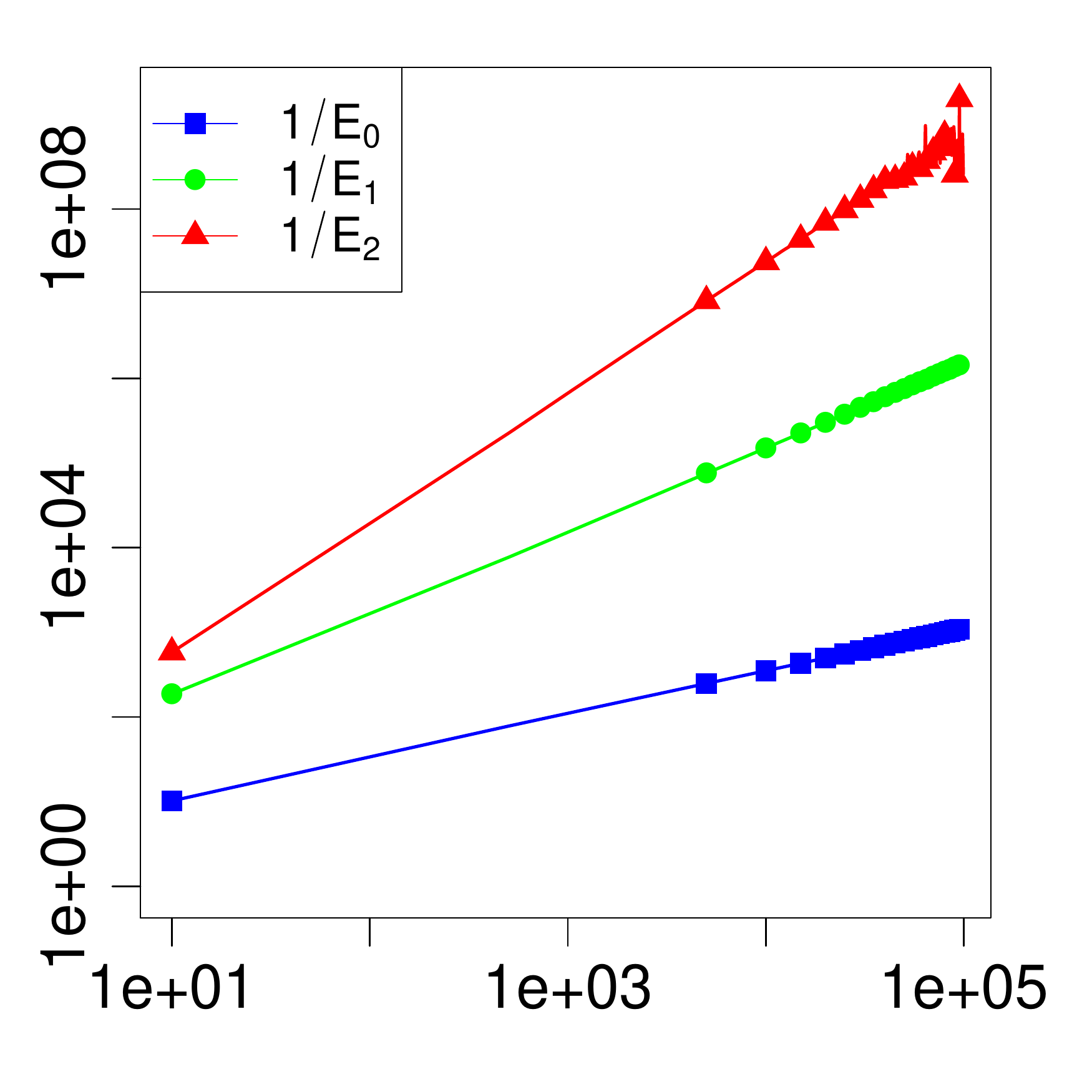}
            \vspace{-0.9cm}
            \caption{$\bb{\alpha} = (1,4)$ and $\beta = 2$}
        \end{subfigure}
        \begin{subfigure}[b]{0.22\textwidth}
            \centering
            \includegraphics[width=\textwidth, height=0.85\textwidth]{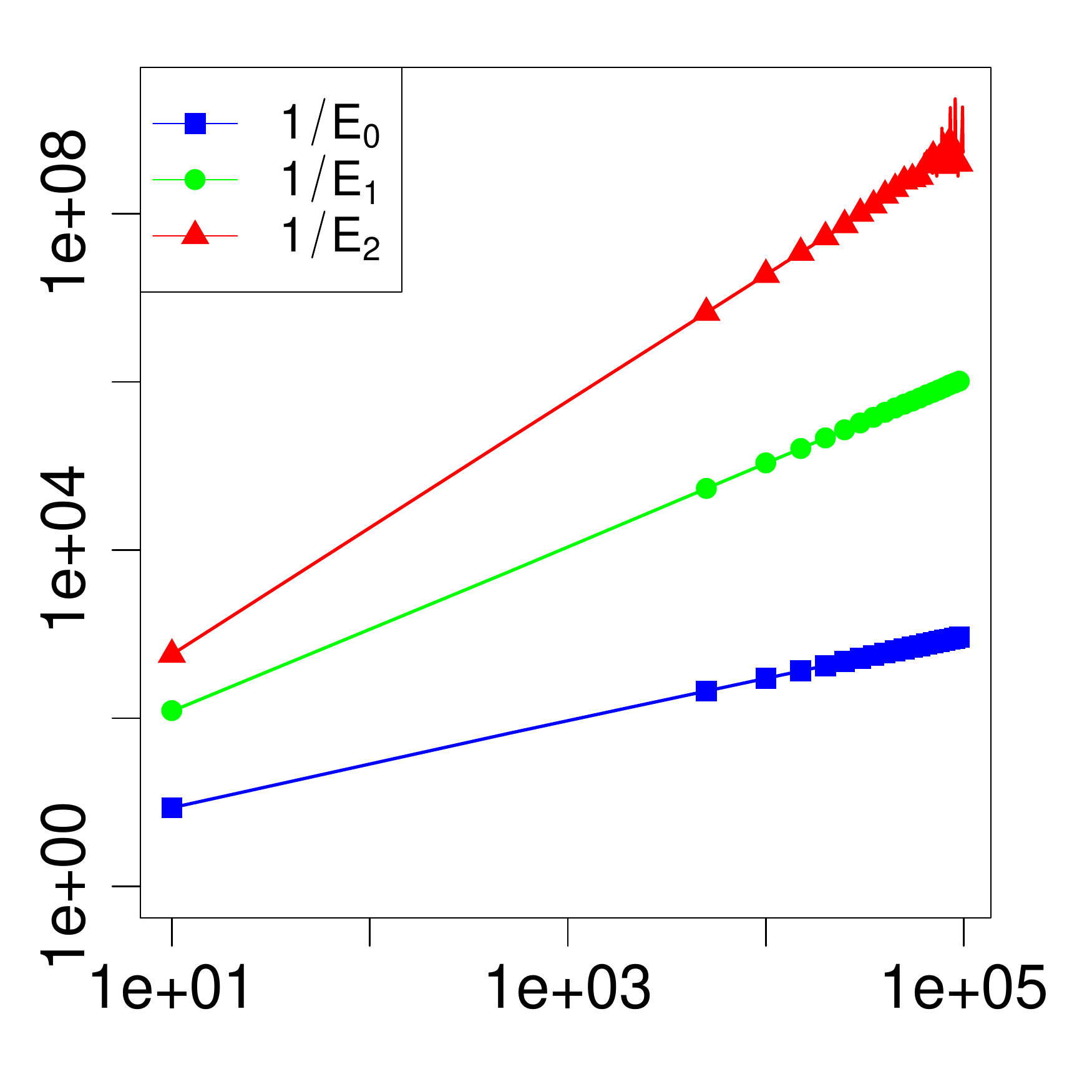}
            \vspace{-0.9cm}
            \caption{$\bb{\alpha} = (2,1)$ and $\beta = 2$}
        \end{subfigure}
        \quad
        \begin{subfigure}[b]{0.22\textwidth}
            \centering
            \includegraphics[width=\textwidth, height=0.85\textwidth]{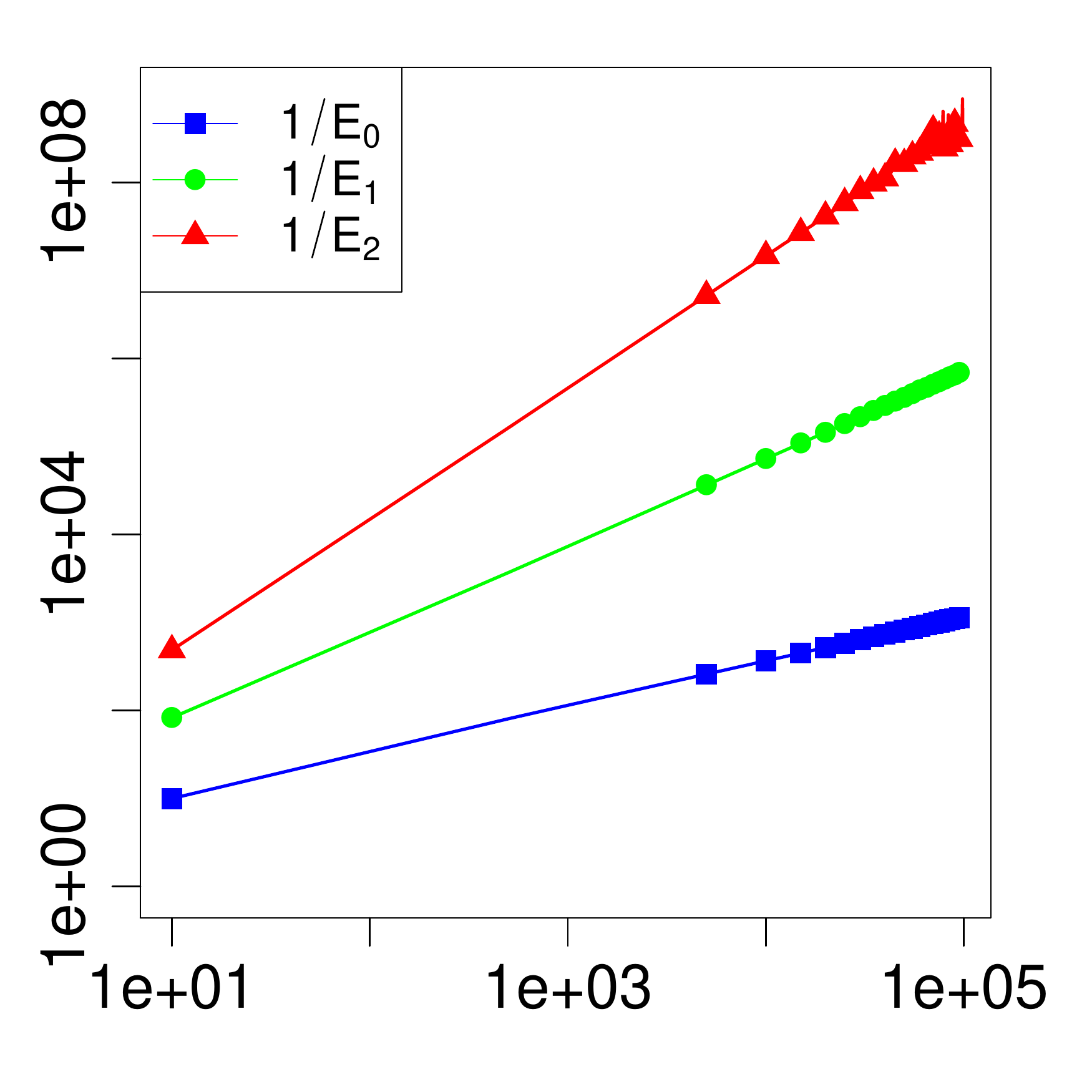}
            \vspace{-0.9cm}
            \caption{$\bb{\alpha} = (2,2)$ and $\beta = 2$}
        \end{subfigure}
        \quad
        \begin{subfigure}[b]{0.22\textwidth}
            \centering
            \includegraphics[width=\textwidth, height=0.85\textwidth]{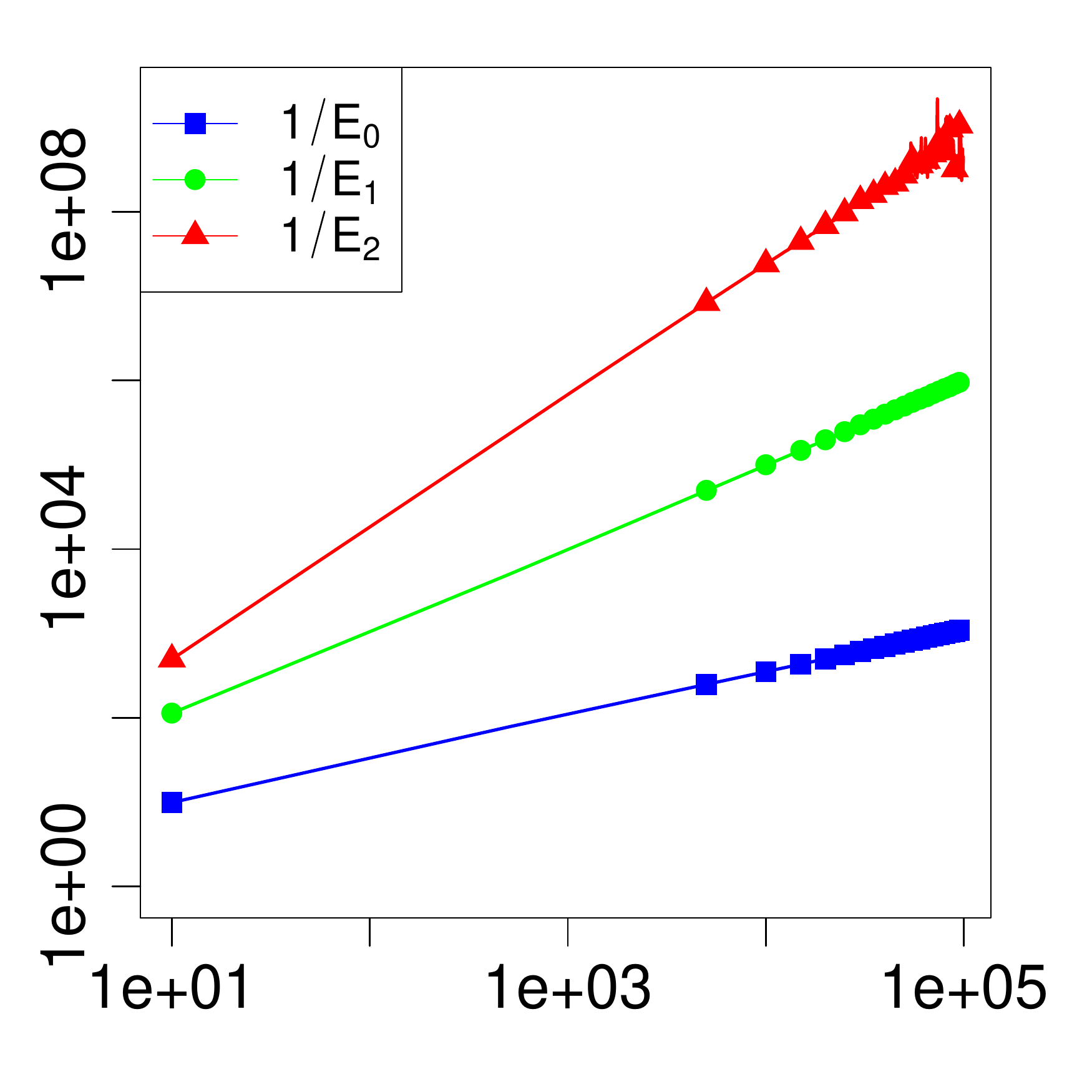}
            \vspace{-0.9cm}
            \caption{$\bb{\alpha} = (2,3)$ and $\beta = 2$}
        \end{subfigure}
        \quad
        \begin{subfigure}[b]{0.22\textwidth}
            \centering
            \includegraphics[width=\textwidth, height=0.85\textwidth]{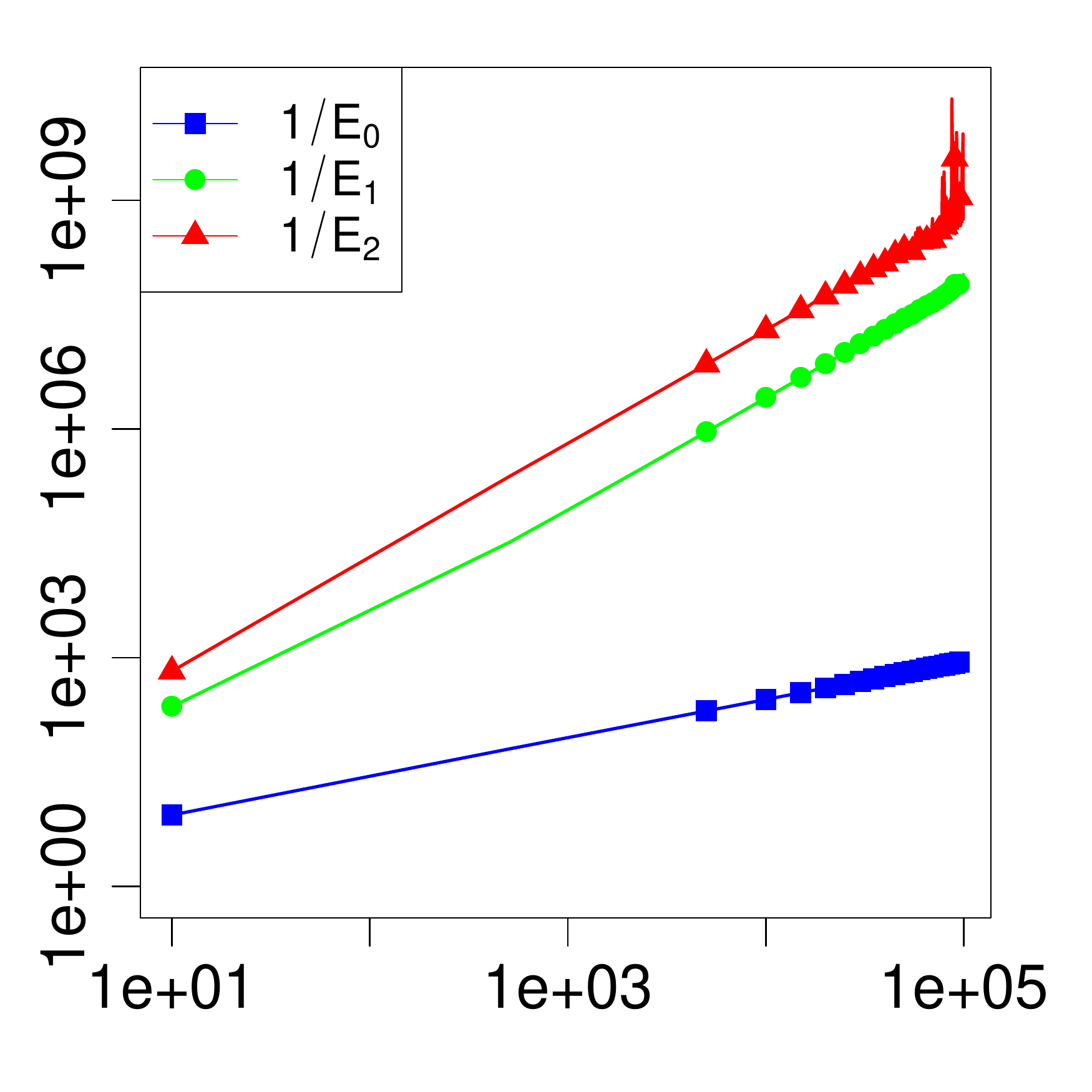}
            \vspace{-0.9cm}
            \caption{$\bb{\alpha} = (2,4)$ and $\beta = 2$}
        \end{subfigure}
        \begin{subfigure}[b]{0.22\textwidth}
            \centering
            \includegraphics[width=\textwidth, height=0.85\textwidth]{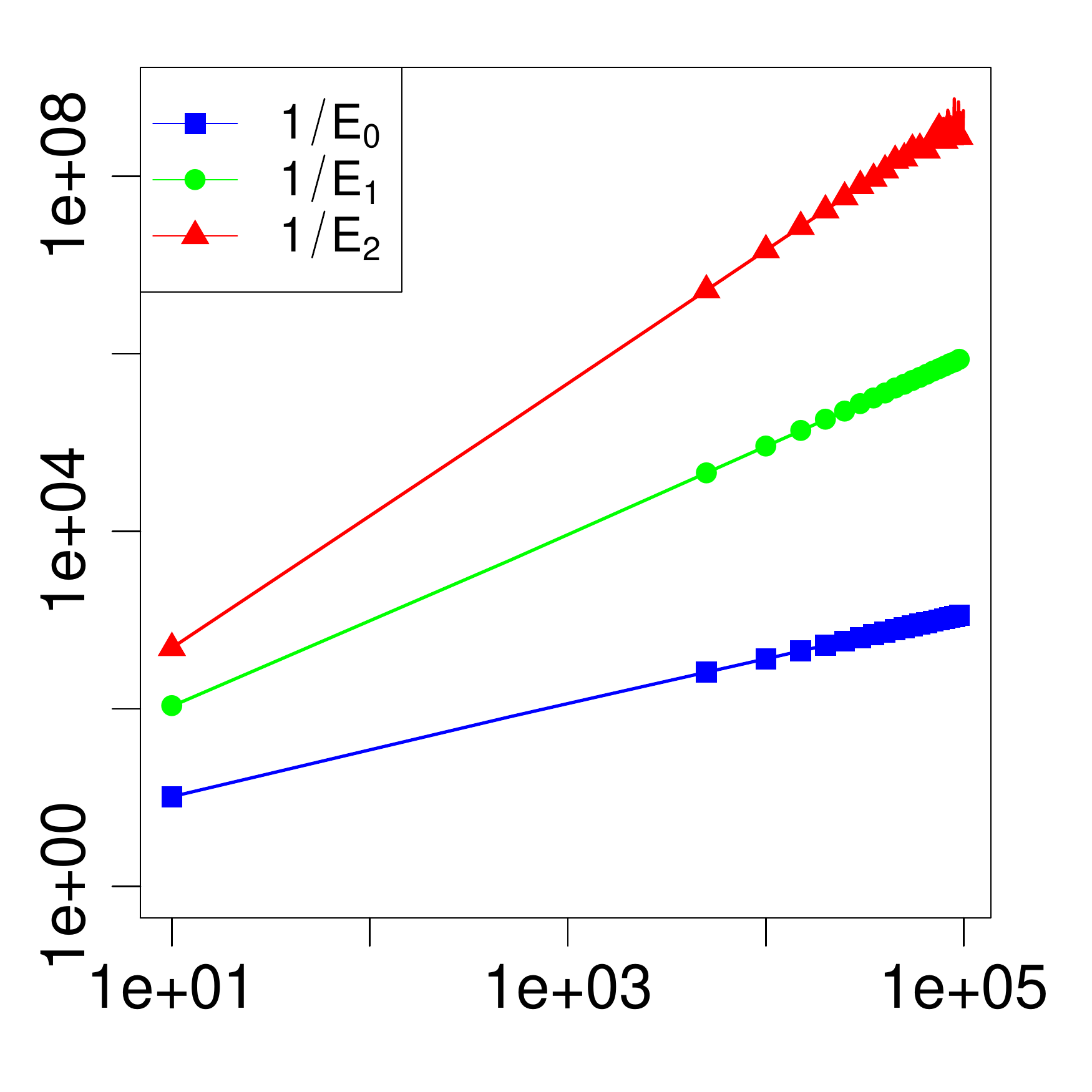}
            \vspace{-0.9cm}
            \caption{$\bb{\alpha} = (3,1)$ and $\beta = 2$}
        \end{subfigure}
        \quad
        \begin{subfigure}[b]{0.22\textwidth}
            \centering
            \includegraphics[width=\textwidth, height=0.85\textwidth]{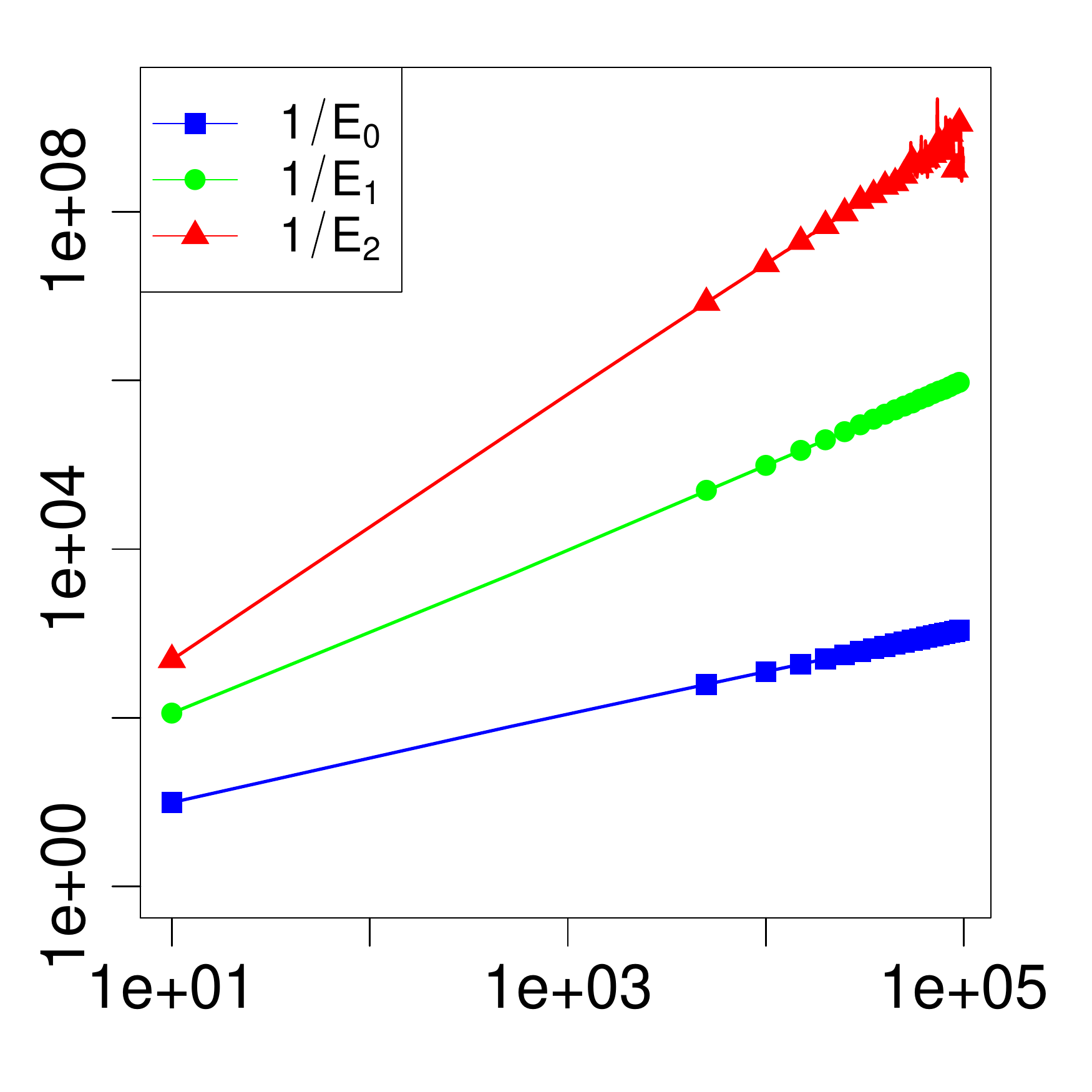}
            \vspace{-0.9cm}
            \caption{$\bb{\alpha} = (3,2)$ and $\beta = 2$}
        \end{subfigure}
        \quad
        \begin{subfigure}[b]{0.22\textwidth}
            \centering
            \includegraphics[width=\textwidth, height=0.85\textwidth]{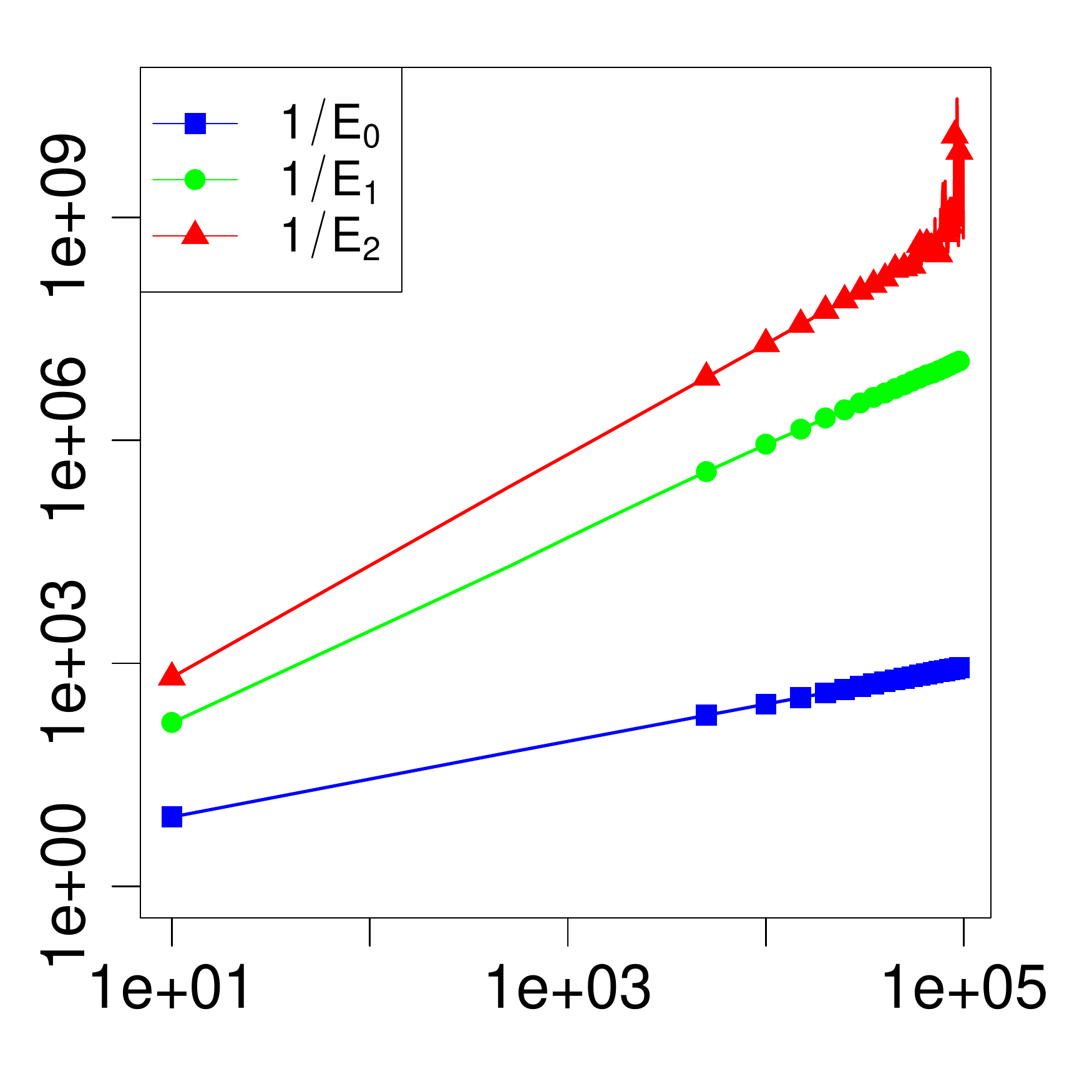}
            \vspace{-0.9cm}
            \caption{$\bb{\alpha} = (3,3)$ and $\beta = 2$}
        \end{subfigure}
        \quad
        \begin{subfigure}[b]{0.22\textwidth}
            \centering
            \includegraphics[width=\textwidth, height=0.85\textwidth]{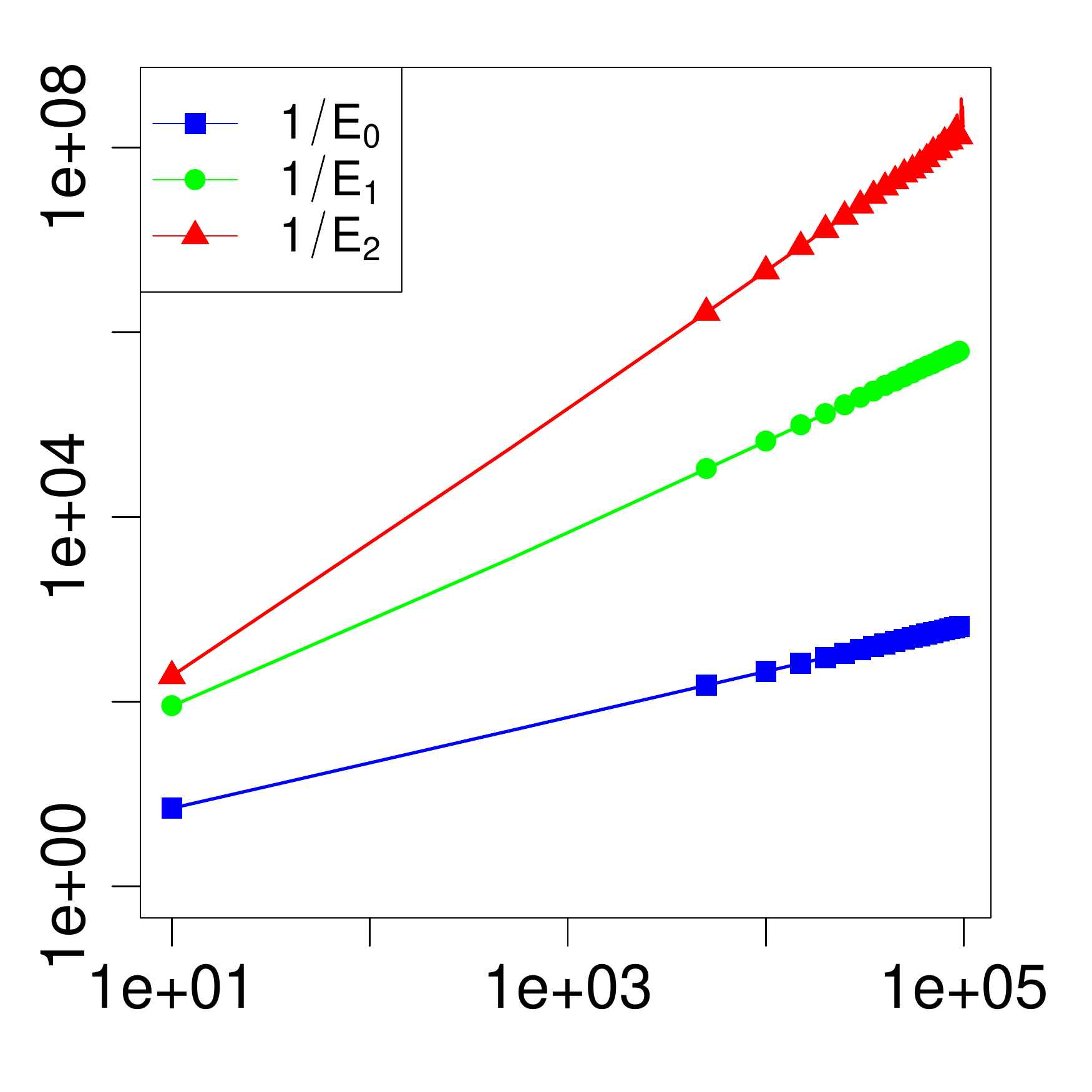}
            \vspace{-0.9cm}
            \caption{$\bb{\alpha} = (3,4)$ and $\beta = 2$}
        \end{subfigure}
        \caption{Plots of $1 / E_i$ as a function of $N$, for various choices of $\bb{\alpha}$, when $\beta = 2$. Both the horizontal and vertical axes are on a logarithmic scale. The plots clearly illustrate how the addition of correction terms from Theorem~\ref{thm:p.k.expansion} to the base approximation \eqref{eq:E.0} improves it.}
        \label{fig:loglog.errors.plots.beta.2}
    \end{figure}
    \begin{figure}[H]
        \captionsetup[subfigure]{labelformat=empty}
        \captionsetup{width=0.8\linewidth}
        \vspace{-0.5cm}
        \centering
        \begin{subfigure}[b]{0.22\textwidth}
            \centering
            \includegraphics[width=\textwidth, height=0.85\textwidth]{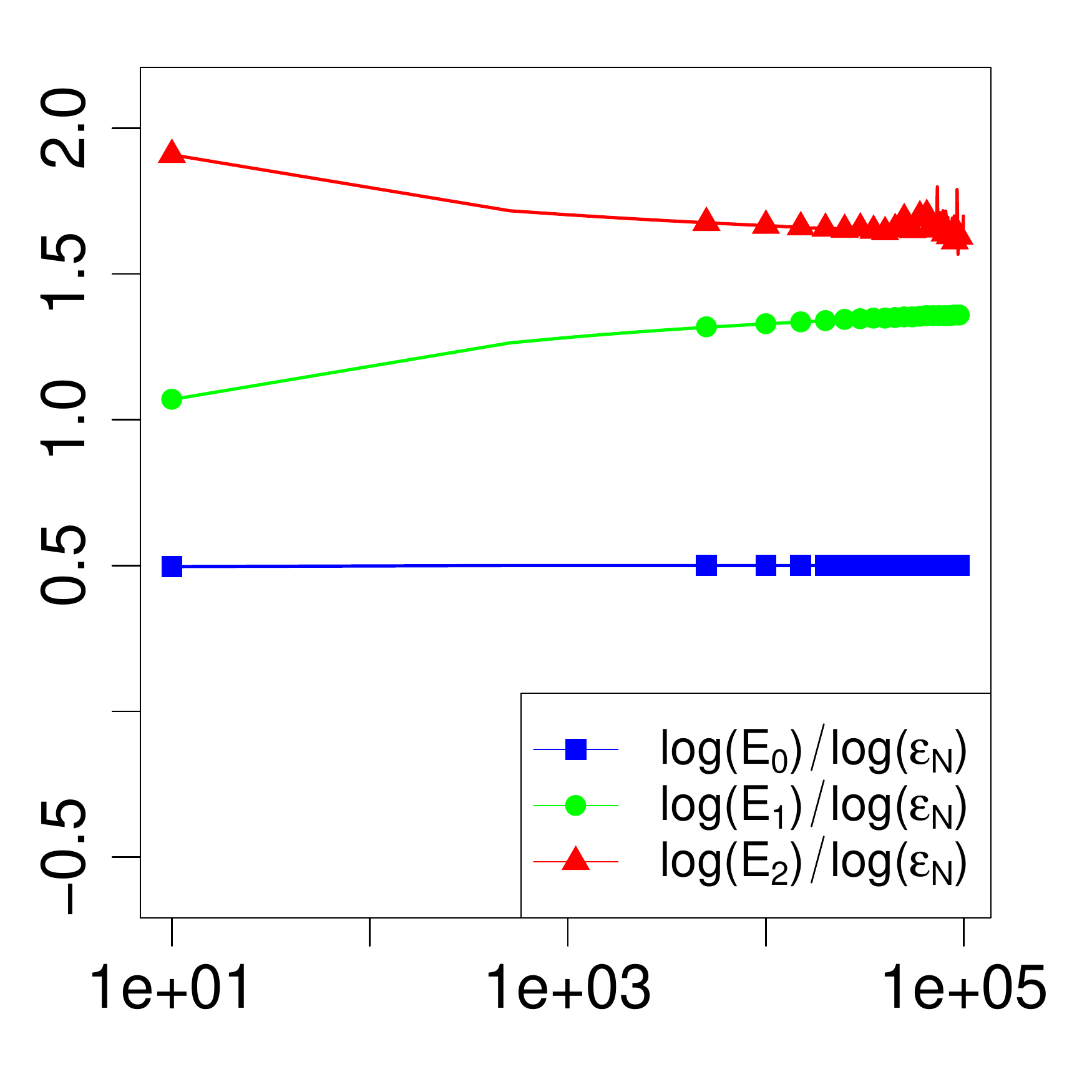}
            \vspace{-0.9cm}
            \caption{$\bb{\alpha} = (1,1)$ and $\beta = 2$}
        \end{subfigure}
        \quad
        \begin{subfigure}[b]{0.22\textwidth}
            \centering
            \includegraphics[width=\textwidth, height=0.85\textwidth]{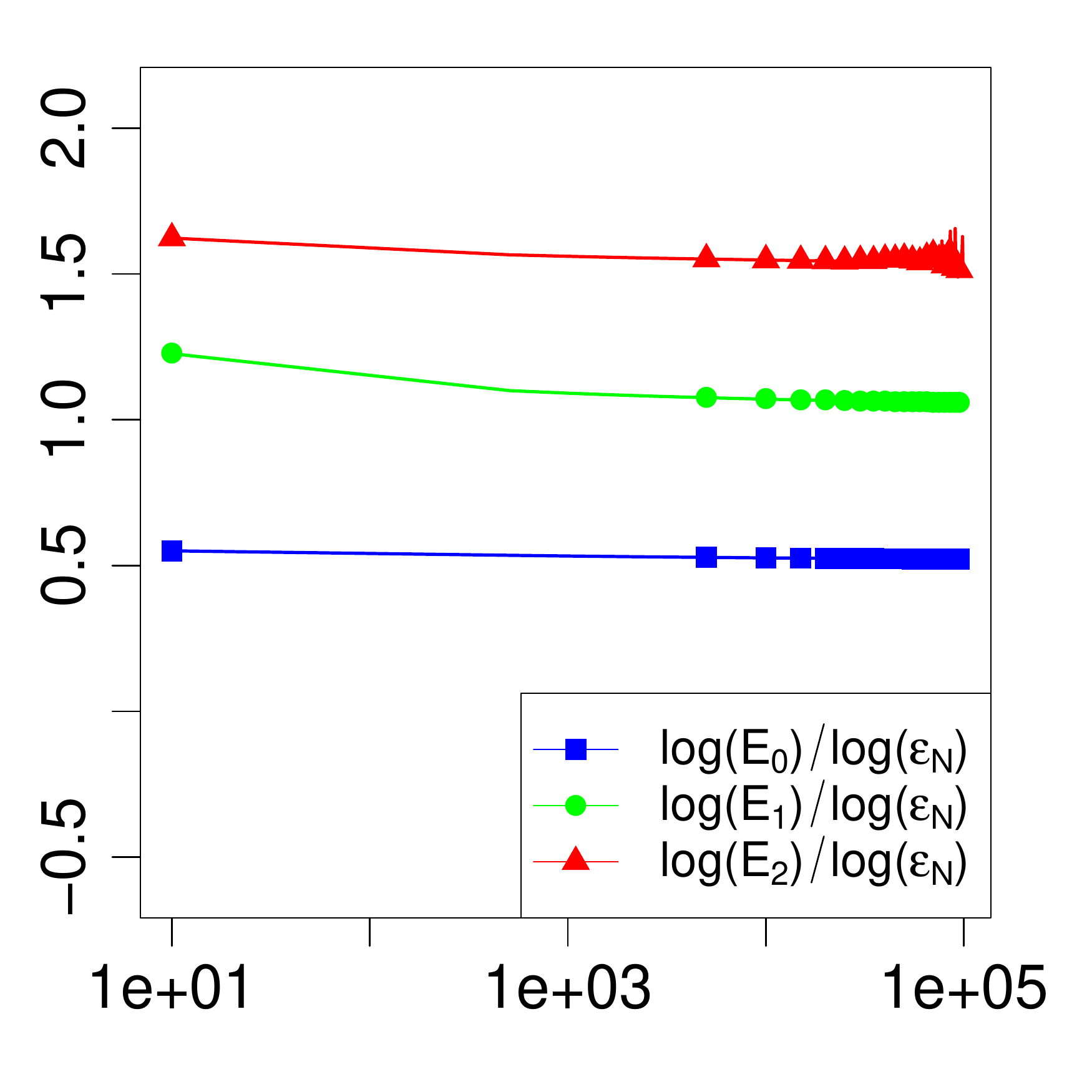}
            \vspace{-0.9cm}
            \caption{$\bb{\alpha} = (1,2)$ and $\beta = 2$}
        \end{subfigure}
        \quad
        \begin{subfigure}[b]{0.22\textwidth}
            \centering
            \includegraphics[width=\textwidth, height=0.85\textwidth]{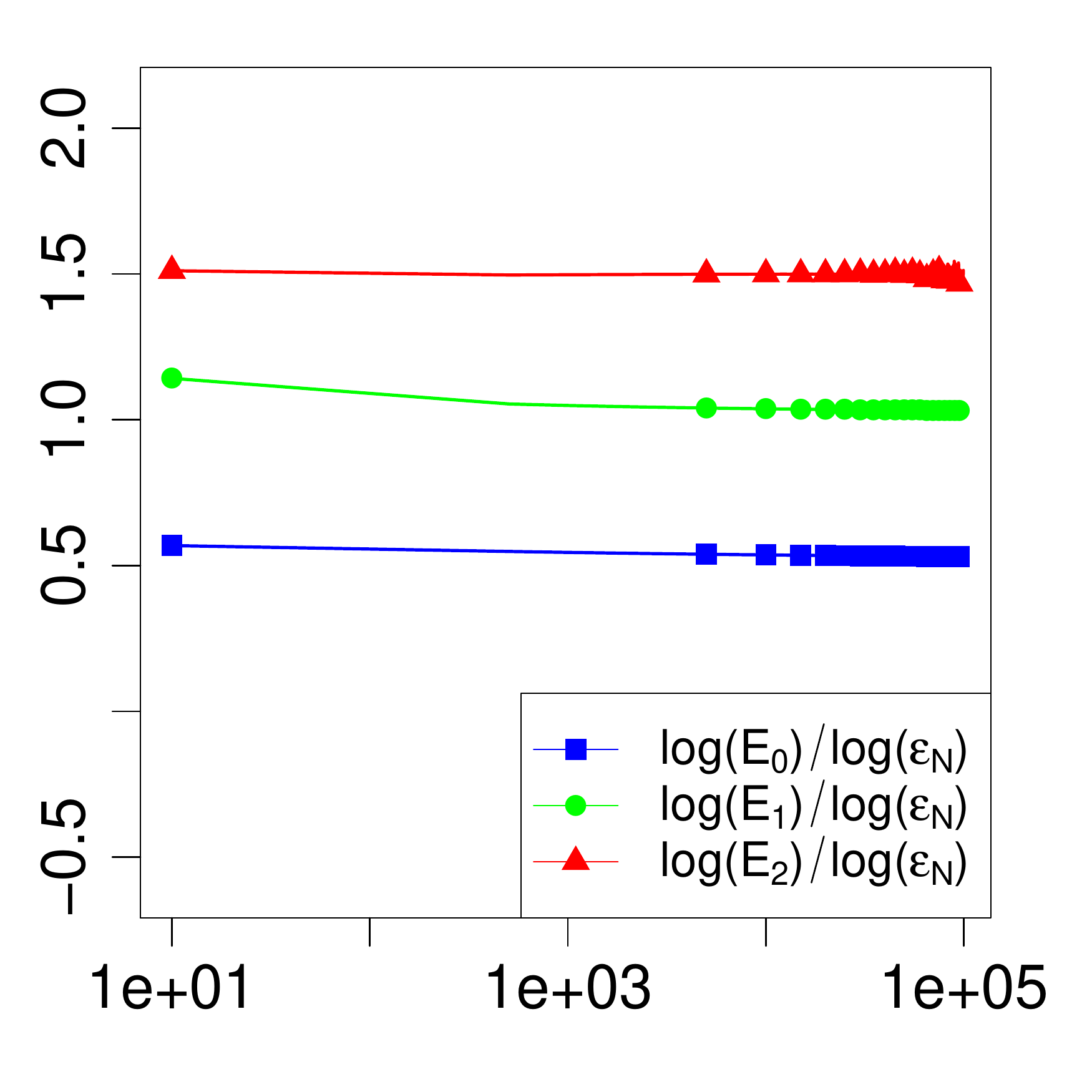}
            \vspace{-0.9cm}
            \caption{$\bb{\alpha} = (1,3)$ and $\beta = 2$}
        \end{subfigure}
        \quad
        \begin{subfigure}[b]{0.22\textwidth}
            \centering
            \includegraphics[width=\textwidth, height=0.85\textwidth]{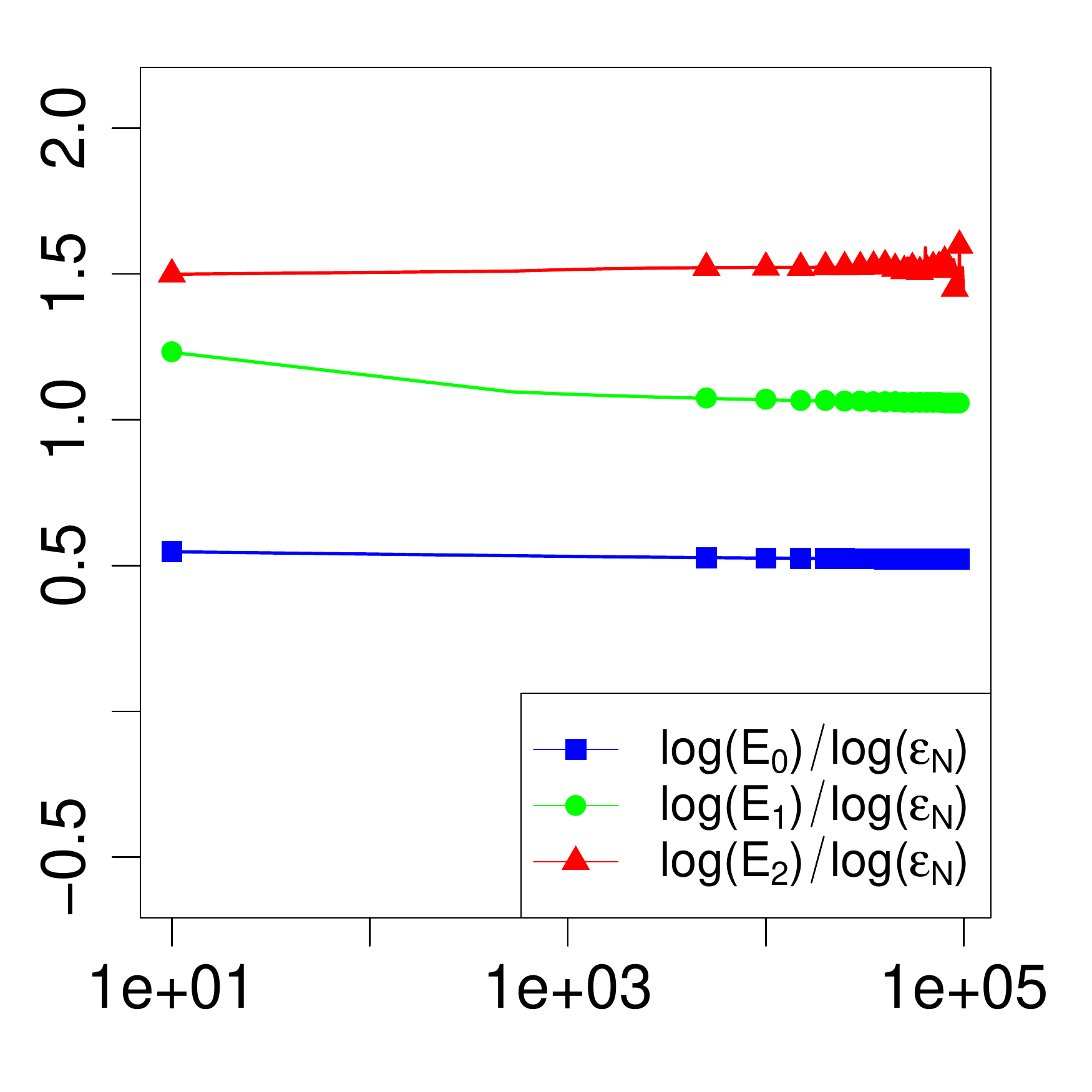}
            \vspace{-0.9cm}
            \caption{$\bb{\alpha} = (1,4)$ and $\beta = 2$}
        \end{subfigure}
        \begin{subfigure}[b]{0.22\textwidth}
            \centering
            \includegraphics[width=\textwidth, height=0.85\textwidth]{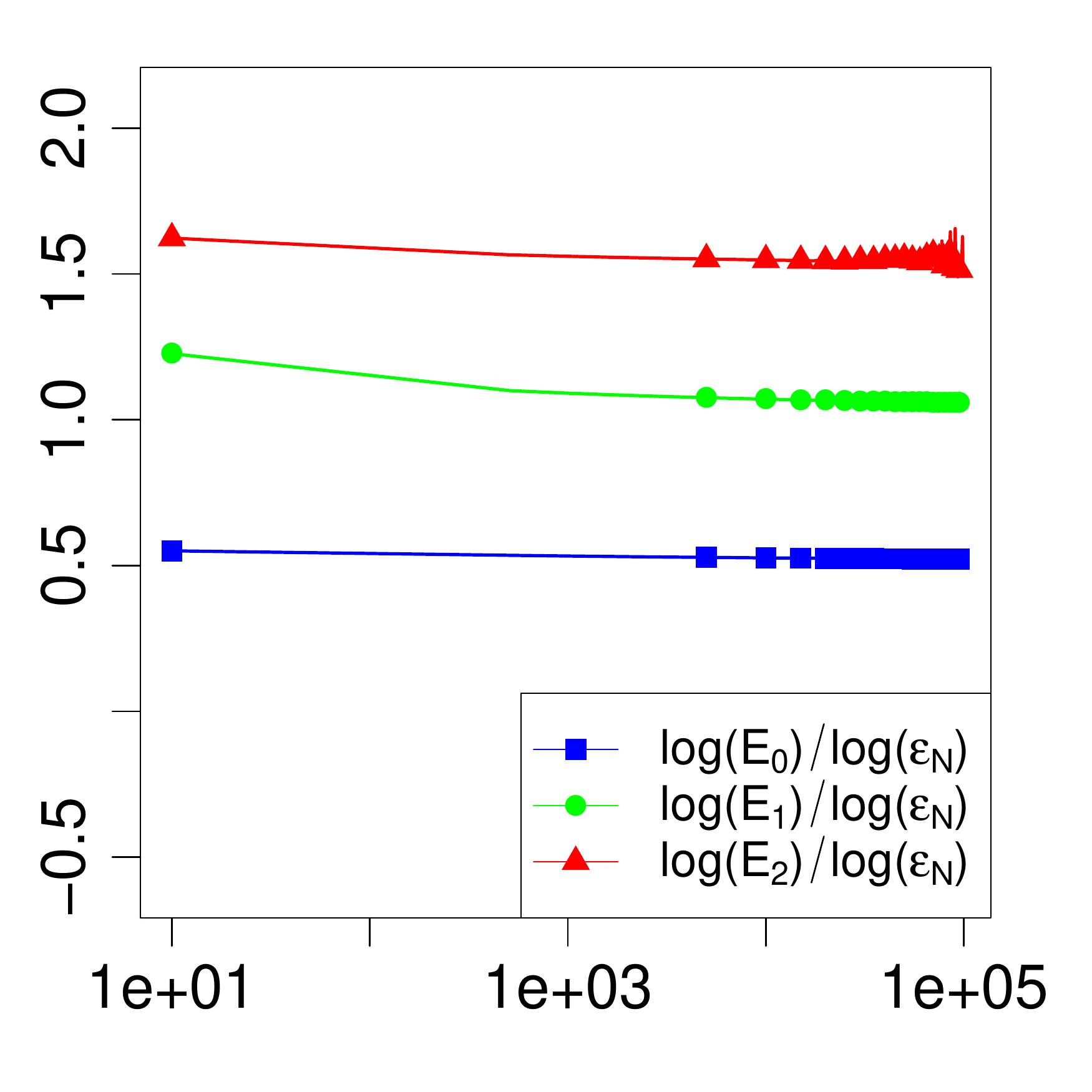}
            \vspace{-0.9cm}
            \caption{$\bb{\alpha} = (2,1)$ and $\beta = 2$}
        \end{subfigure}
        \quad
        \begin{subfigure}[b]{0.22\textwidth}
            \centering
            \includegraphics[width=\textwidth, height=0.85\textwidth]{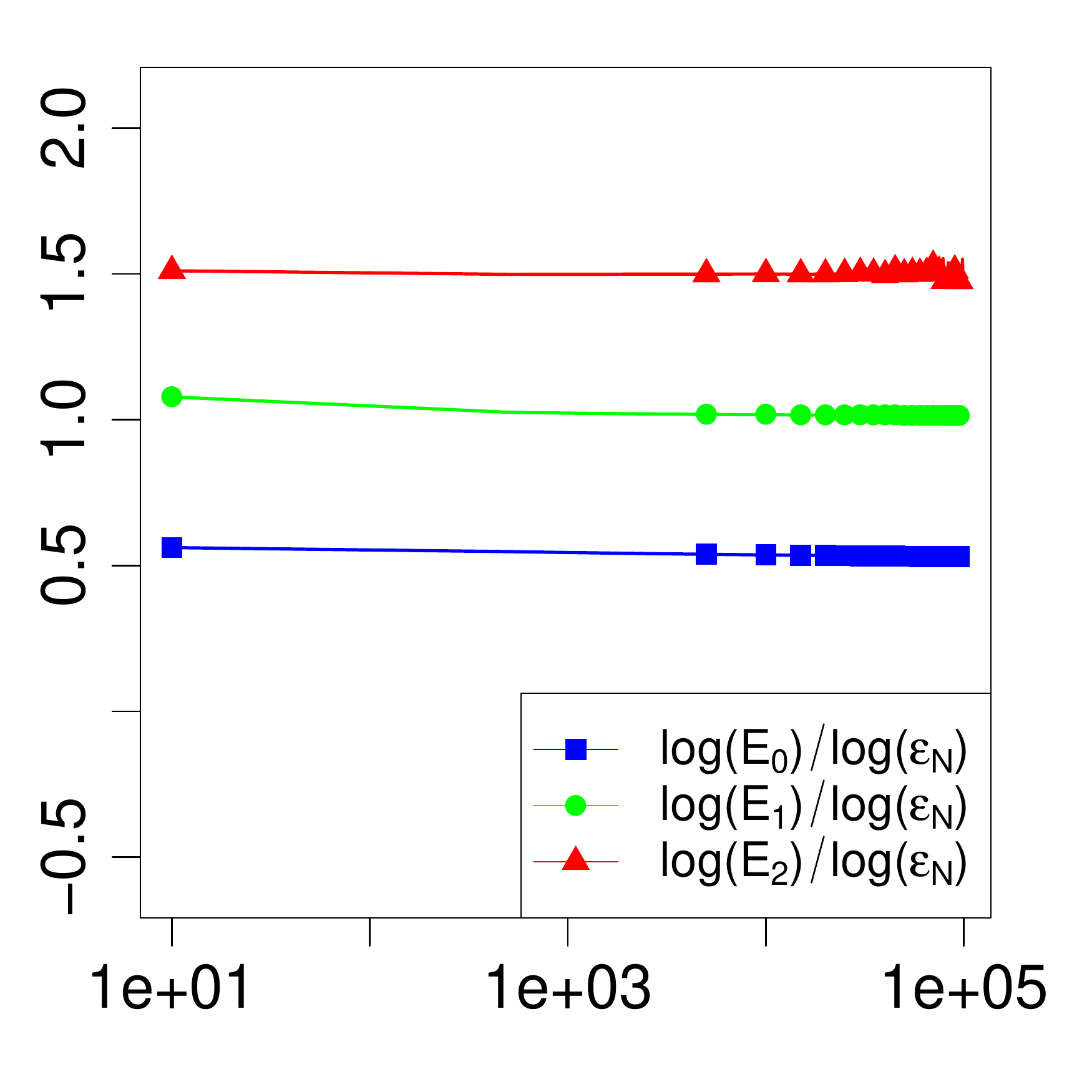}
            \vspace{-0.9cm}
            \caption{$\bb{\alpha} = (2,2)$ and $\beta = 2$}
        \end{subfigure}
        \quad
        \begin{subfigure}[b]{0.22\textwidth}
            \centering
            \includegraphics[width=\textwidth, height=0.85\textwidth]{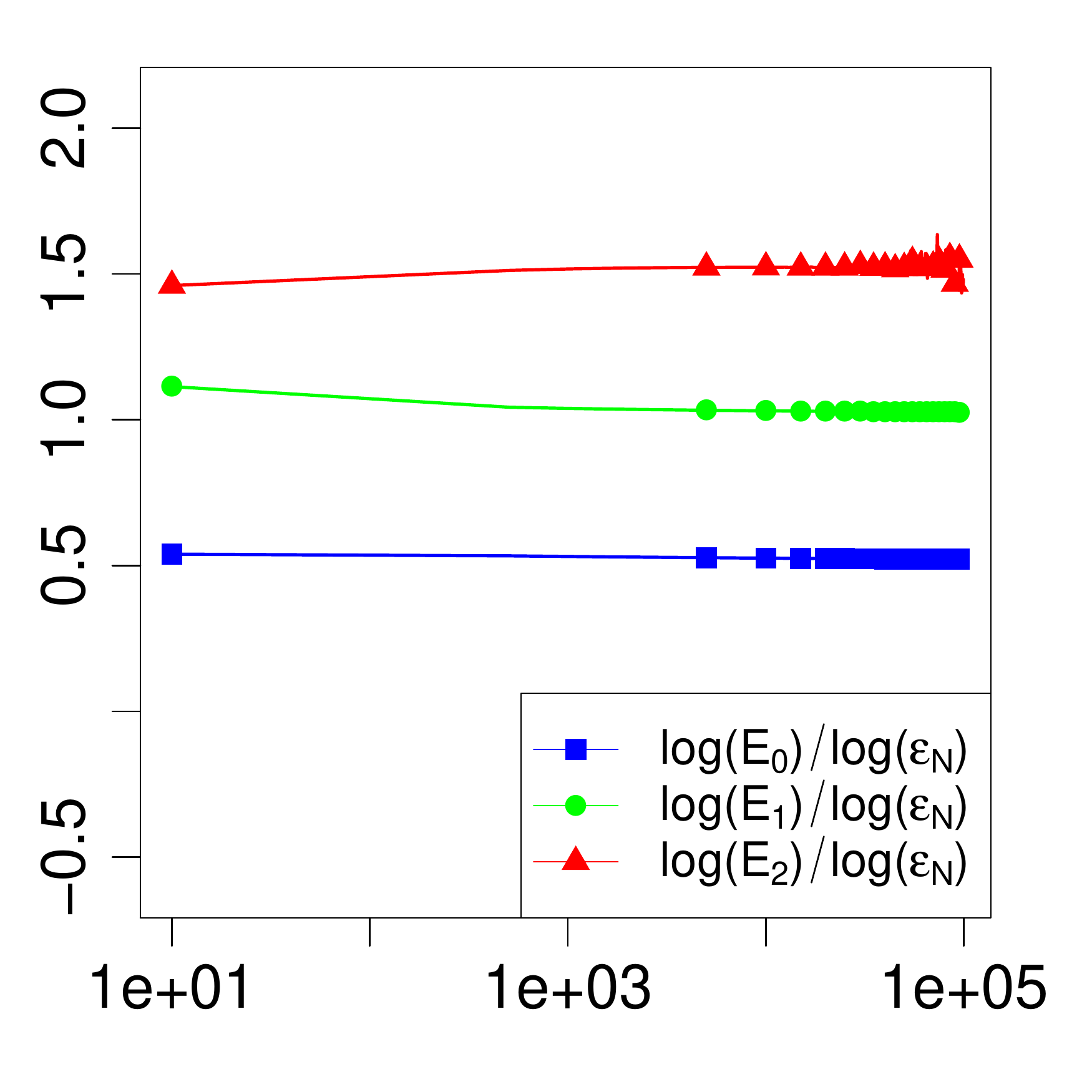}
            \vspace{-0.9cm}
            \caption{$\bb{\alpha} = (2,3)$ and $\beta = 2$}
        \end{subfigure}
        \quad
        \begin{subfigure}[b]{0.22\textwidth}
            \centering
            \includegraphics[width=\textwidth, height=0.85\textwidth]{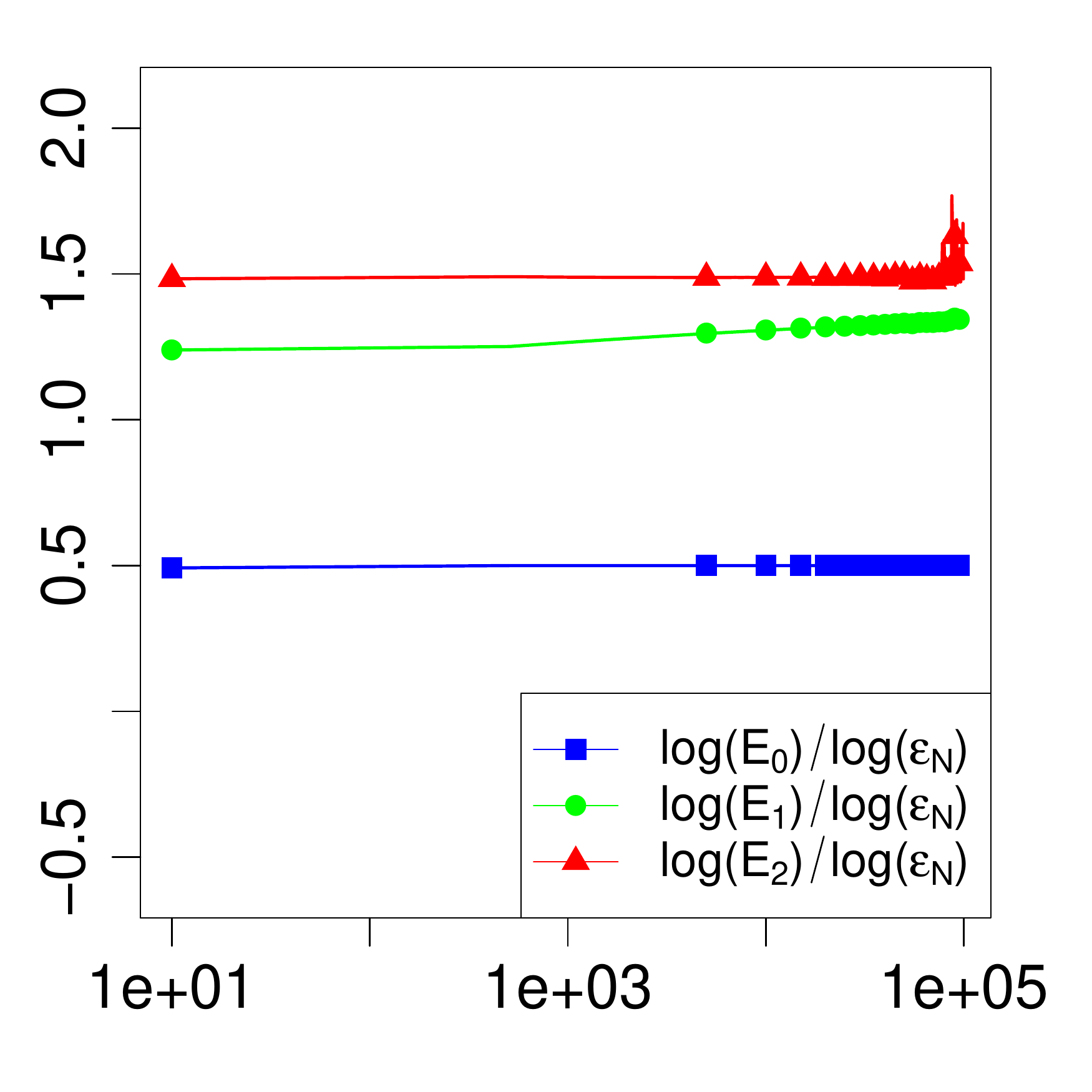}
            \vspace{-0.9cm}
            \caption{$\bb{\alpha} = (2,4)$ and $\beta = 2$}
        \end{subfigure}
        \begin{subfigure}[b]{0.22\textwidth}
            \centering
            \includegraphics[width=\textwidth, height=0.85\textwidth]{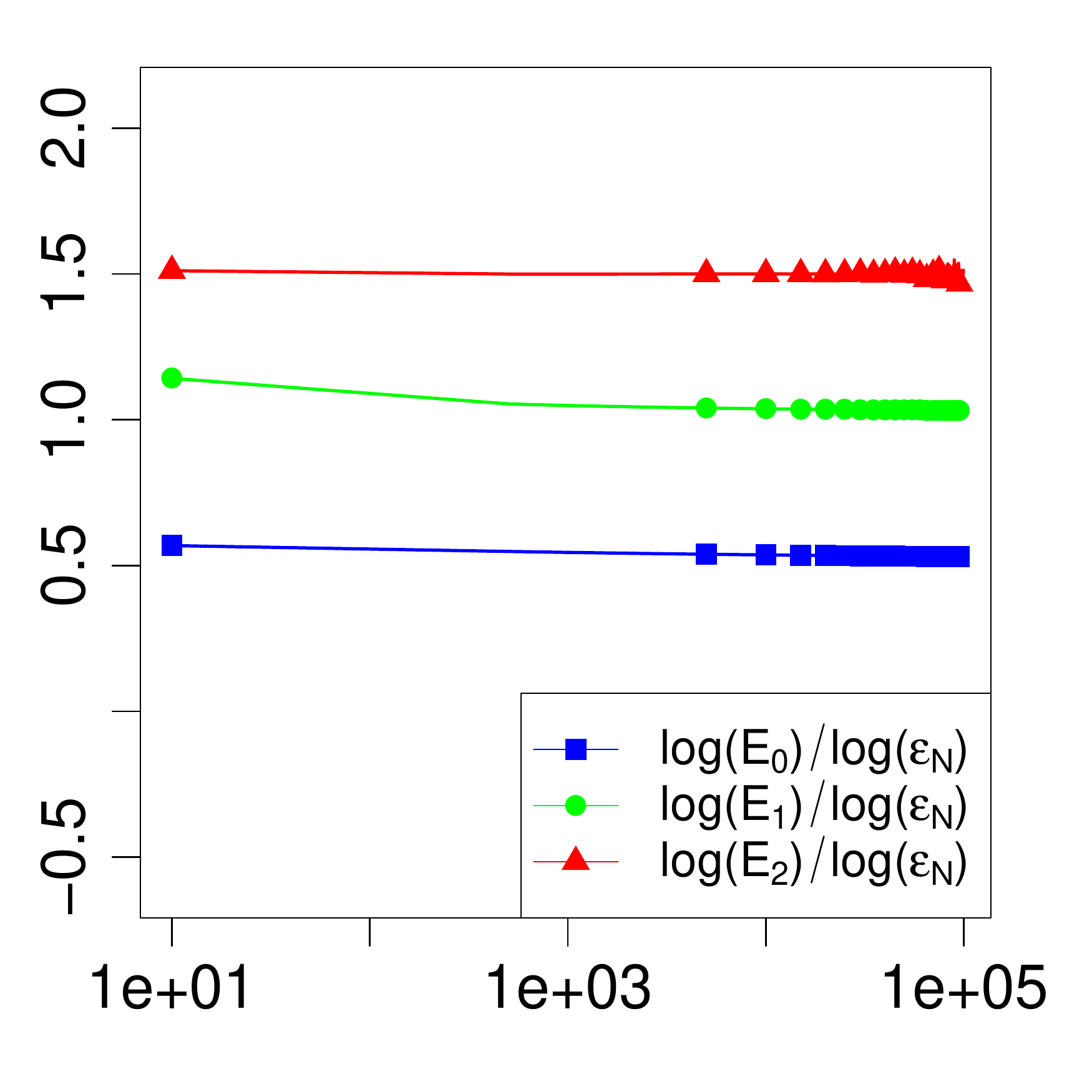}
            \vspace{-0.9cm}
            \caption{$\bb{\alpha} = (3,1)$ and $\beta = 2$}
        \end{subfigure}
        \quad
        \begin{subfigure}[b]{0.22\textwidth}
            \centering
            \includegraphics[width=\textwidth, height=0.85\textwidth]{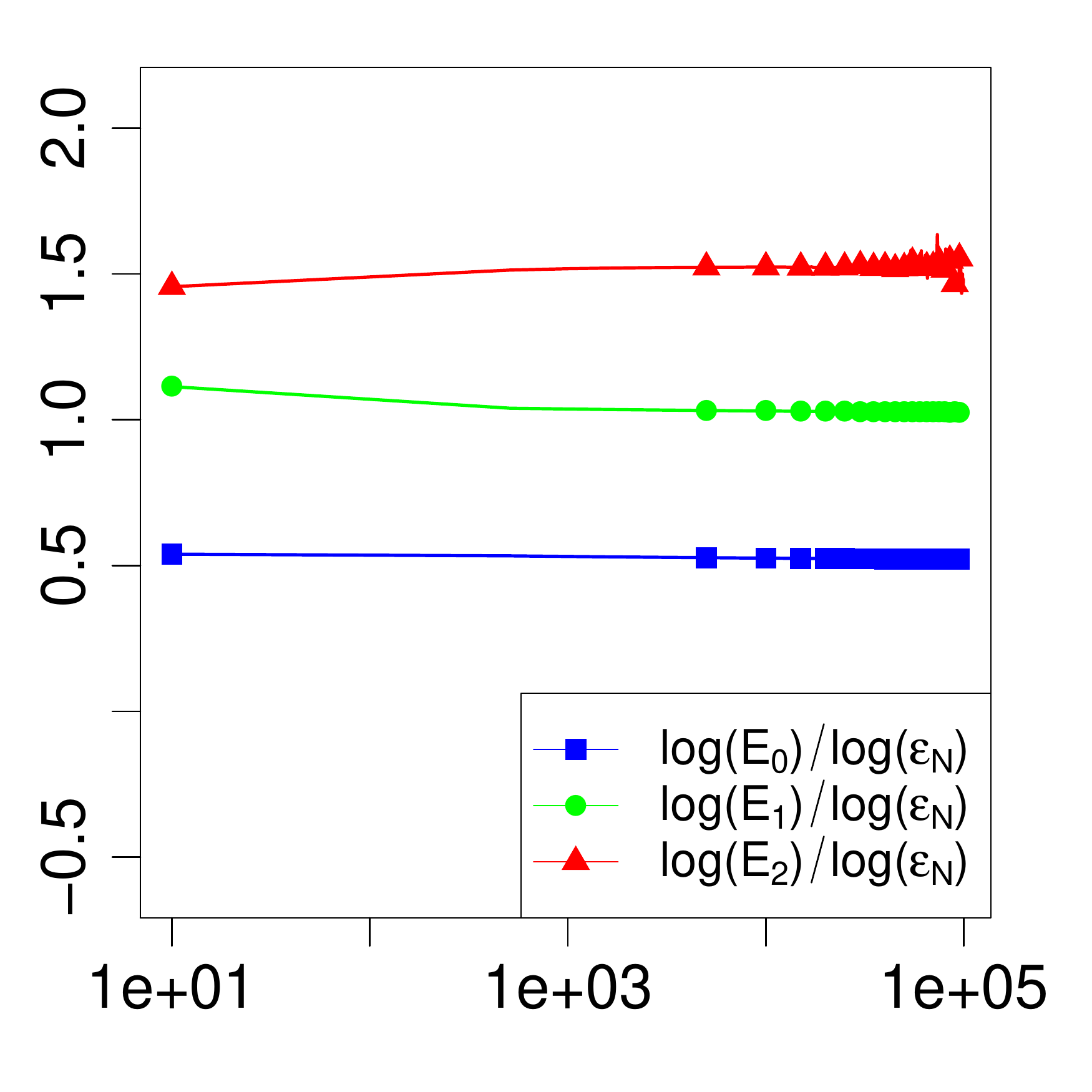}
            \vspace{-0.9cm}
            \caption{$\bb{\alpha} = (3,2)$ and $\beta = 2$}
        \end{subfigure}
        \quad
        \begin{subfigure}[b]{0.22\textwidth}
            \centering
            \includegraphics[width=\textwidth, height=0.85\textwidth]{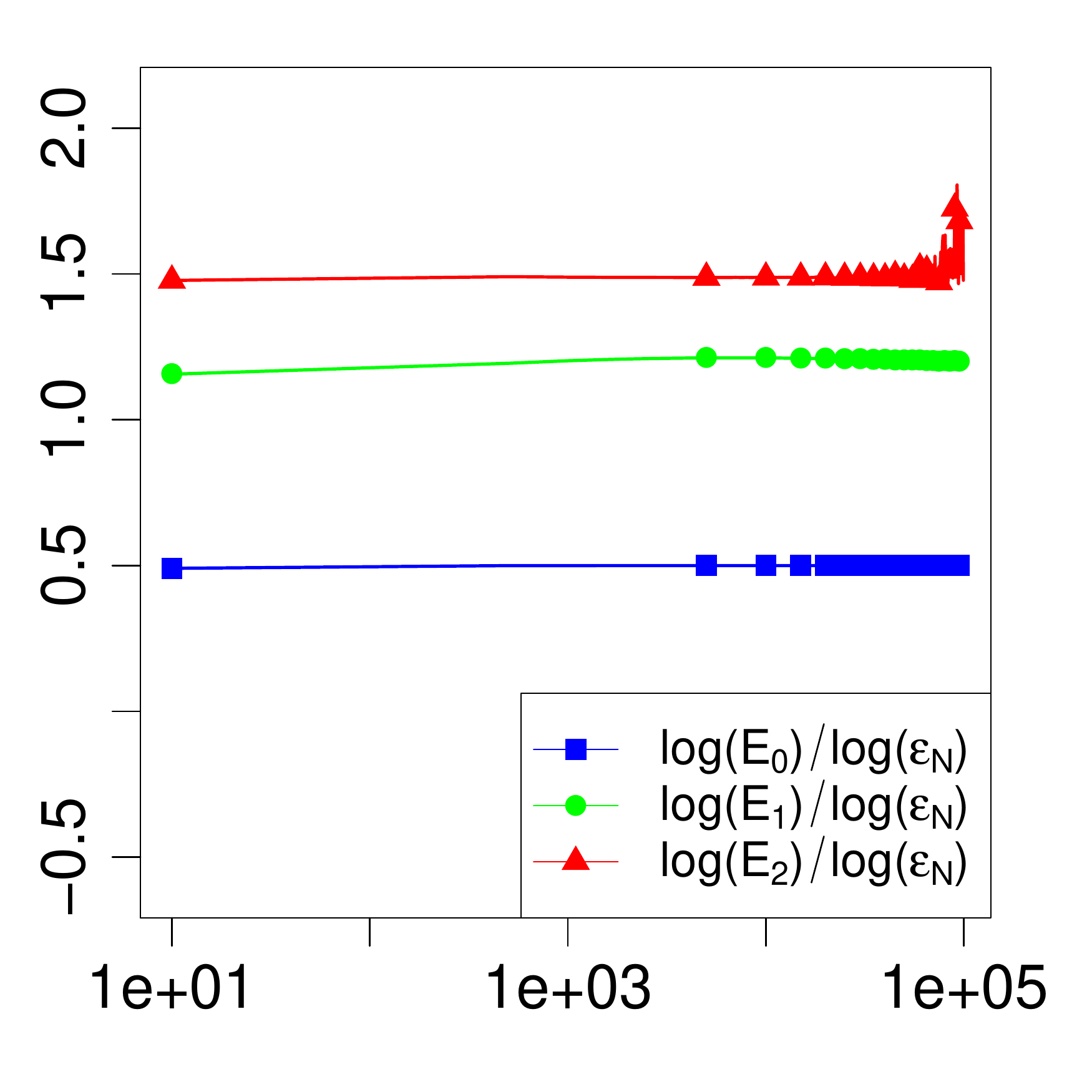}
            \vspace{-0.9cm}
            \caption{$\bb{\alpha} = (3,3)$ and $\beta = 2$}
        \end{subfigure}
        \quad
        \begin{subfigure}[b]{0.22\textwidth}
            \centering
            \includegraphics[width=\textwidth, height=0.85\textwidth]{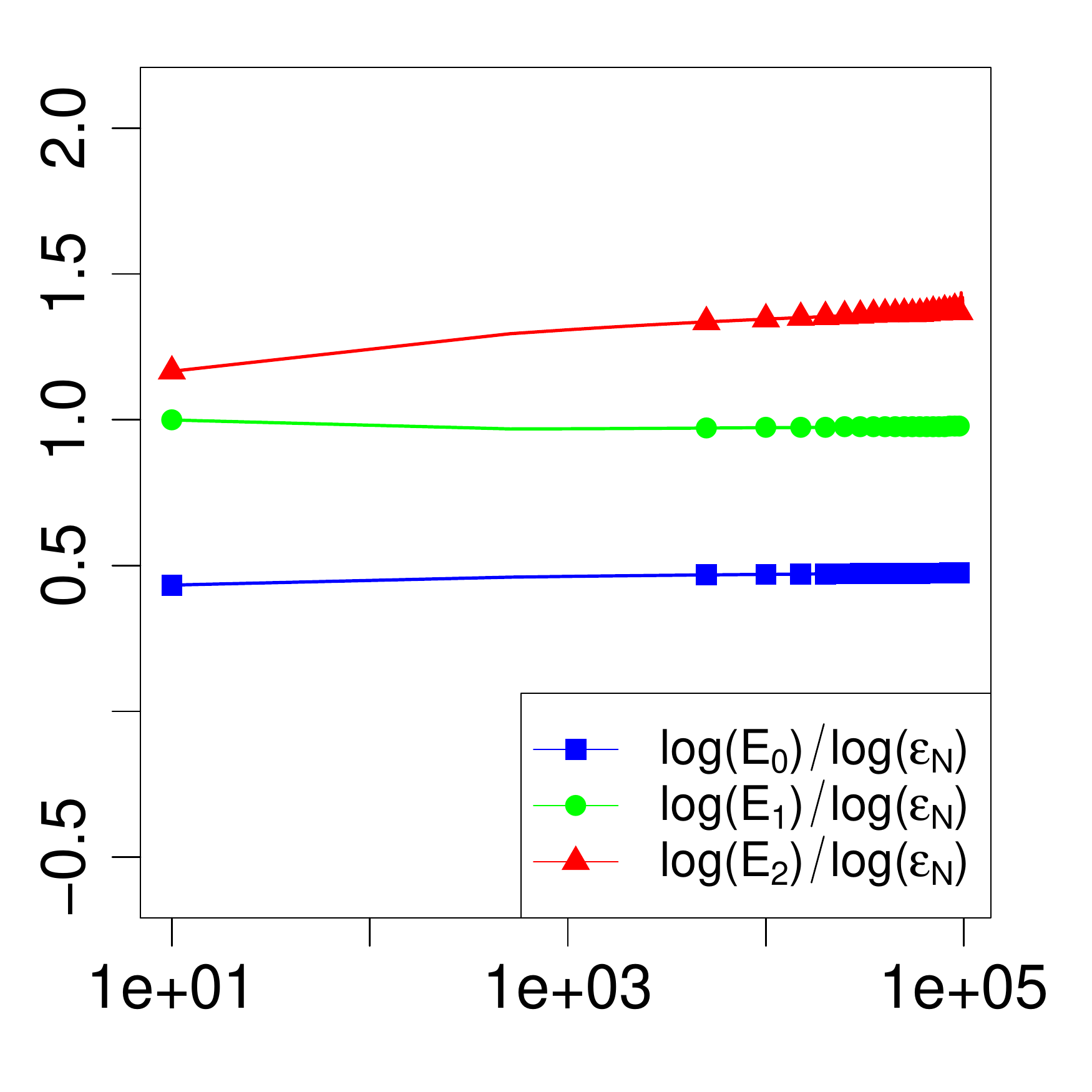}
            \vspace{-0.9cm}
            \caption{$\bb{\alpha} = (3,4)$ and $\beta = 2$}
        \end{subfigure}
        \caption{Plots of $\log E_i / \log \e_N$ as a function of $N$, for various choices of $\bb{\alpha}$, when $\beta = 2$. The horizontal axis is on a logarithmic scale. The plots confirm \eqref{eq:liminf.exponent.bound} and bring strong evidence for the validity of Theorem~\ref{thm:p.k.expansion}.}
        \label{fig:error.exponents.plots.beta.2}
    \end{figure}

    \phantom{vertical spacing}
    \begin{figure}[H]
        \captionsetup[subfigure]{labelformat=empty}
        \captionsetup{width=0.8\linewidth}
        \vspace{-0.5cm}
        \centering
        \begin{subfigure}[b]{0.22\textwidth}
            \centering
            \includegraphics[width=\textwidth, height=0.85\textwidth]{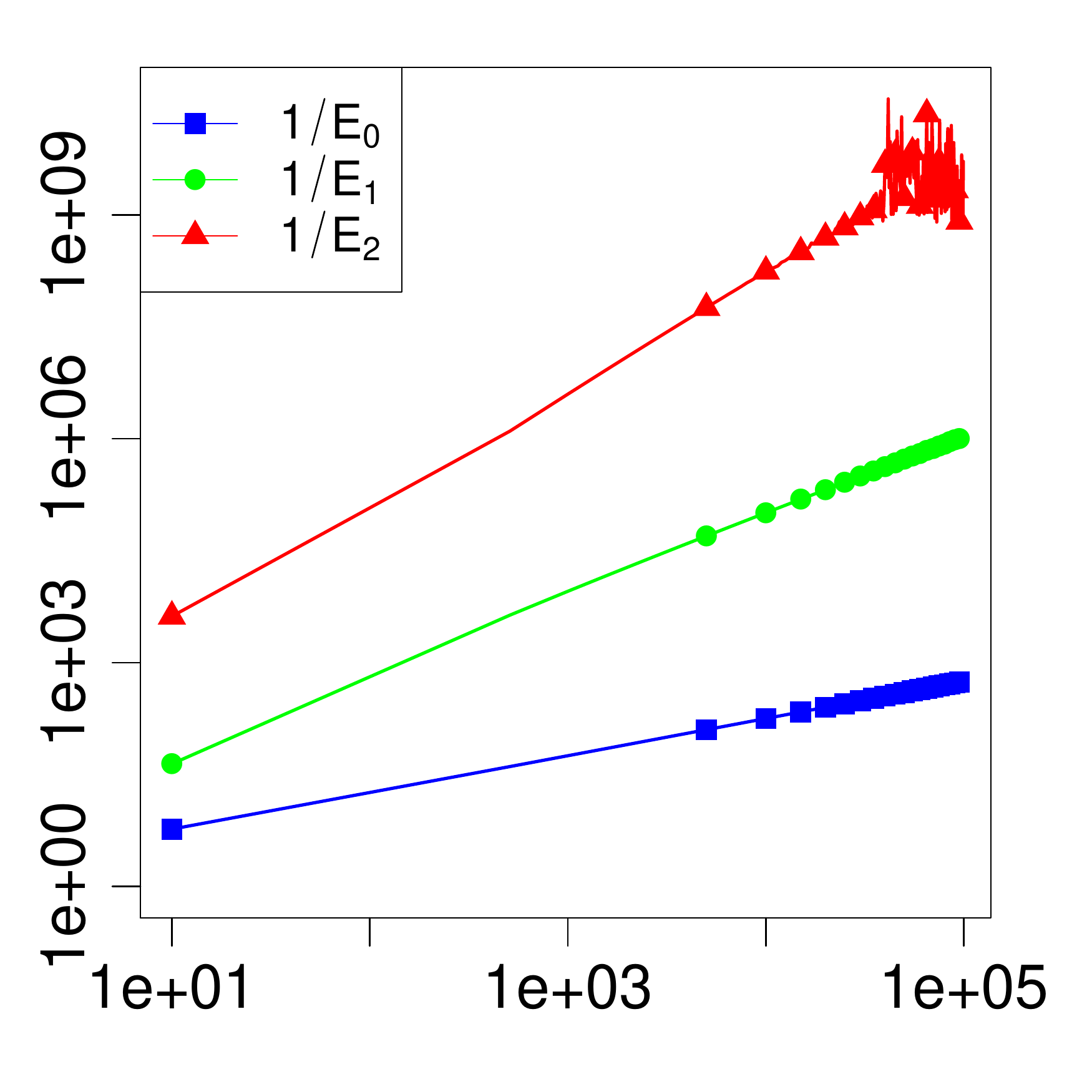}
            \vspace{-0.9cm}
            \caption{$\bb{\alpha} = (1,1)$ and $\beta = 3$}
        \end{subfigure}
        \quad
        \begin{subfigure}[b]{0.22\textwidth}
            \centering
            \includegraphics[width=\textwidth, height=0.85\textwidth]{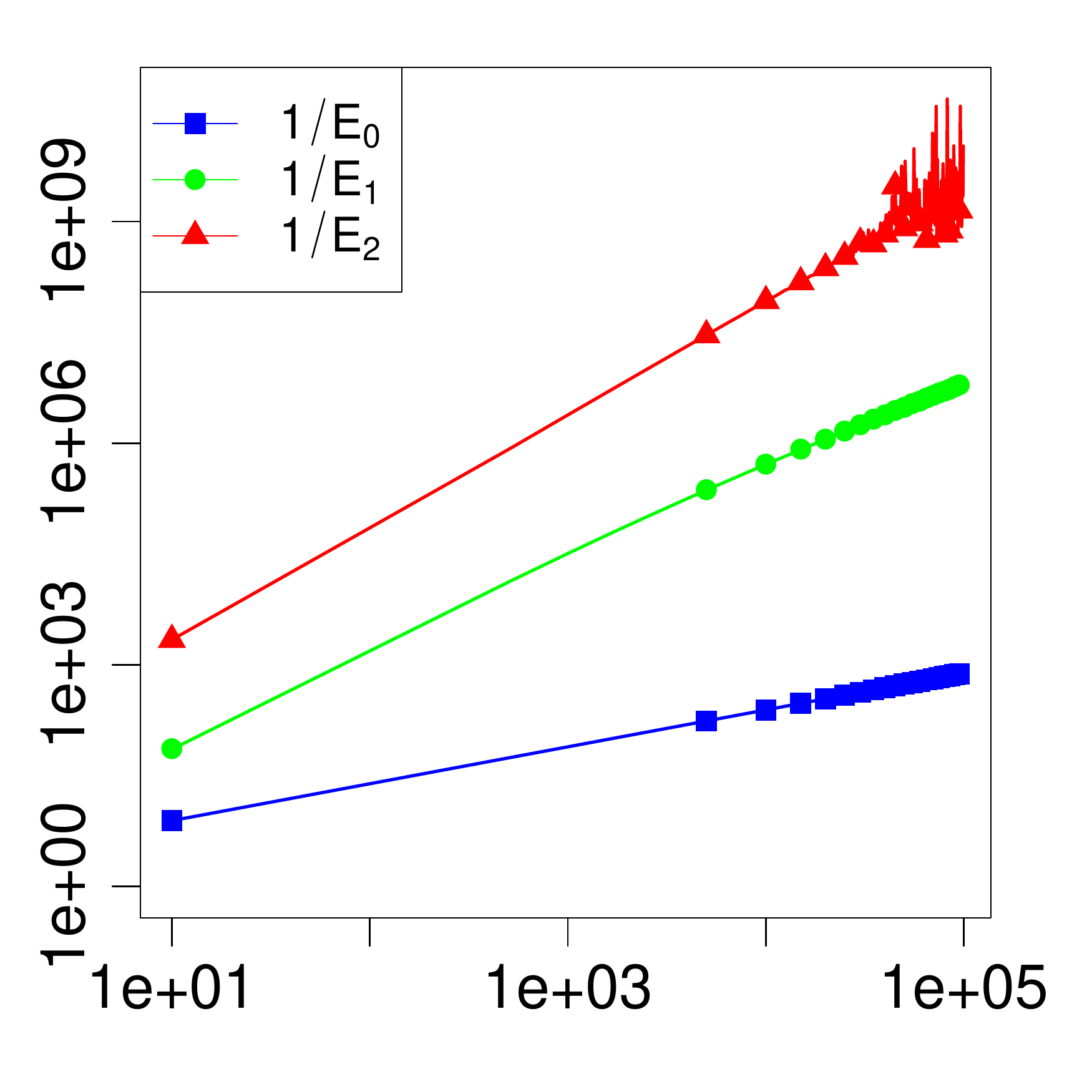}
            \vspace{-0.9cm}
            \caption{$\bb{\alpha} = (1,2)$ and $\beta = 3$}
        \end{subfigure}
        \quad
        \begin{subfigure}[b]{0.22\textwidth}
            \centering
            \includegraphics[width=\textwidth, height=0.85\textwidth]{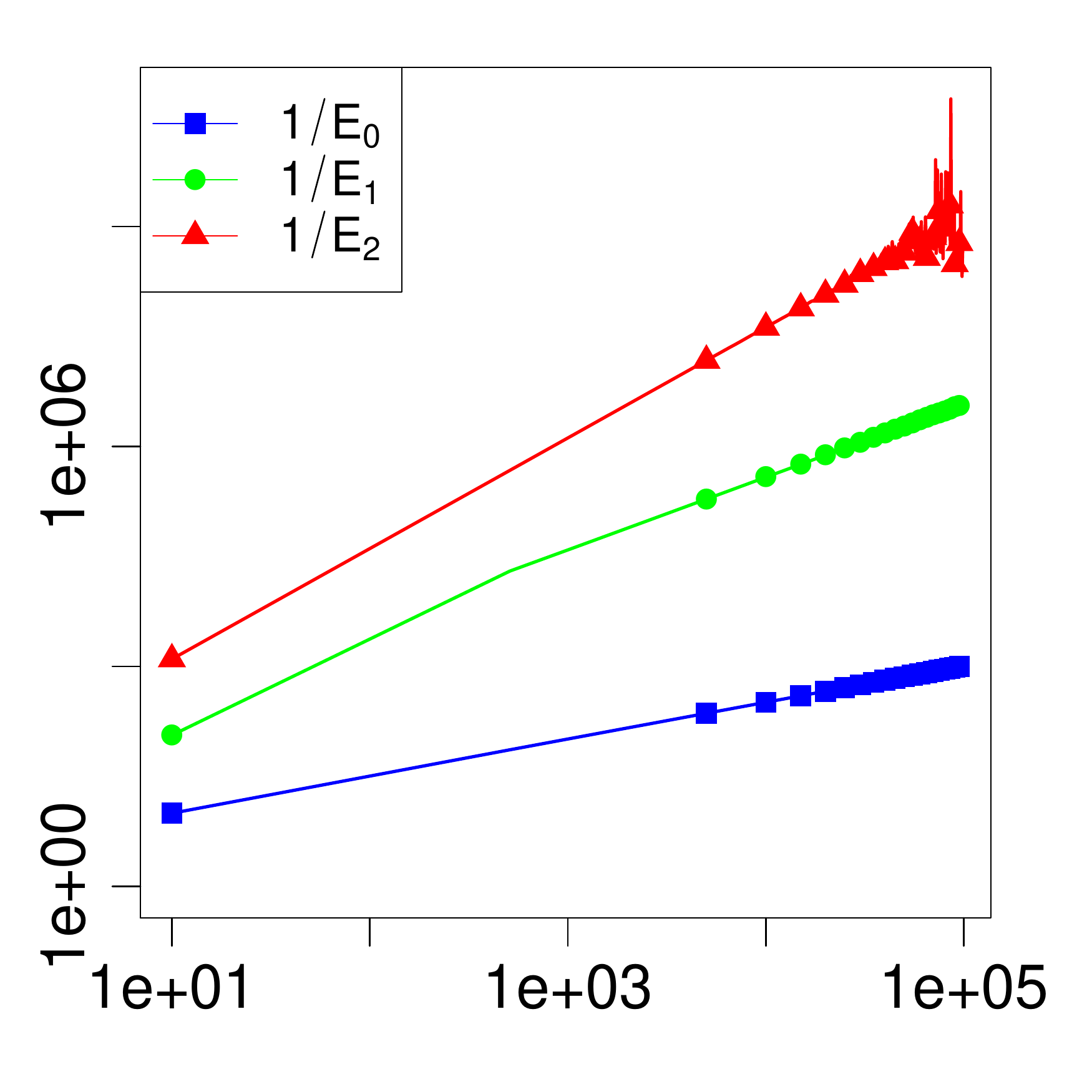}
            \vspace{-0.9cm}
            \caption{$\bb{\alpha} = (1,3)$ and $\beta = 3$}
        \end{subfigure}
        \quad
        \begin{subfigure}[b]{0.22\textwidth}
            \centering
            \includegraphics[width=\textwidth, height=0.85\textwidth]{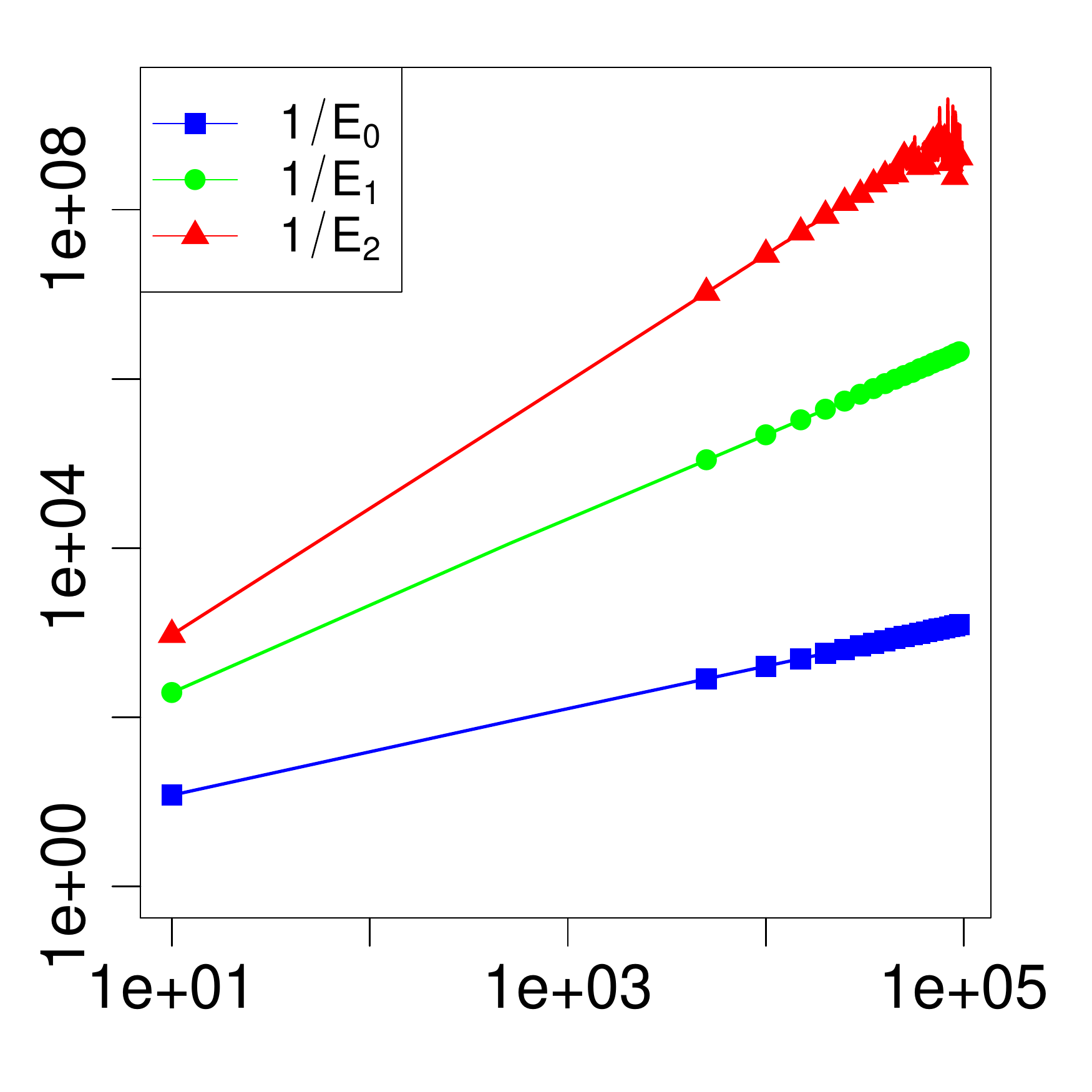}
            \vspace{-0.9cm}
            \caption{$\bb{\alpha} = (1,4)$ and $\beta = 3$}
        \end{subfigure}
        \begin{subfigure}[b]{0.22\textwidth}
            \centering
            \includegraphics[width=\textwidth, height=0.85\textwidth]{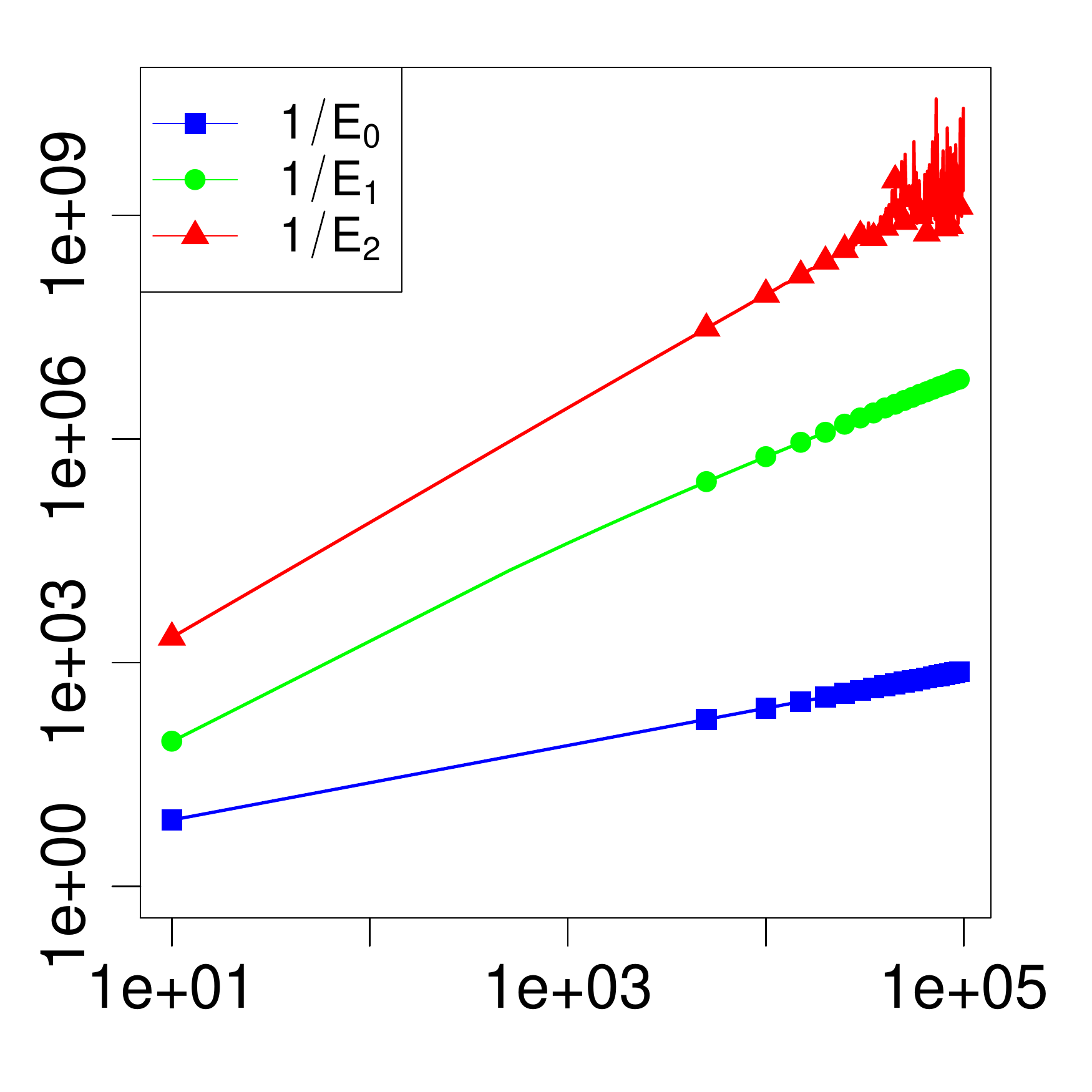}
            \vspace{-0.9cm}
            \caption{$\bb{\alpha} = (2,1)$ and $\beta = 3$}
        \end{subfigure}
        \quad
        \begin{subfigure}[b]{0.22\textwidth}
            \centering
            \includegraphics[width=\textwidth, height=0.85\textwidth]{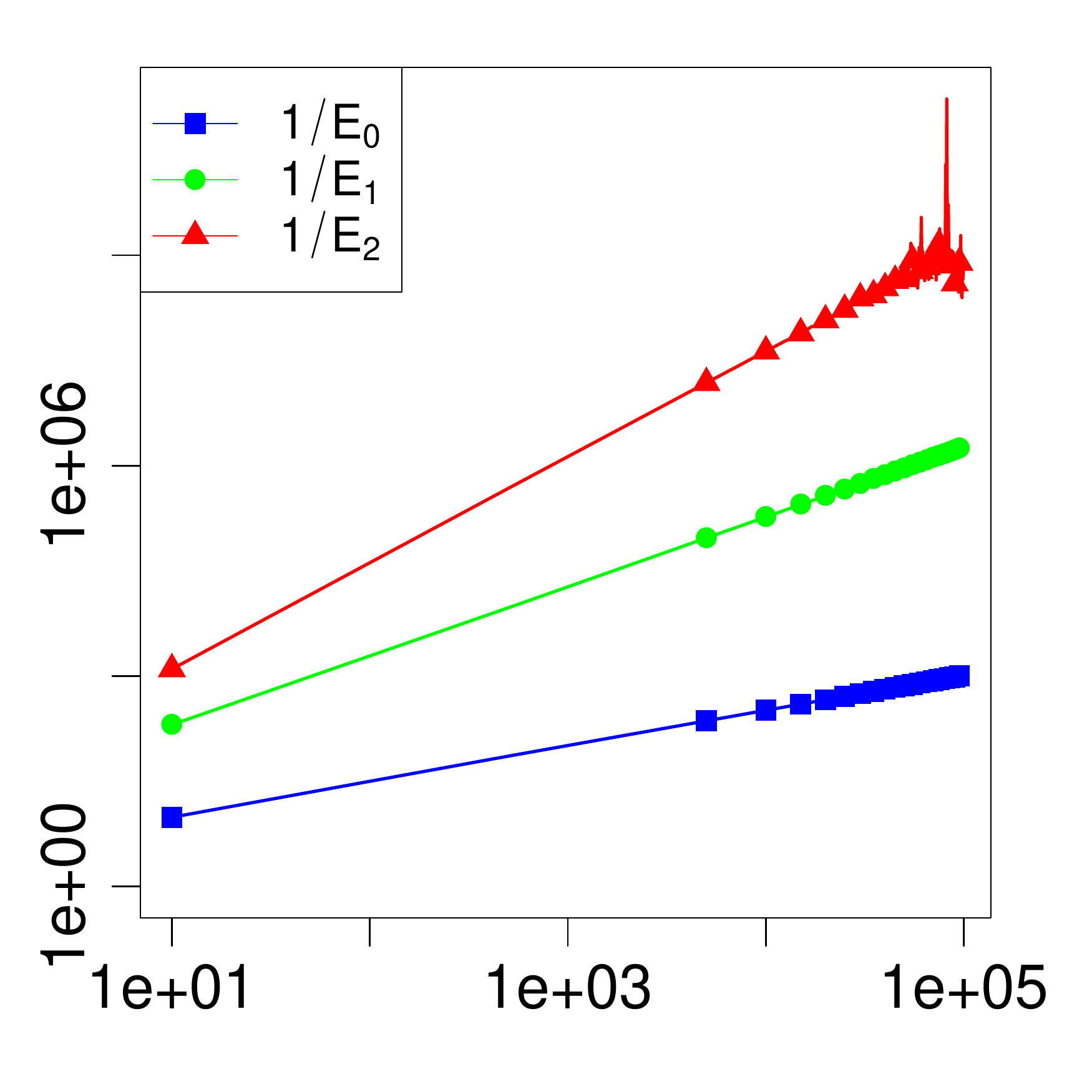}
            \vspace{-0.9cm}
            \caption{$\bb{\alpha} = (2,2)$ and $\beta = 3$}
        \end{subfigure}
        \quad
        \begin{subfigure}[b]{0.22\textwidth}
            \centering
            \includegraphics[width=\textwidth, height=0.85\textwidth]{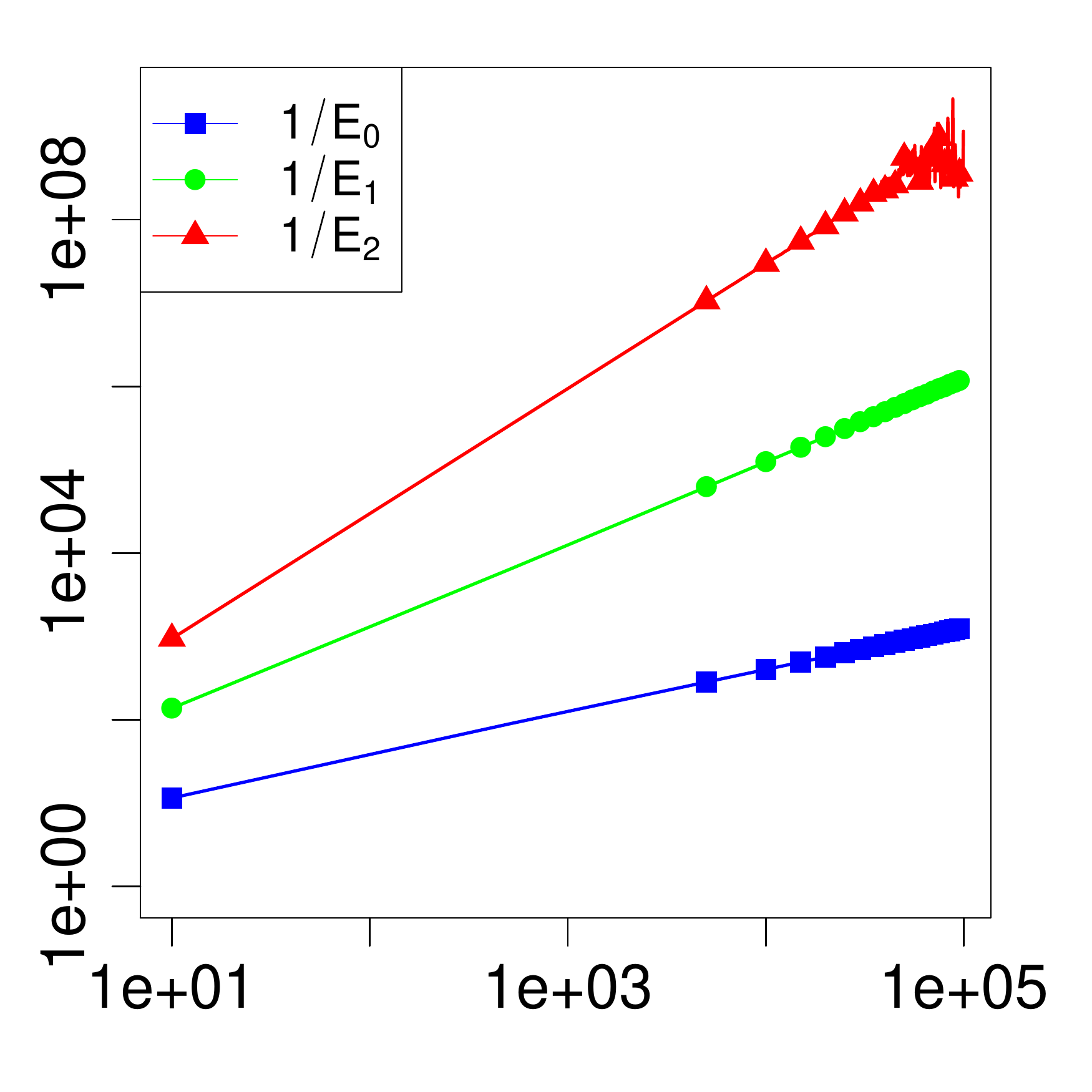}
            \vspace{-0.9cm}
            \caption{$\bb{\alpha} = (2,3)$ and $\beta = 3$}
        \end{subfigure}
        \quad
        \begin{subfigure}[b]{0.22\textwidth}
            \centering
            \includegraphics[width=\textwidth, height=0.85\textwidth]{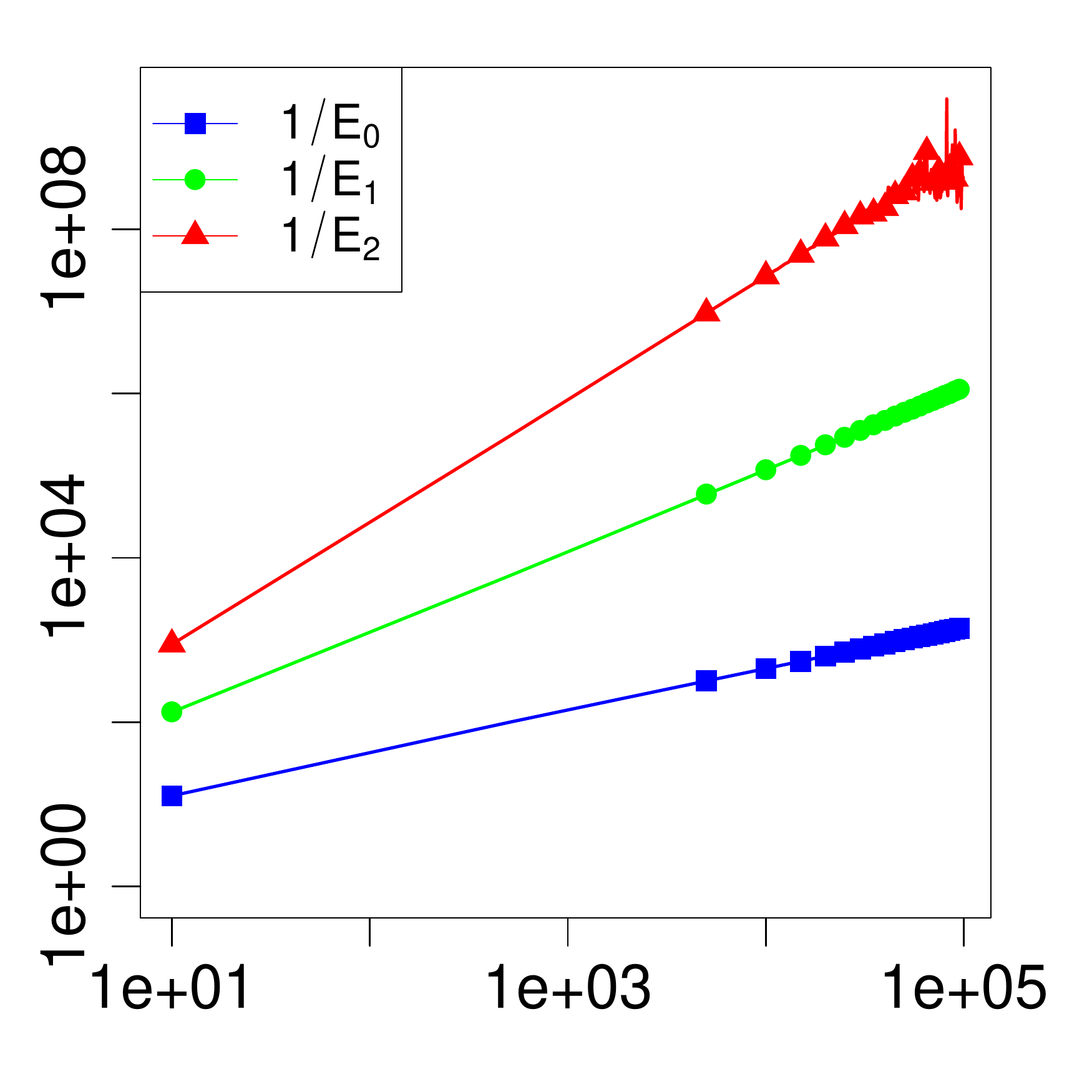}
            \vspace{-0.9cm}
            \caption{$\bb{\alpha} = (2,4)$ and $\beta = 3$}
        \end{subfigure}
        \begin{subfigure}[b]{0.22\textwidth}
            \centering
            \includegraphics[width=\textwidth, height=0.85\textwidth]{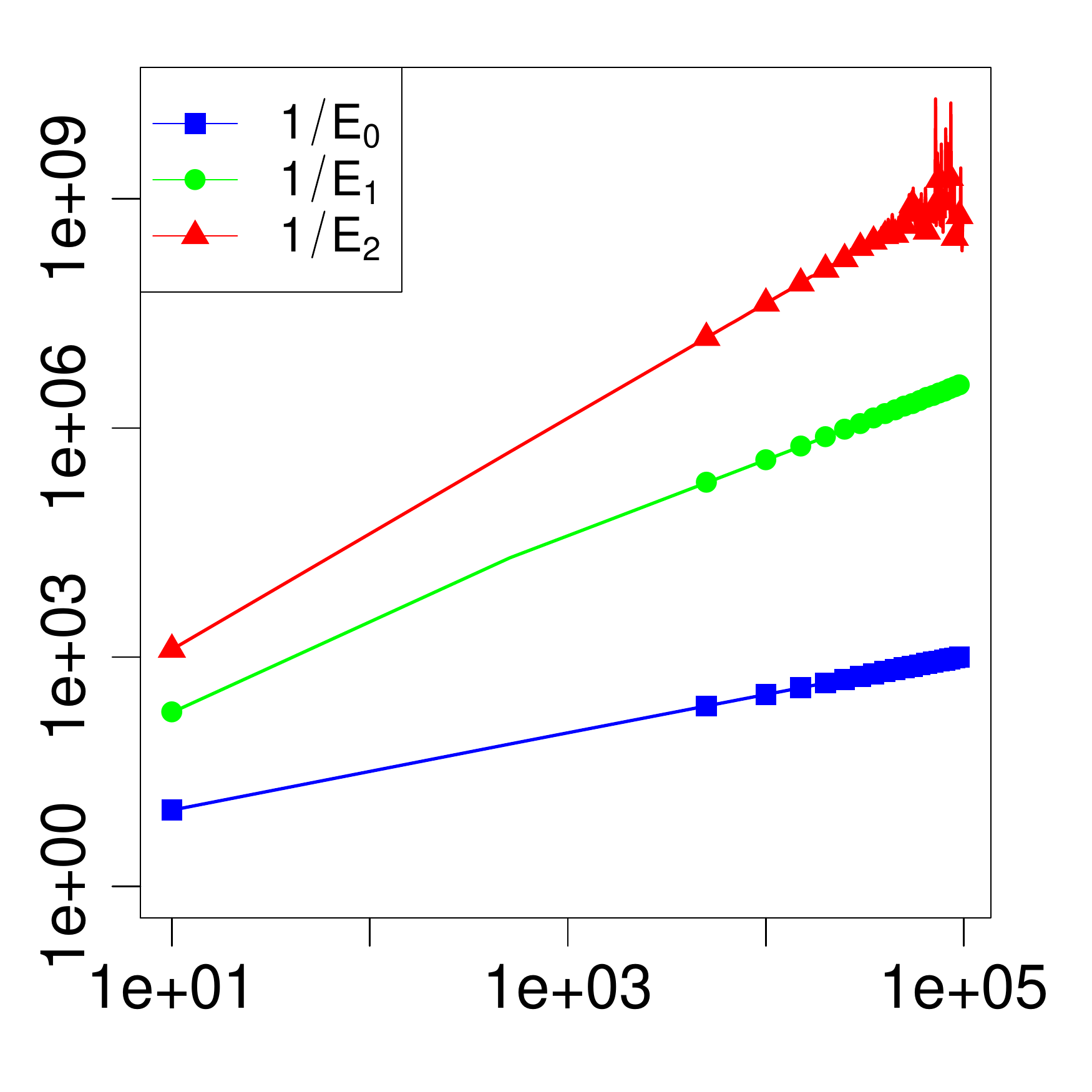}
            \vspace{-0.9cm}
            \caption{$\bb{\alpha} = (3,1)$ and $\beta = 3$}
        \end{subfigure}
        \quad
        \begin{subfigure}[b]{0.22\textwidth}
            \centering
            \includegraphics[width=\textwidth, height=0.85\textwidth]{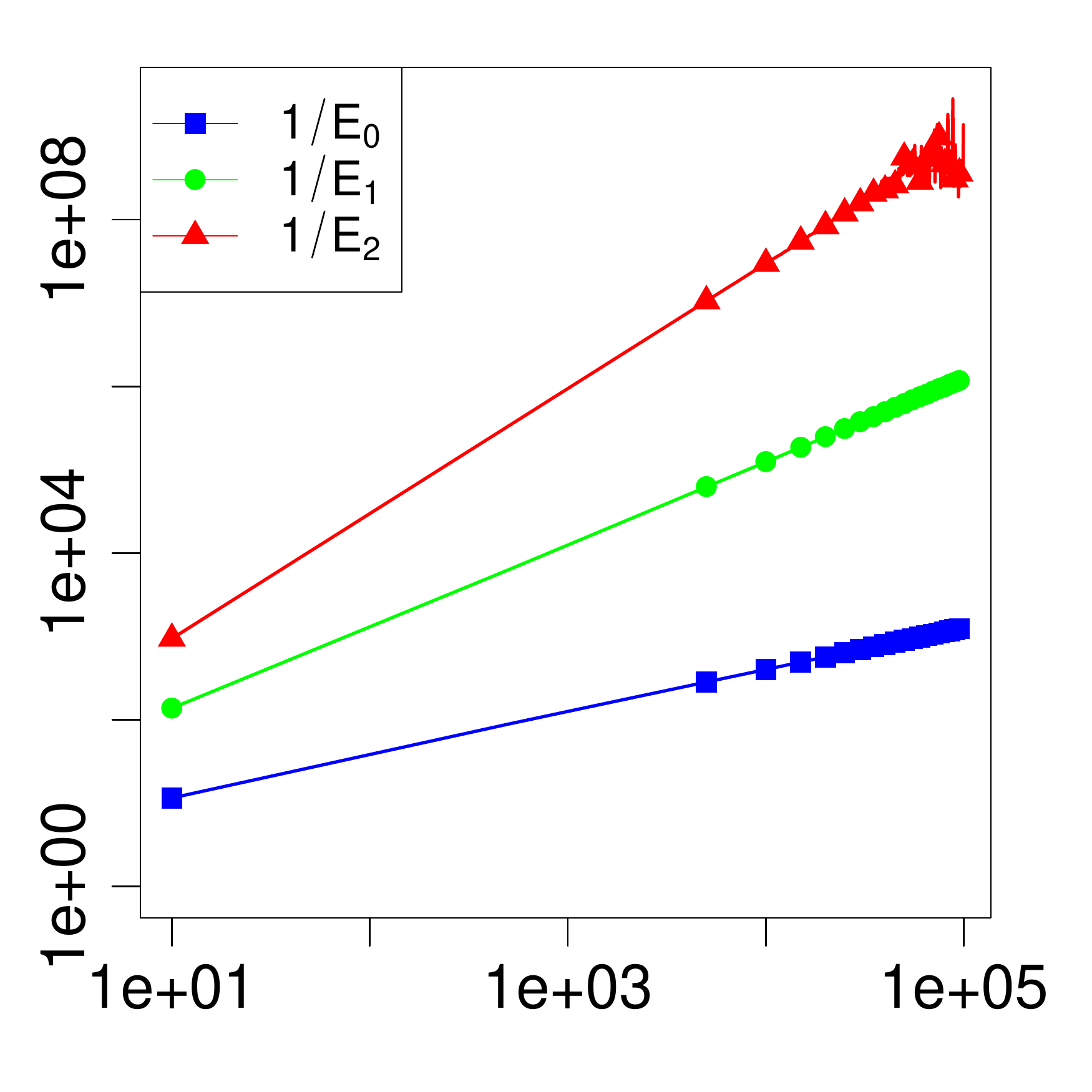}
            \vspace{-0.9cm}
            \caption{$\bb{\alpha} = (3,2)$ and $\beta = 3$}
        \end{subfigure}
        \quad
        \begin{subfigure}[b]{0.22\textwidth}
            \centering
            \includegraphics[width=\textwidth, height=0.85\textwidth]{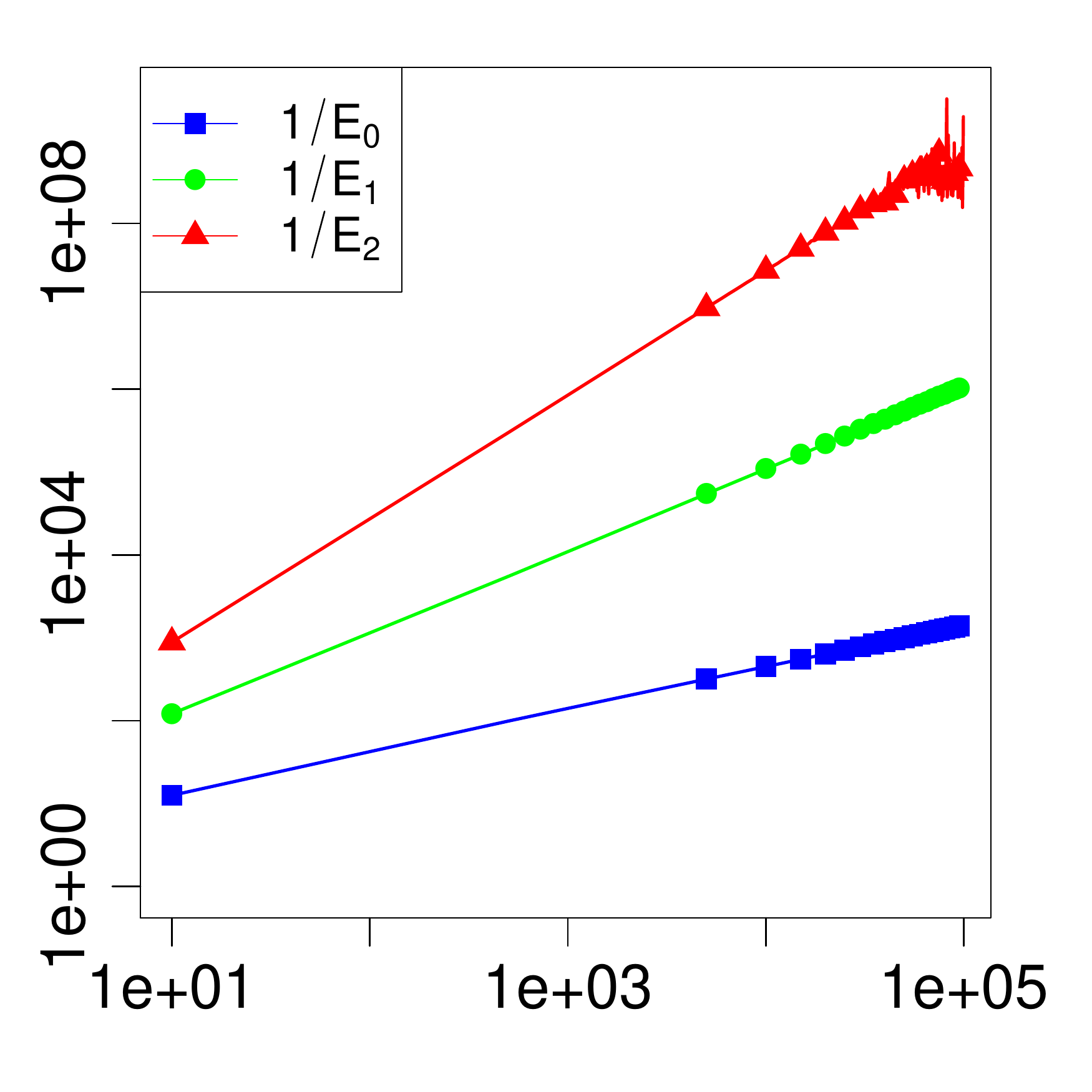}
            \vspace{-0.9cm}
            \caption{$\bb{\alpha} = (3,3)$ and $\beta = 3$}
        \end{subfigure}
        \quad
        \begin{subfigure}[b]{0.22\textwidth}
            \centering
            \includegraphics[width=\textwidth, height=0.85\textwidth]{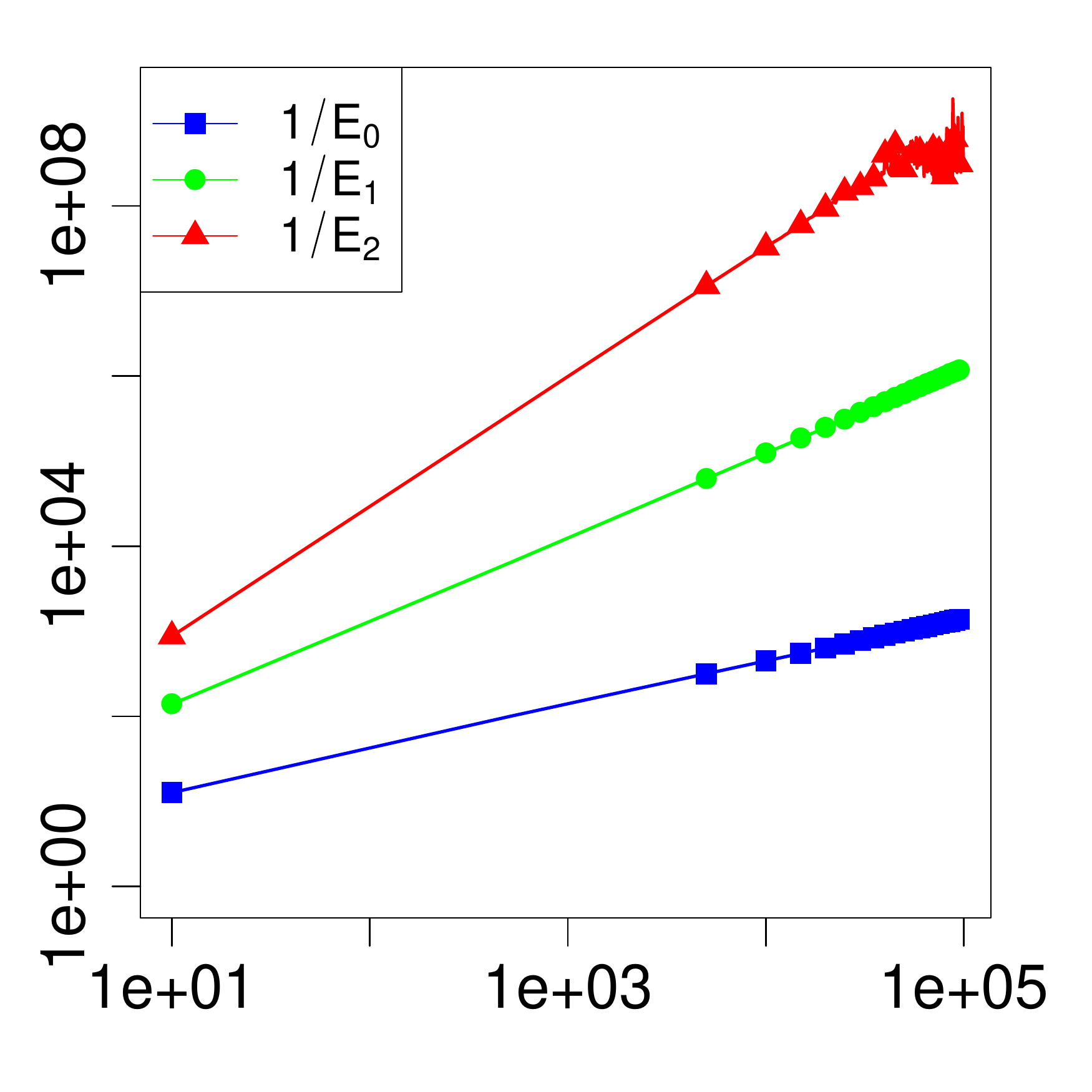}
            \vspace{-0.9cm}
            \caption{$\bb{\alpha} = (3,4)$ and $\beta = 3$}
        \end{subfigure}
        \caption{Plots of $1 / E_i$ as a function of $N$, for various choices of $\bb{\alpha}$, when $\beta = 3$. Both the horizontal and vertical axes are on a logarithmic scale. The plots clearly illustrate how the addition of correction terms from Theorem~\ref{thm:p.k.expansion} to the base approximation \eqref{eq:E.0} improves it.}
        \label{fig:loglog.errors.plots.beta.3}
    \end{figure}
    \begin{figure}[H]
        \captionsetup[subfigure]{labelformat=empty}
        \captionsetup{width=0.8\linewidth}
        \vspace{-0.5cm}
        \centering
        \begin{subfigure}[b]{0.22\textwidth}
            \centering
            \includegraphics[width=\textwidth, height=0.85\textwidth]{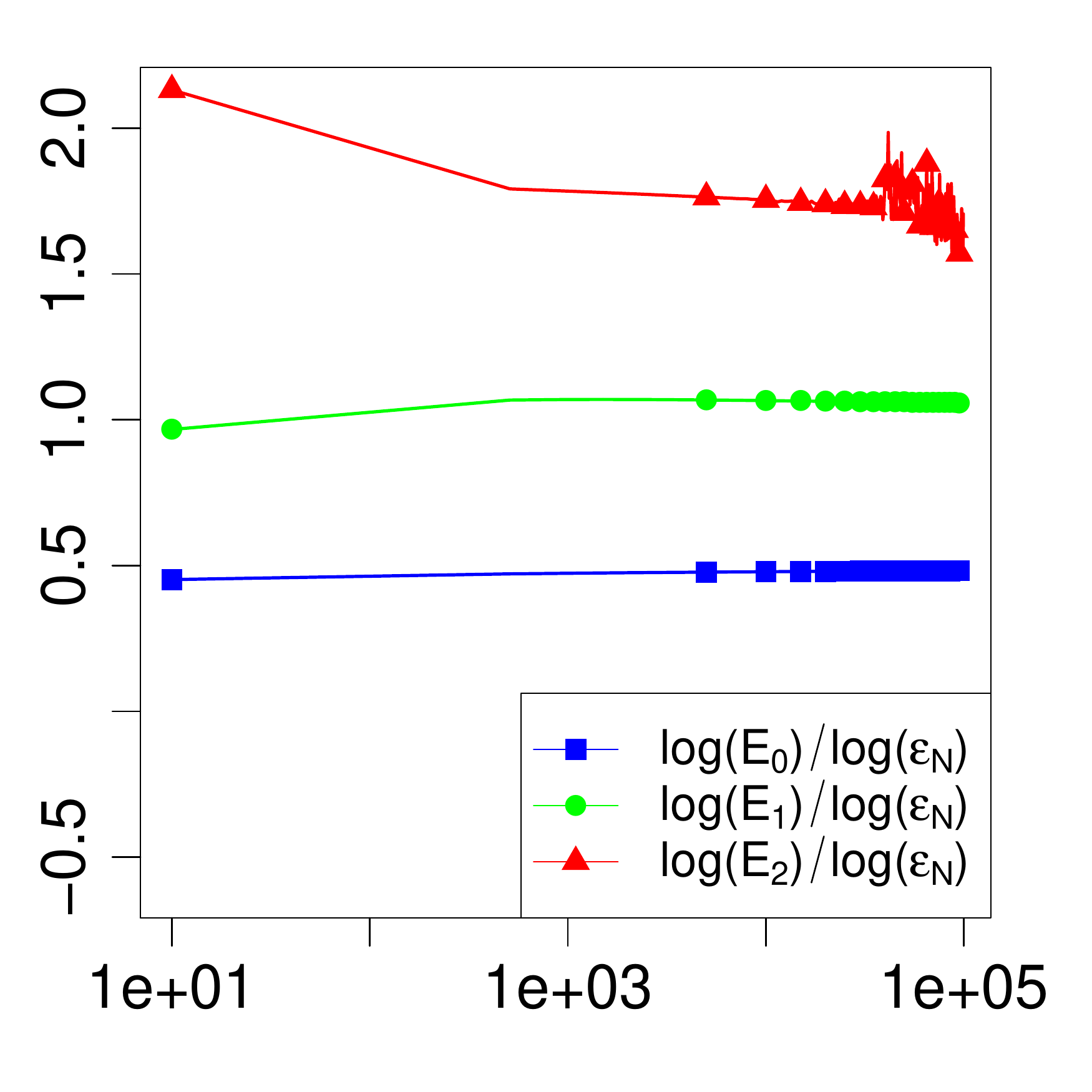}
            \vspace{-0.9cm}
            \caption{$\bb{\alpha} = (1,1)$ and $\beta = 3$}
        \end{subfigure}
        \quad
        \begin{subfigure}[b]{0.22\textwidth}
            \centering
            \includegraphics[width=\textwidth, height=0.85\textwidth]{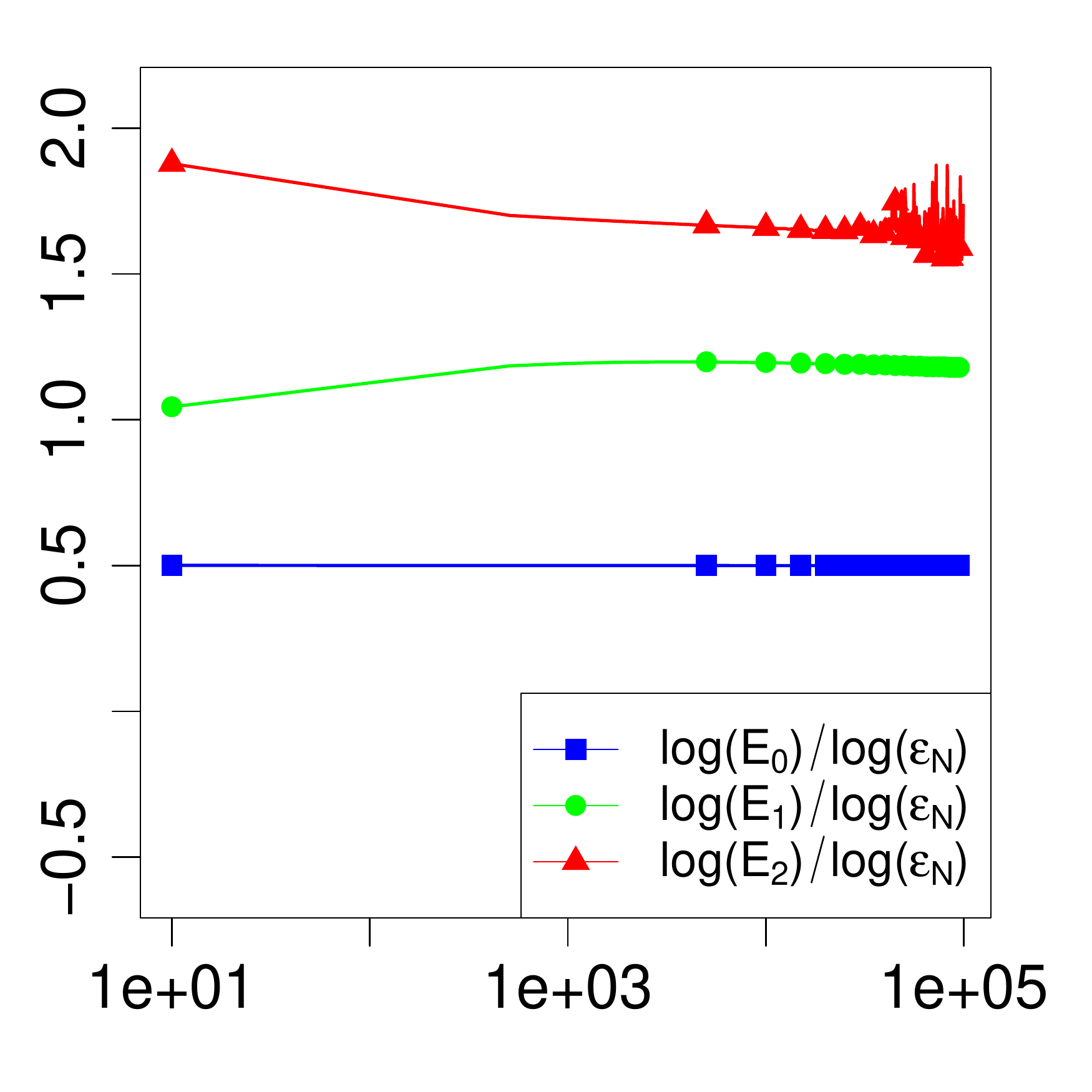}
            \vspace{-0.9cm}
            \caption{$\bb{\alpha} = (1,2)$ and $\beta = 3$}
        \end{subfigure}
        \quad
        \begin{subfigure}[b]{0.22\textwidth}
            \centering
            \includegraphics[width=\textwidth, height=0.85\textwidth]{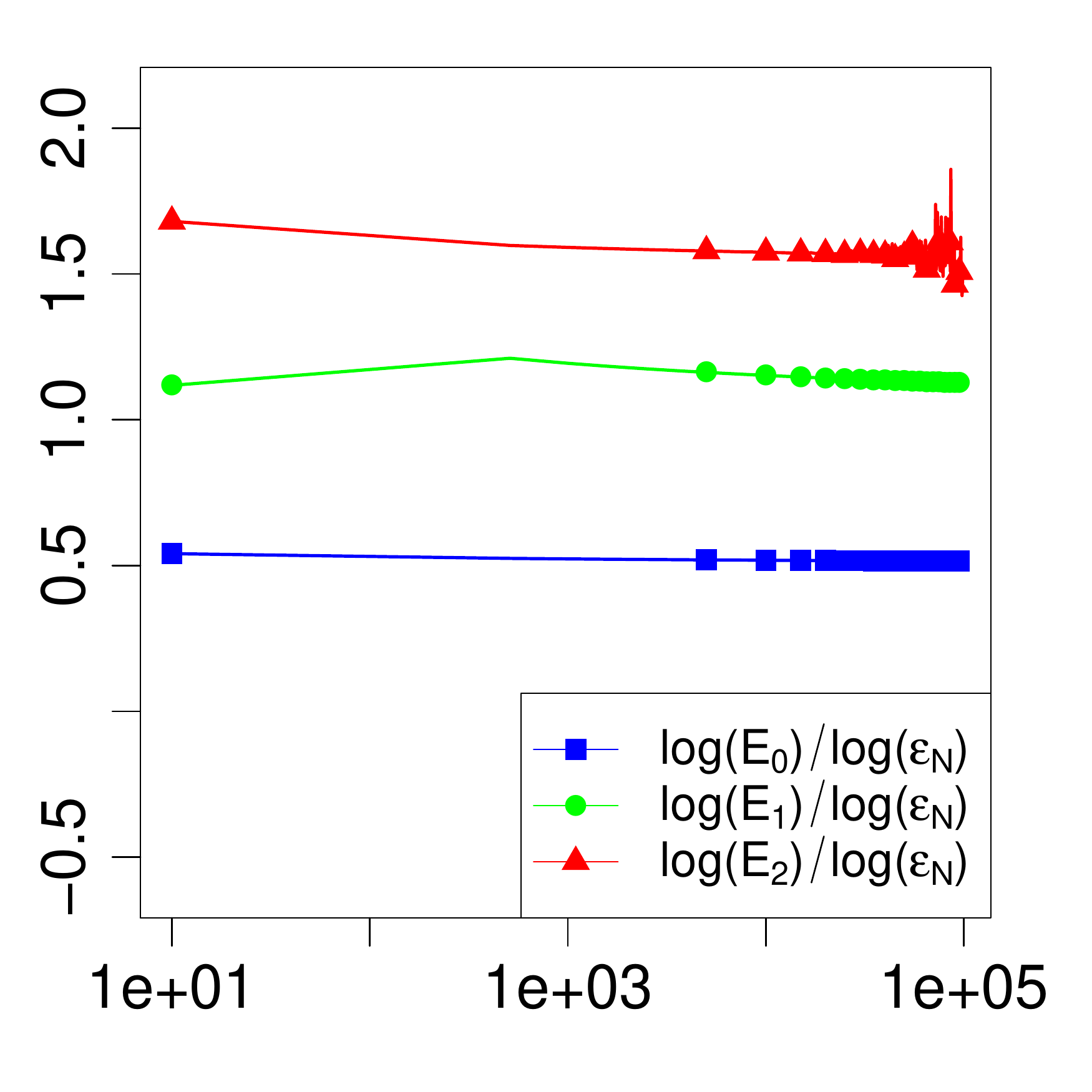}
            \vspace{-0.9cm}
            \caption{$\bb{\alpha} = (1,3)$ and $\beta = 3$}
        \end{subfigure}
        \quad
        \begin{subfigure}[b]{0.22\textwidth}
            \centering
            \includegraphics[width=\textwidth, height=0.85\textwidth]{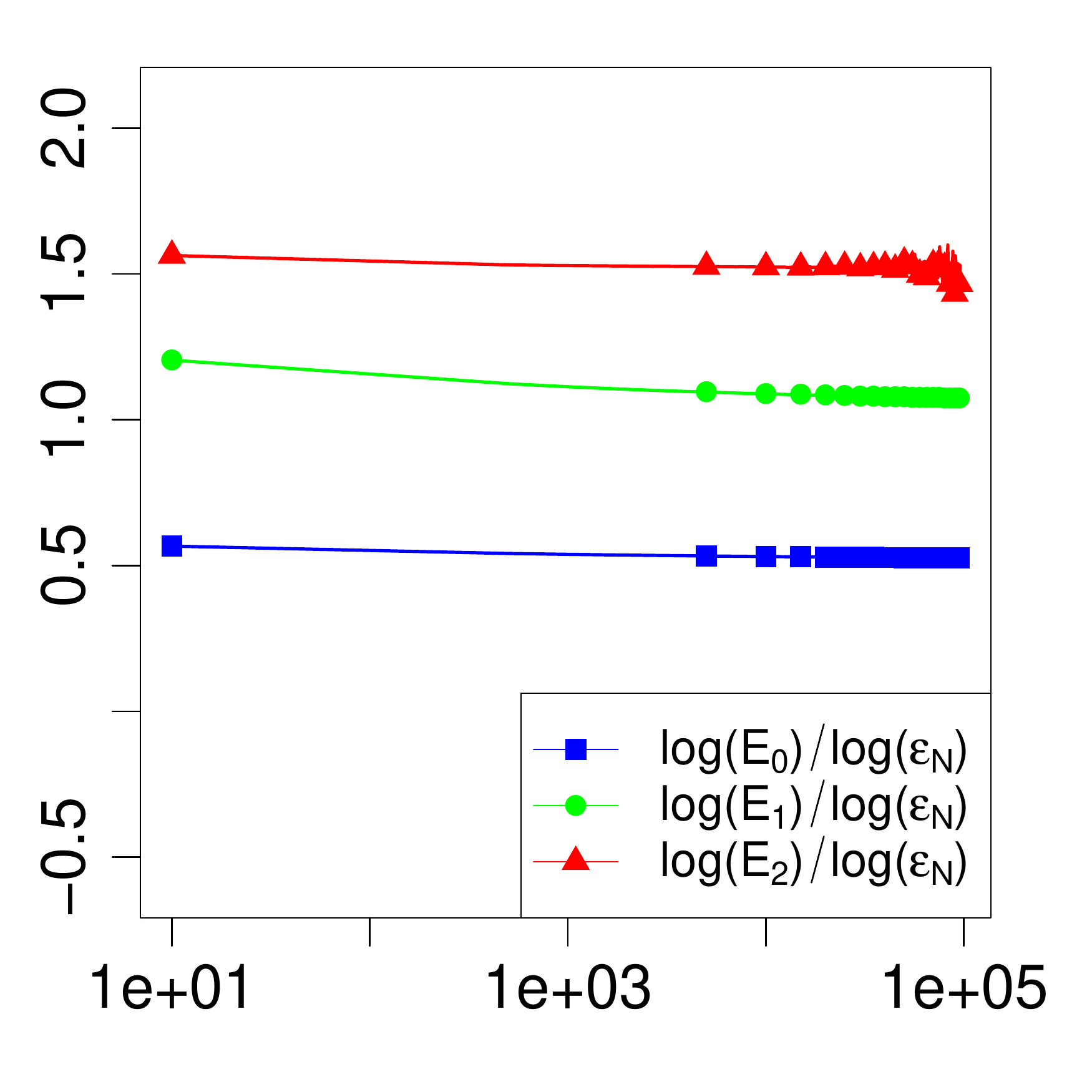}
            \vspace{-0.9cm}
            \caption{$\bb{\alpha} = (1,4)$ and $\beta = 3$}
        \end{subfigure}
        \begin{subfigure}[b]{0.22\textwidth}
            \centering
            \includegraphics[width=\textwidth, height=0.85\textwidth]{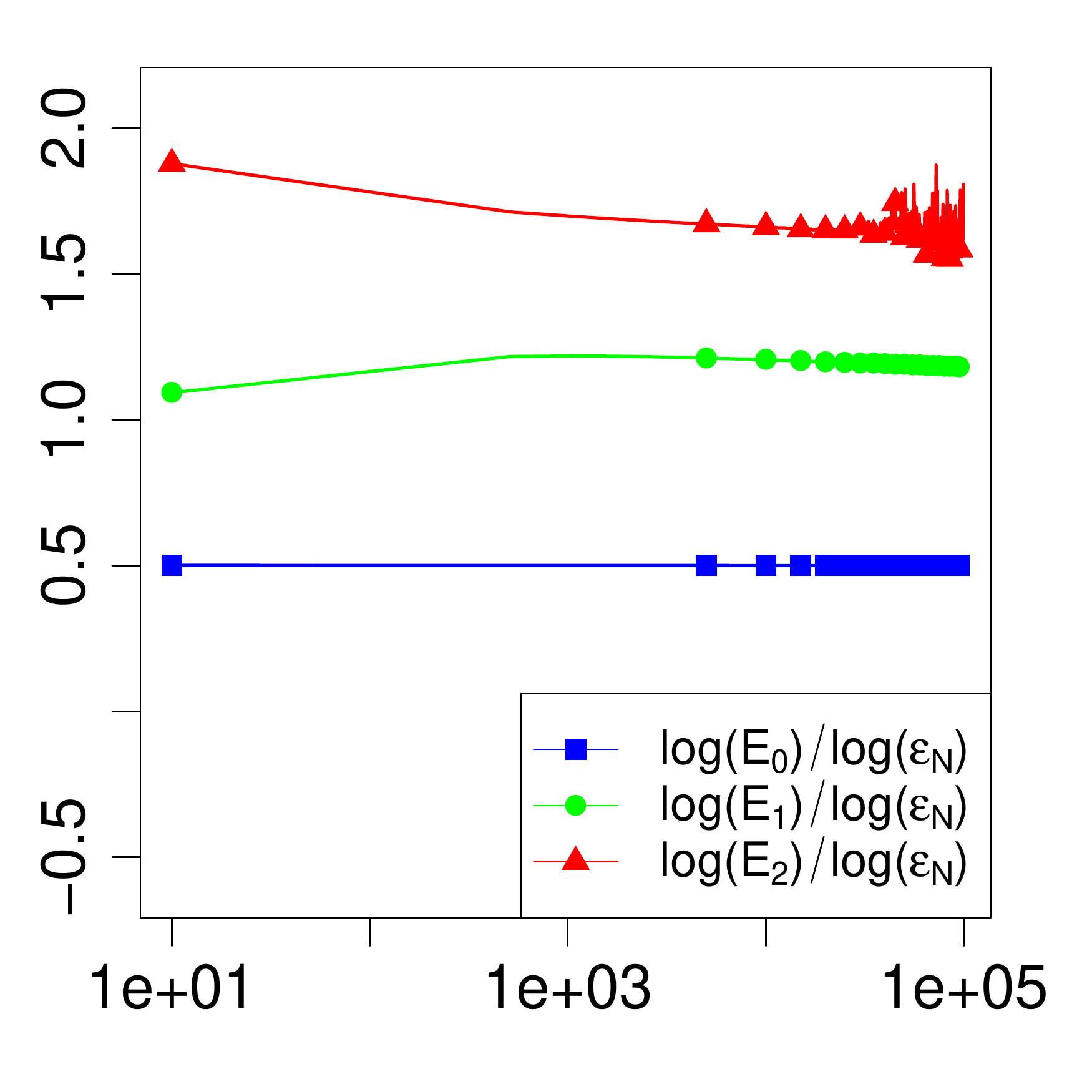}
            \vspace{-0.9cm}
            \caption{$\bb{\alpha} = (2,1)$ and $\beta = 3$}
        \end{subfigure}
        \quad
        \begin{subfigure}[b]{0.22\textwidth}
            \centering
            \includegraphics[width=\textwidth, height=0.85\textwidth]{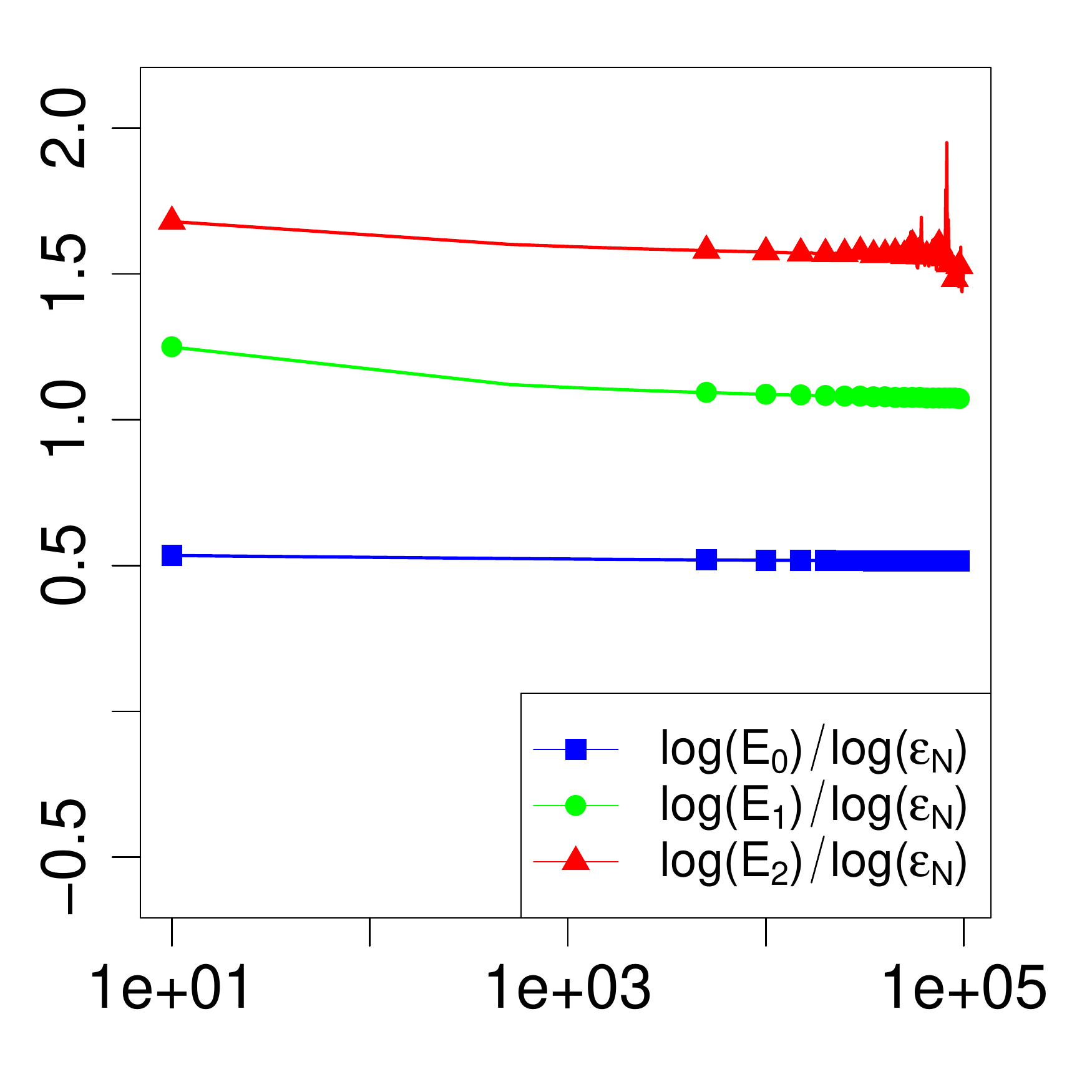}
            \vspace{-0.9cm}
            \caption{$\bb{\alpha} = (2,2)$ and $\beta = 3$}
        \end{subfigure}
        \quad
        \begin{subfigure}[b]{0.22\textwidth}
            \centering
            \includegraphics[width=\textwidth, height=0.85\textwidth]{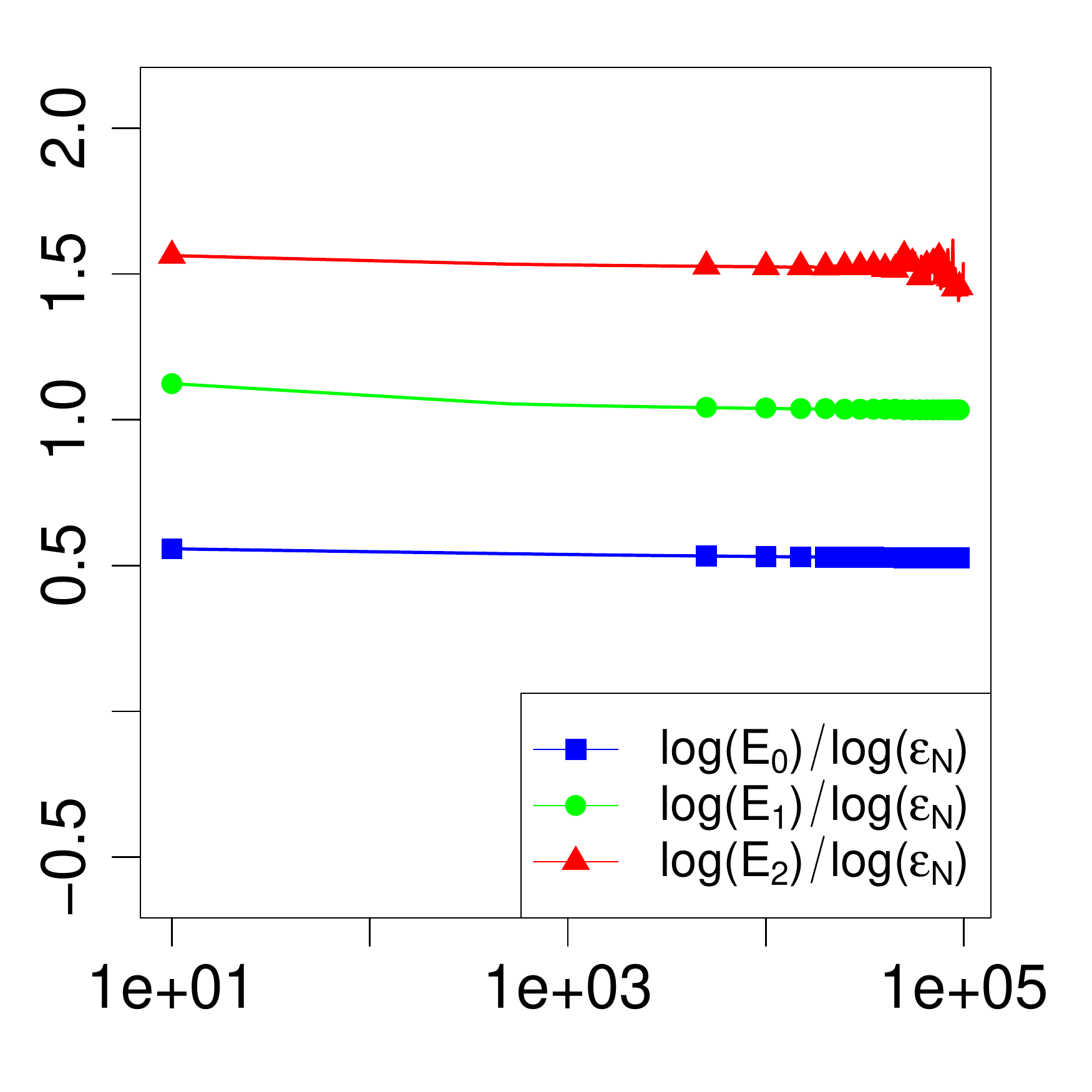}
            \vspace{-0.9cm}
            \caption{$\bb{\alpha} = (2,3)$ and $\beta = 3$}
        \end{subfigure}
        \quad
        \begin{subfigure}[b]{0.22\textwidth}
            \centering
            \includegraphics[width=\textwidth, height=0.85\textwidth]{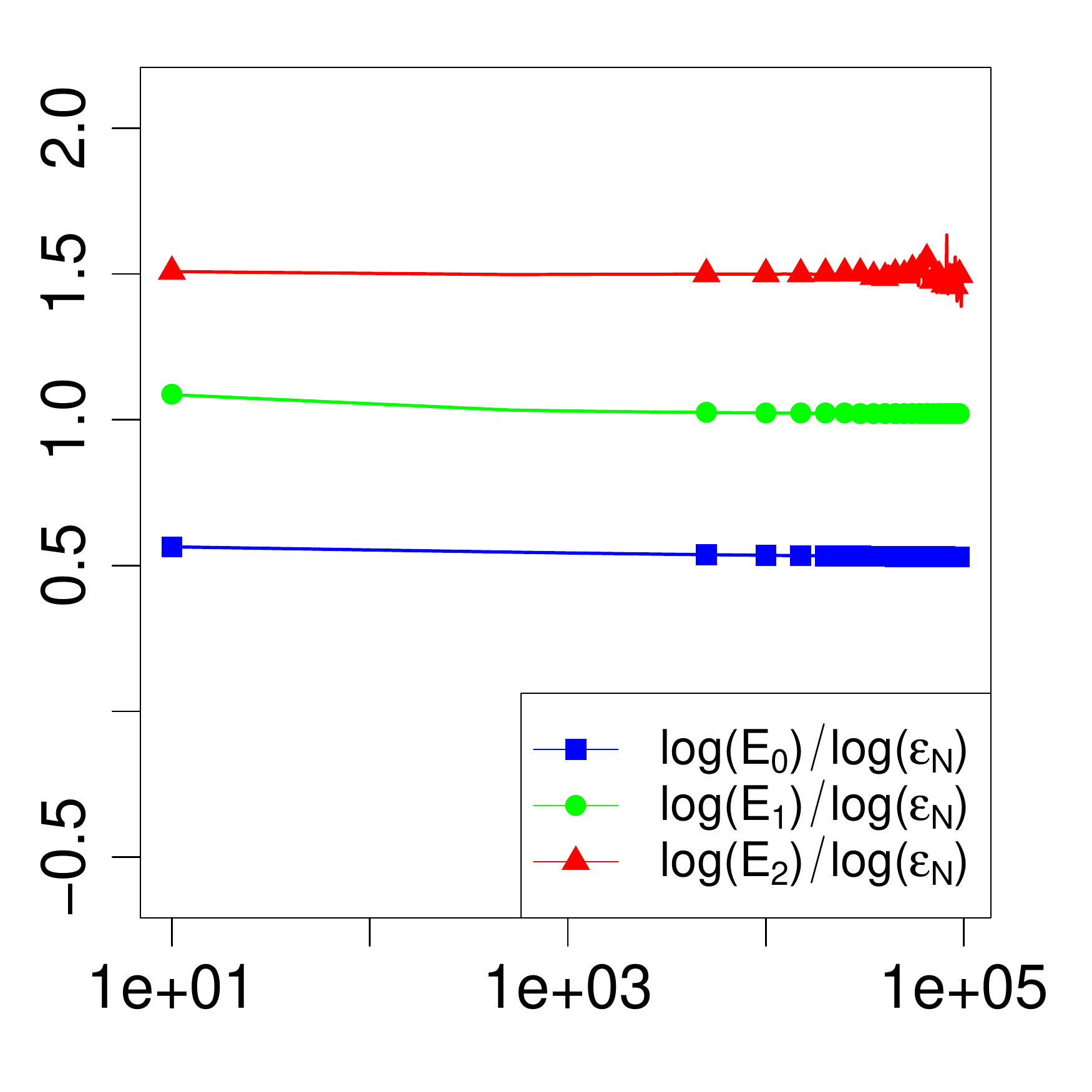}
            \vspace{-0.9cm}
            \caption{$\bb{\alpha} = (2,4)$ and $\beta = 3$}
        \end{subfigure}
        \begin{subfigure}[b]{0.22\textwidth}
            \centering
            \includegraphics[width=\textwidth, height=0.85\textwidth]{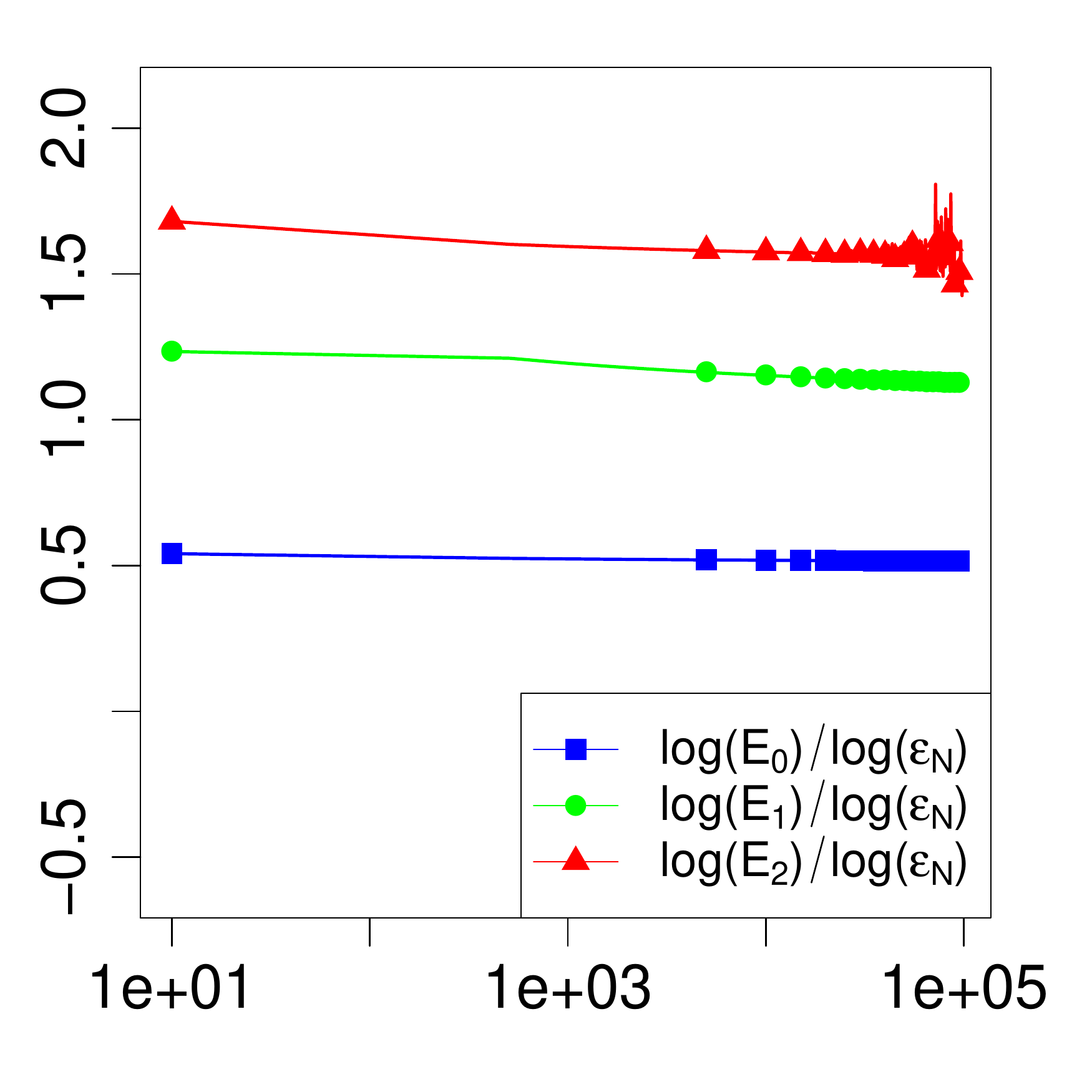}
            \vspace{-0.9cm}
            \caption{$\bb{\alpha} = (3,1)$ and $\beta = 3$}
        \end{subfigure}
        \quad
        \begin{subfigure}[b]{0.22\textwidth}
            \centering
            \includegraphics[width=\textwidth, height=0.85\textwidth]{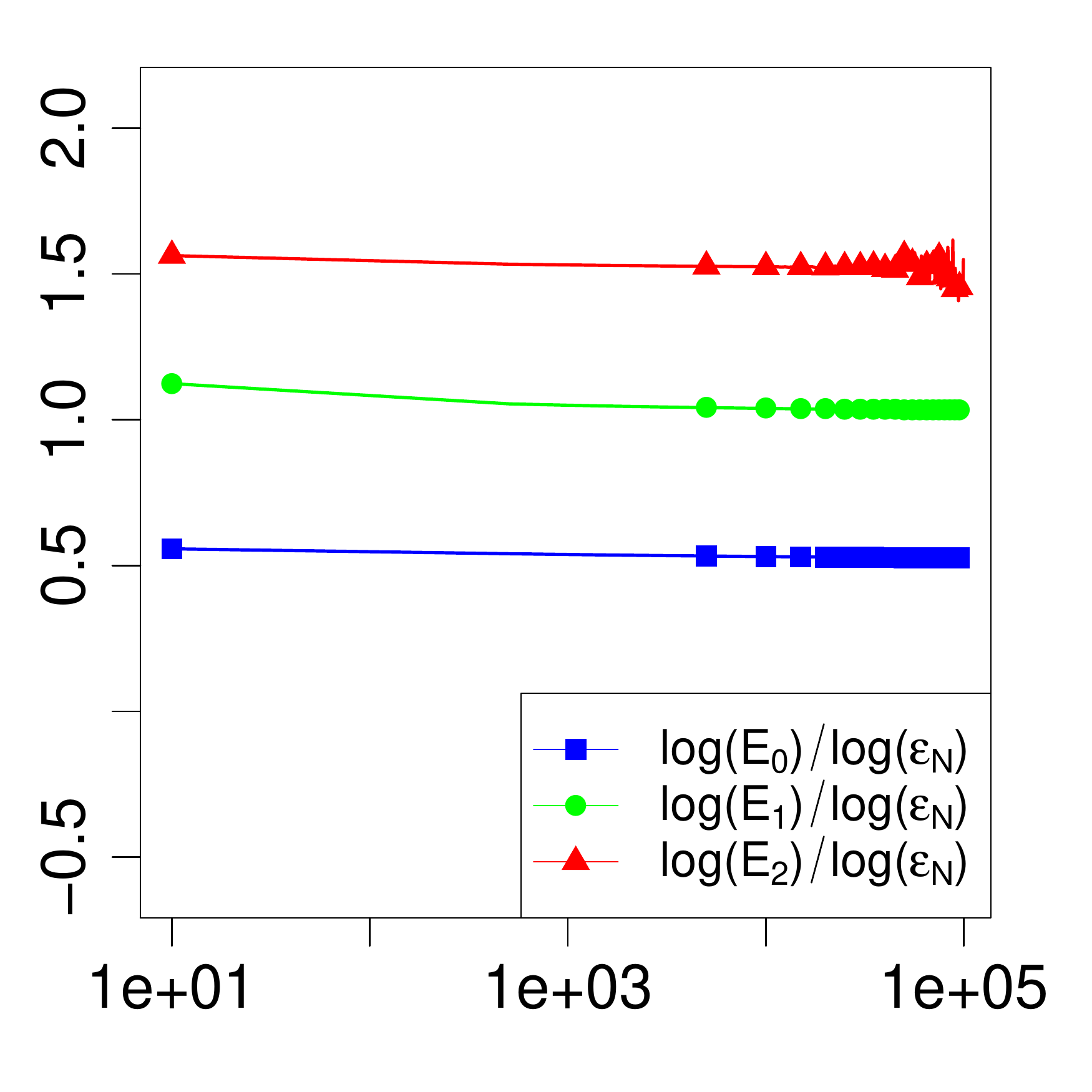}
            \vspace{-0.9cm}
            \caption{$\bb{\alpha} = (3,2)$ and $\beta = 3$}
        \end{subfigure}
        \quad
        \begin{subfigure}[b]{0.22\textwidth}
            \centering
            \includegraphics[width=\textwidth, height=0.85\textwidth]{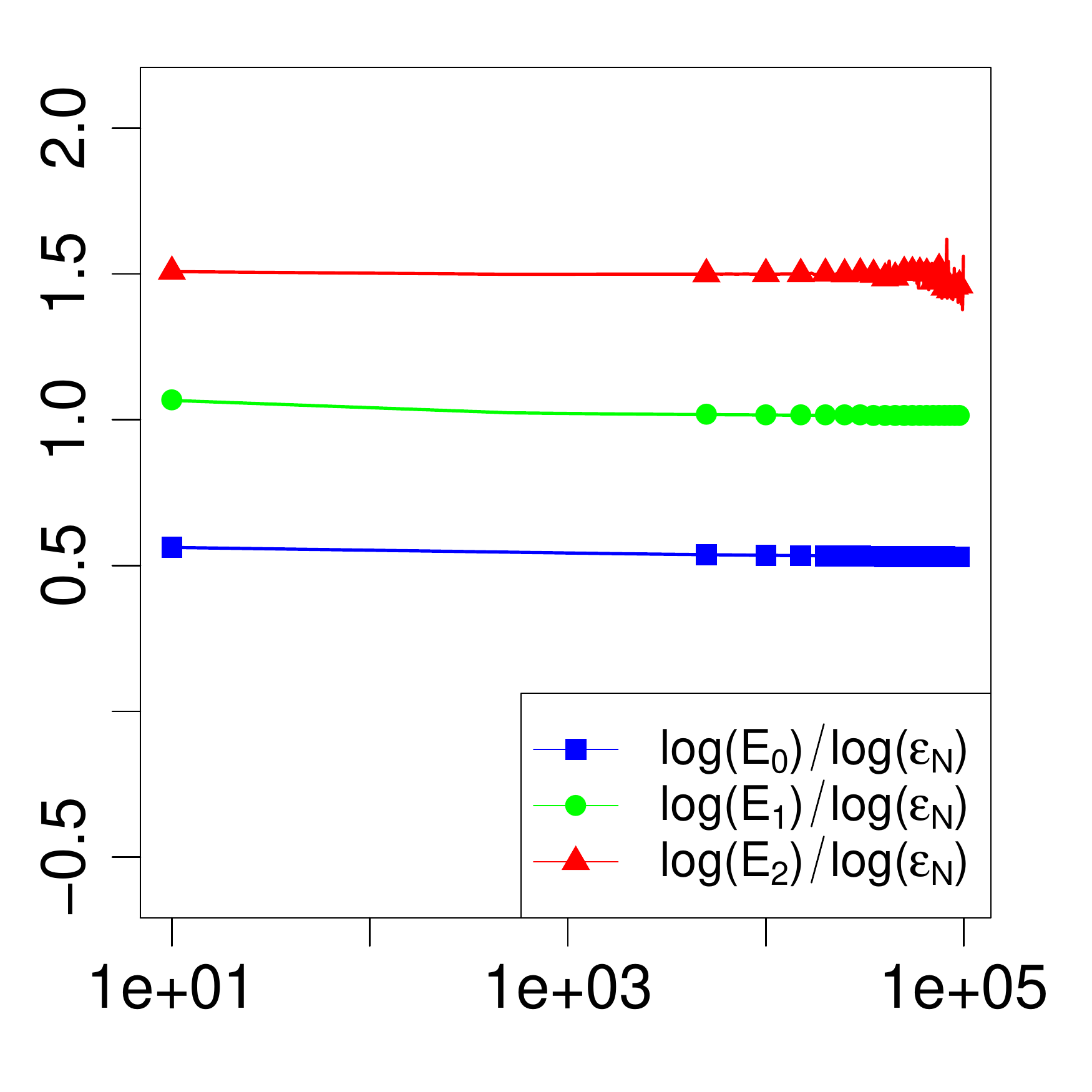}
            \vspace{-0.9cm}
            \caption{$\bb{\alpha} = (3,3)$ and $\beta = 3$}
        \end{subfigure}
        \quad
        \begin{subfigure}[b]{0.22\textwidth}
            \centering
            \includegraphics[width=\textwidth, height=0.85\textwidth]{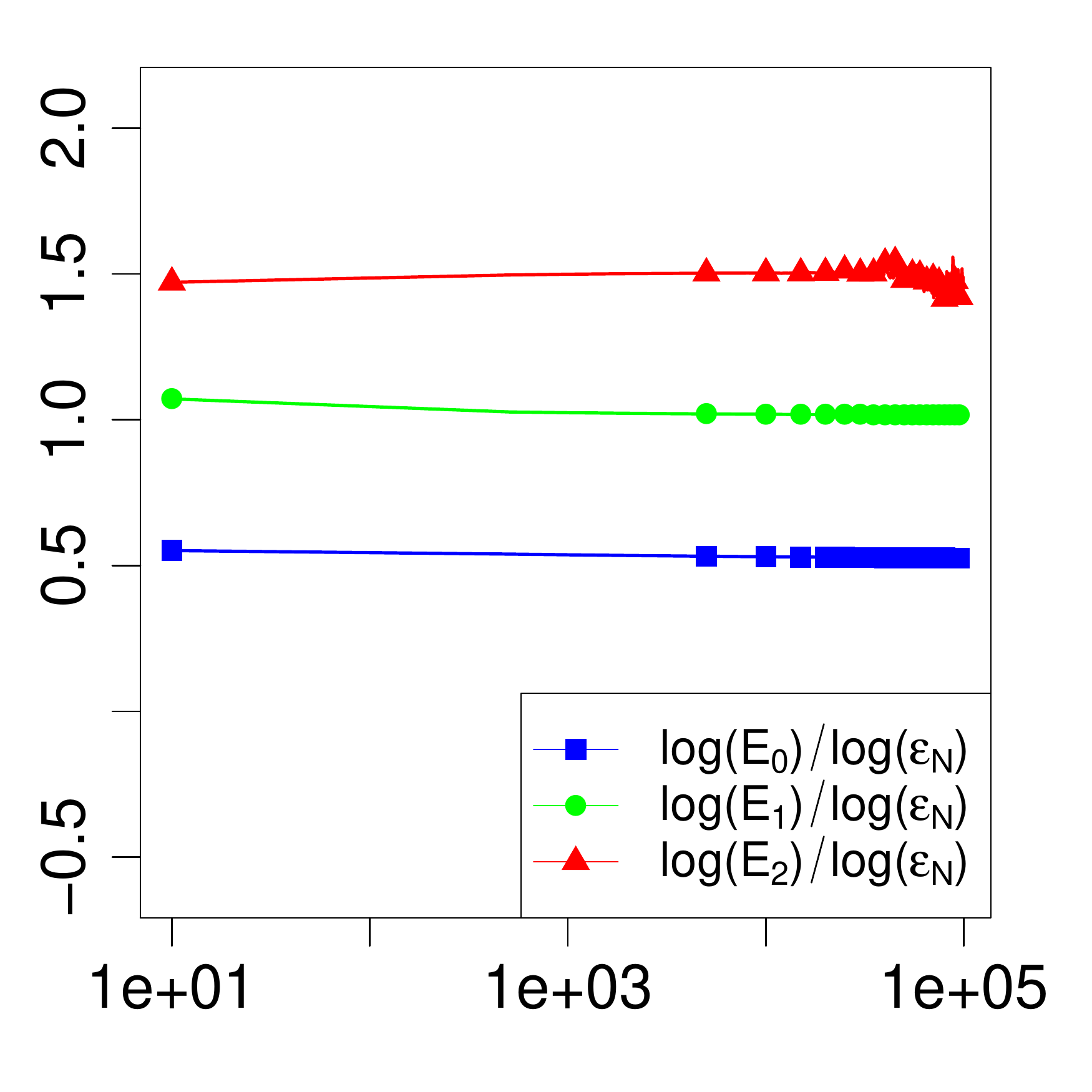}
            \vspace{-0.9cm}
            \caption{$\bb{\alpha} = (3,4)$ and $\beta = 3$}
        \end{subfigure}
        \caption{Plots of $\log E_i / \log \e_N$ as a function of $N$, for various choices of $\bb{\alpha}$, when $\beta = 3$. The horizontal axis is on a logarithmic scale. The plots confirm \eqref{eq:liminf.exponent.bound} and bring strong evidence for the validity of Theorem~\ref{thm:p.k.expansion}.}
        \label{fig:error.exponents.plots.beta.3}
    \end{figure}

\clearpage
\section*{Acknowledgments}

The author acknowledges support of a postdoctoral fellowship from the NSERC (PDF) and the FRQNT (B3X supplement).
We thank the referees for their valuable comments that led to improvements in the presentation of this paper.

\vspace{-5mm}
\section*{Data Availability Statement}

The \texttt{R} code that generated all the figures in Appendix~\ref{sec:simulations} is available as supplemental material online at \url{https://doi.org/10.1002/sta4.410}.

%
% ----------  B I B L I O G R A P H Y  ----------
%

\vspace{-5mm}
\bibliography{Ouimet_2021_LLT_Dirichlet_bib}

\end{document}